\documentclass[12pt,english]{article}
\usepackage[a4paper,bindingoffset=0.2in,%
            left=1in,right=1in,top=1in,bottom=1in,%
            footskip=.25in]{geometry}
\usepackage{xspace}
\usepackage{tablists}
\title{An important paper}
\usepackage{lmodern}

\usepackage{fancybox}
\cornersize {2}
\usepackage{geometry}
\usepackage[T1]{fontenc}
\usepackage{cite}
\usepackage{fancyhdr}
\usepackage{textcomp}
\usepackage[dvips]{epsfig}
\usepackage{amscd}
\usepackage{amssymb}
\usepackage{amsthm}
\usepackage{latexsym}
\usepackage{mathrsfs}
\usepackage{upgreek}
\usepackage{amsmath}
\usepackage{amsfonts}
\usepackage{amssymb}
\usepackage{graphicx}
\usepackage{caption}
\usepackage{subcaption}
\usepackage{color}
\usepackage{minitoc}
\usepackage{calligra}
\usepackage[toc,page,title]{appendix}
\usepackage[utf8]{inputenc}
\usepackage[]{lastpage}
\usepackage{tikz}
\usetikzlibrary{arrows,calc,shapes,decorations.pathreplacing}
\usetikzlibrary{decorations.pathmorphing}
\usetikzlibrary{shapes,arrows}
\usetikzlibrary{patterns}

\usetikzlibrary{arrows,shapes,positioning}
\usetikzlibrary{decorations.markings}
\tikzstyle arrowstyle=[scale=2]
\tikzstyle directed=[postaction={decorate,decoration={markings,
    mark=at position .65 with {\arrow[arrowstyle]{stealth}}}}]
\tikzstyle reverse directed=[postaction={decorate,decoration={markings,
    mark=at position .65 with {\arrowreversed[arrowstyle]{stealth};}}}]
\usepackage{enumerate}
\usepackage{tikz}
\theoremstyle{plain}
\newtheorem{thm}{Theorem}[section]
\theoremstyle{plain}
\newtheorem{lem}[thm]{Lemma}
\newtheorem{prop}[thm]{Proposition}
\newtheorem{cor}[thm]{Corollary}

\theoremstyle{definition}
\newtheorem{defi}[thm]{Definition}

\newtheorem{RQ}[thm]{Remark}
\usepackage{float}
\usepackage{textcomp}
\RequirePackage[bookmarks,%
                colorlinks,
                urlcolor=blue,%
                citecolor=blue,%
                linkcolor=black,%
                hyperfigures,%
                pdfcreator=LaTeX,%
                breaklinks=true,%
         pdfpagelayout=SinglePage,%
                bookmarksopen=true,%
                bookmarksopenlevel=2]{hyperref}

\usepackage[utf8]{inputenc}
\usepackage[english]{babel}
\usepackage[affil-it]{authblk}
\DeclareUnicodeCharacter{0301}{\'{e}}

\begin{document}
\title{\textbf{Hypoelliptic and spectral estimates for the  linearized Landau operator}}
\author{\scshape{Mohamad} Rachid\footnote{Université de Nantes,
Laboratoire de Mathématiques Jean Leray,
 2 rue de la Houssinière\\
BP 92208 F-44322 Nantes Cedex 3, France.\\
Email: \href{mailto: Mohamad.Rachid@univ-nantes.fr}{Mohamad.Rachid@univ-nantes.fr}
 }}


\date{}
\maketitle
\begin{abstract}
We are interested in the inhomogeneous Landau equation which describes the evolution of a particle density $f=f(t,x,v)$ representing at time  $t\geq 0$, the density of particles at position $x\in \mathbb{R}^{3}$ and  velocity $v\in \mathbb{R}^{3}$. The study  is motivated by the linearization of the Landau equation near  Maxwellian distribution. In this article, we establish hypoelliptic estimates, a localization of the spectrum and estimates of the resolvent of the  the linearized Landau operator with hard potentials and Maxwellian molecules.
The proof is based on a multiplier method and requires  refined   pseudo-differential calculus tools.
\end{abstract}
\tableofcontents

\newpage
\section{Introduction}
\subsection{{The model.}}In this paper, we study hypoellipticity and  spectral properties associated to the spatially inhomogeneous Landau equation. This equation is a kinetic model in plasma physics that describes the evolution of the density function  $F=F(t, x, v)$ representing at time $t\in \mathbb{R^{+}}$, the density of particles at position  $x\in \mathbb{R}^{3}$ and  velocity $v\in \mathbb{R}^{3}$. This equation is given by 
\begin{align}\label{N1}
 \left\{
    \begin{array}{ll}
    \partial_{t}F+v\cdot\nabla_{x}F=Q(F,F)\\
    {F}_{|t=0}=F_{0},
     \end{array}
\right.
    \end{align}
    where $Q$ is the so-called Landau collision operator which acts on the variable $v$  and which contains   diffusion in velocity. More precisely, the Landau operator is defined by
    \begin{align}\label{W0}
        Q(G,F)= \partial_{i}\int_{\mathbb{R}^{3}}a_{ij}(v-v_{*})[G_{*}\partial_{j}F-F\partial_{j}G_{*}]\hspace{0.1cm}\mathrm{d}v_{*},
    \end{align}
   where we use the convention of summation of repeated indices, and the derivatives are in the velocity variable $v$ i.e. $\partial_{i}=\partial_{v_{i}}$. Hereafter we use the shorthand notations $G_{*}=G(v_{*})$, $F=F(v)$, $\partial_{j}G_{*}=\partial_{{v_{*}}_{j}}G(v_{*})$, $\partial_{j}F=\partial_{v_{j}}F(v)$, etc. The matrix $A(v)=(a_{ij}(v))_{1\leq  i,j \leq 3}$ is symmetric, positive,  definite, depends on the interaction between particles and is given by
    $$ a_{ij}(v)=\vert v \vert ^{\gamma+2}\left(\delta_{ij}-\frac{v_{i}v_{j}}{\vert v \vert^{2}}\right),\hspace{0.2cm} \gamma \in [{-3,1}].$$
    We recall the standard classification: we call hard potentials if $\gamma \in ({0,1}]$, Maxwellian molecules if $\gamma=0$, moderately soft potentials if $\gamma \in [{-2,0})$, very soft potentials if $\gamma \in ({-3,-2})$ and Coulombian potential if $\gamma=-3$. Hereafter we shall consider the cases of hard potentials, Maxwellian molecules, i.e. $\gamma \in [{0,1}]$. We denote by 
    $$ \mu(v)=(2\pi)^{-3/2}e^{-\vert v \vert^{2}/2}$$
    the normalized Maxwellian which is a global equilibrium. We linearize the Landau equation around $\mu$ with the perturbation
    $$ F=\mu+\mu^{1/2}f.$$
   The Landau equation (\fcolorbox{red}{ white}{\ref{N1}}) for $f=f(t, x, v)$ takes the form
    \begin{align}\label{N2}
 \left\{
    \begin{array}{ll}
    \partial_{t}f+v\cdot\nabla_{x}f -\mu^{-1/2}Q(\mu^{1/2}f,\mu)-\mu^{-1/2}Q(\mu,\mu^{1/2}f)=\mu^{-1/2}Q(\mu^{1/2}f,\mu^{1/2}f)\\
    {f}_{|t=0}=f_{0}=\mu^{-1/2}(F_{0}-\mu),
     \end{array}
\right.
    \end{align}
    since $Q(\mu,\mu)=0$. Using the notation
    $$\Gamma(f,g)=\mu^{-1/2}Q(\mu^{1/2}f,\mu^{1/2}g),$$
    we may rewrite the above equation as
      \begin{align}\label{N3}
 \left\{
    \begin{array}{ll}
    \partial_{t}f +\mathcal{P}f=\Gamma(f,f)\\
    {f}_{|t=0}=f_{0},
     \end{array}
\right.
    \end{align}
    where the linearized Landau operator
   $\mathcal{P}$ takes the form
    \begin{align}\label{N4}
        \mathcal{P}= v\cdot\nabla_{x} -\mathcal{L}    
        \end{align}
    with
    $$\mathcal{L}=\mathcal{L}_{1}+\mathcal{L}_{2},\hspace{0.3cm}\mathcal{L}_{1}=\Gamma(\sqrt{\mu},f),\hspace{0.3cm}\mathcal{L}_{2}=\Gamma(f,\sqrt{\mu}).$$
    Operator $\mathcal{P}$ acts only in variables $(x,v)$, is non selfadjoint, and consists of a transport
part which is skew-adjoint, a diffusion part acting only in the $v$ variable and a compact part (see for example Proposition 2.1 in \cite{degond1997dispersion}). Using for example \cite{Wu2013ExponentialTD}, \cite{strain2008exponential},   the diffusion part $\mathcal{L}_{1}$ is written as follows
    $$\mathcal{L}_{1}f=\nabla_{v}\cdot [\mathbf{A}(v)\nabla_{v}f] -\Big(\mathbf{A}(v)\frac{v}{2}\cdot \frac{v}{2}\Big)f + \nabla_{v}\cdot\Big[ \mathbf{A}(v)\frac{v}{2}\Big]f,$$
   where $\mathbf{A}(v)=(\overline{a}_{ij}(v))_{1\leq  i,j \leq 3}$ is a symmetric matrix with
    $$ \overline{a}_{ij}={a}_{ij}*_{v}\mu,$$
    and the compact part $\mathcal{L}_{2}$ is given by
    $$ \mathcal{L}_{2}f=-\mu^{-1/2}\partial_{i}\left\{\mu \left[a_{ij}*_{v}\left\{\mu^{1/2}\Big[\partial_{j}f+\frac{v_{j}}{2}f\Big]\right\}\right]\right\}.$$
    \begin{RQ}
    Here we do not follow the same convention as the one in \cite{Guo2002} for operators $\mathcal{L}_{1}$ and $\mathcal{L}_{2}$.
    \end{RQ}
    \subsection{{Notations.}}
    Throughout the paper we shall adopt the following notations: we work in dimension
    $n=3$ and denote by $(x,v)\in \mathbb{R}_{x}^{3}\times\mathbb{R}_{v}^{3}$ the space-velocity variables. For $v\in \mathbb{R}^{3}$ we denote ${\langle v \rangle}=(1+\vert v \vert^{2})^{1/2}$,  where we recall that $\vert v \vert$ is the canonical Euclidian norm of $v$ in $\mathbb{R}^{3}$. The gradient in velocity (resp. space) will be denoted by $\partial_{v}$ (resp. $\partial_{x}$). We
shall also denote $D_{v}=\frac{1}{i}\partial_{v}$ (resp. $D_{x}=\frac{1}{i}\partial_{x}$), and denote $\xi$ the dual variable of $x$, $\eta$ the dual variable of $v$. For simplicity of notations, $a\sim b$ means that there exist constants  $c_{1}$, $c_{2}> 0$ such that $c_{1}b\leq a\leq c_{2}b$; we abbreviate ``$\leq C\hspace{0.1cm}$'' to ``$\lesssim$'', where $C$ is a positive constant depending only on fixed number. Finally, the space of distributions on $\Omega$ is denoted by $\mathcal{D'}(\Omega)$ where $\Omega\subseteq \mathbb{R}^{n}$ is an open set.
     \subsection{{Main results and comments.}}
In this article, we will show a localization property of the  pseudospectrum of the Landau operator $\mathcal{P}$, that is to say the region of the complex plane where its resolvent is a priori large (in fact we more precisely give a description of a large region where the resolvent \it is \rm controlled, which is included in the complementary of the pseudospectrum).  This result is given by the following theorem.
 \begin{thm}\label{N5}
 Let $\mathcal{P}$  be the Landau operator on      $L^{2}(\mathbb{R}_{x}^{3}\times\mathbb{R}_{v}^{3})$ defined in (\fcolorbox{red}{ white}{\ref{N4}}) with $\gamma \in [{0,1}]$. Then there are two constants $C_{\mathcal{P}}>0$ and $Q_{\mathcal{P}}>0$  so that:
 \begin{itemize}
    \item[a)] The spectrum of $\mathcal{P}$  verifies  $$\sigma(\mathcal{P})\subset S_{\mathcal{P}}\cap \lbrace{\Re{e}\hspace{0.05cm} z \geq 0\rbrace},$$
with
\begin{align}\label{N112}
 S_{\mathcal{P}}=\left\{z\in \mathbb{C},\hspace{0.1cm}\vert z+1 \vert^{1/3}\leq C_{\mathcal{P}} \big( \Re{e}\hspace{0.05cm}z +1\big),\hspace{0.1cm}\Re{e}\hspace{0.05cm} z\geq  -\frac{1}{2}  \right\}.
 \end{align}
    \item[b)] For any $z\not\in S_{\mathcal{P}}$ with $\Re{e}\hspace{0.05cm}z\geq -\frac{1}{2}$, the resolvent is estimated by
    \begin{align}\label{N113}
{\Vert (z-\mathcal{P})^{-1} \Vert}_{\mathcal{B}(L^{2}_{x,v})} \leq Q_{\mathcal{P}}\vert z+1 \vert^{-1/3}.
    \end{align}
    Notice that if $\Re{e}\hspace{0.05cm}z\leq -\frac{1}{2}$ then
    \begin{align}\label{W3}
    {\Vert (z-\mathcal{P})^{-1} \Vert}_{\mathcal{B}(L^{2}_{x,v})} \leq \vert \Re{e}\hspace{0.05cm}z \vert^{-1}.
      \end{align}
\end{itemize}
\end{thm}
The results of Theorem \fcolorbox{red}{ white}{\ref{N5}} can be illustrated by the following figure:
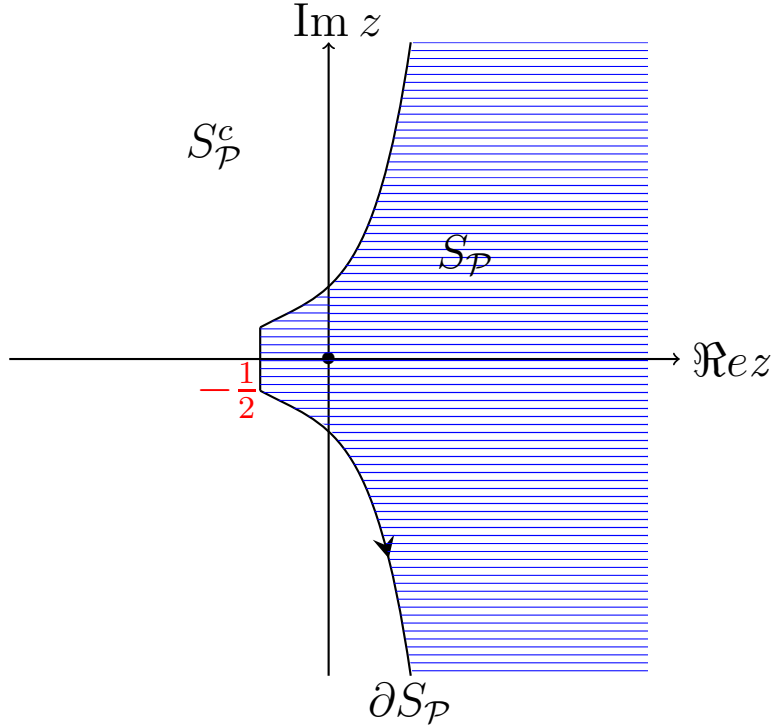
\begin{figure}[H]
\begin{center}
\begin{tikzpicture}[scale=0.6]
\draw[thick,->] (-7,0) -- (7.7,0);
\draw (7.6,0) node[scale=1.5,right] {$\Re e\textit{z}$};
\draw [thick,->] (0,-7) -- (0,7);
\draw (0.2,7.5) node[scale=1.5] {$\operatorname{Im}\hspace{0.01cm}z $};
\draw [black](0,0) node {$\bullet$};
\draw[red] (-2.2,-0.7) node[scale=1.5] {$-\frac{1}{2}$};
   \draw[thick, black] (-1.5, 0.7) .. controls (0, 1.5) and (1, 1.5) .. (1.8, 7);
    \draw[thick, black , directed] (-1.5, -0.7) .. controls (0, -1.5) and (1, -1.5) .. (1.8, -7);
    
    \draw[thick, black] [thick] (-1.5, 0.7) -- (-1.5,-0.7);  
    
    \fill [thick, pattern= horizontal lines, pattern color=blue] (-1.5, 0.7) --  (-1.5, -0.7) -- (0,-1.5) -- (0,+1.5)  -- cycle;
    
 \fill [thick, pattern= horizontal lines,  pattern color=blue]  (0, 1.5) -- (0, -1.5) -- (0.7, -2.55) -- ( 7,-2.55) -- (7,2.55) --(0.7,2.55)-- (0, 1.5)-- cycle;

  \fill [thick, pattern= horizontal lines,  pattern color=blue]  (0.7, -2.55) -- (1.24, -4) -- (7, -4) -- ( 7,-2.55)-- cycle;
 
 \fill [thick, pattern= horizontal lines,  pattern color=blue]  (0.7, 2.55) -- (1.24, 4) -- (7, 4) -- ( 7,2.55)-- cycle;
 
 \fill [thick, pattern= horizontal lines,  pattern color=blue]  (1.24, -4) -- (1.85, -7) -- (7, -7) -- ( 7,-4)-- cycle;
 
 \fill [thick, pattern= horizontal lines,  pattern color=blue]  (1.24, 4) -- (1.85, 7) -- (7, 7) -- ( 7,4)-- cycle;
 
 
\draw(-2.5,4.7) node[scale=1.5]
{${S_\mathcal{P}^{c}}$};
\draw (3,2.3) node[scale=1.5] {$\displaystyle{S_\mathcal{P}}$};
\draw (1.8,-7.6) node[scale=1.5] {$\partial S_\mathcal{P}$};
\end{tikzpicture}
\caption{Localization of the pseudospectrum of the Landau operator $\mathcal{P}.$ }\label{M176}
\end{center}
\end{figure}

In this figure  $\partial S_{\mathcal{P}}$ oriented from $+i\infty$ to $-i\infty$. The hatched part is where the spectrum is localized and the non hatched part is the zone where we have good resolvent estimates (see (\fcolorbox{red}{ white}{\ref{N113}}) and (\fcolorbox{red}{ white}{\ref{W3}})).\\
Let us give now some comments and motivations of this result. The cuspidal form of the pseudospectrum of linear or linearized kinetic operators was first shown in \cite{article}, \cite {herau_isotropic_2004} for the Fokker-Planck operator, and then extended to operators appearing in statistical mechanics by \cite{Eckmann2000NonEquilibriumSM}. The main motivation for this type of study is to be able to understand the so-called pseudospectrum properties and derive possible
trend to the equilibrium or regularization properties for the related evolution equation thanks to Cauchy formulae (see e.g. \cite {herau_isotropic_2004} for a complete result in the Fokker-Planck case). Anyway in this article, as a first step of a complete study of properties of this type,  we only focus on the pseudospectral localization, but in the case of the much more complicated linearized Landau operator.

The aim of this article is double. First we show that the linearized Landau operator has indeed a pseudospectrum of cuspidal form, which is a good clue for thinking that this is a very general property of general linear or linearized kinetic operators. Second we propose a very robust and self contained pseudo-differential framework (see Section \fcolorbox{red}{ white}{\ref{W1}}) in order to show this type of result. Our hope is that it can be used for many other kinetic models. These tools are greatly inspired by previous works (see \cite{alexandre_global_2012}, \cite{herau_anisotropic_2011}, \cite{herau:hal-01596009}) concerning kinetic equations, following fundamental ideas in \cite{lerner_metrics_2011} on Wick and Weyl pseudodifferential calculus. Indeed the main remark done in all these works, following preliminary works by Alexandre and Villani \cite{Alexandre02onthe} is that the linearized Landau operator $L$
is, up to controlled/bounded operators,  a pseudodifferential operator, namely that there exists a symbol, in a Hörmander-type class of symbol (see Section \fcolorbox{red}{ white}{\ref{W1}}) such that
$$
L \equiv a^w + \textrm{ controlled terms},
$$
where $w$ stands for the Weyl quantization. The equivalence above can be rigourosly stated and we will do it in section 2 below. The understanding of Landau type operators has been greatly improved in the past twenty years, following in particular works by Guo (see \cite{Guo2002}).

Using these tools, the proof of the result given in Theorem \fcolorbox{red}{ white}{\ref{N5}} will follow the lines
of \cite{herau_anisotropic_2011} (see also \cite{alexandre_global_2012},\cite {herau_isotropic_2004}), where a multiplier 
method is used for proving regularization (hypoelliptic) properties of the linearized landau operator. Let us emphasize here  that no regularization property is shown in this article, since we only focus on spectral estimates. Mention anyway -  and this is a remarkable feature of all these hypoelliptic/hypocoercive techniques -  that the same (in spirit) methods give very strong and precise results as the one given in the main theorem.

\textbf{Organization of the article.} In Section \fcolorbox{red}{ white}{\ref{z1}} we give some properties of the Landau operator. In Section  \fcolorbox{red}{ white}{\ref{z2}}, we prove hypoelliptic estimates with respect to the velocity variable for a parametric operator. In Section \fcolorbox{red}{ white}{\ref{z3}}
we give hypoelliptic estimates for the linearized Landau operator. Section \fcolorbox{red}{ white}{\ref{z4}} is devoted to the proof of Theorem \fcolorbox{red}{ white}{\ref{N5}}. An appendix is devoted to a short review of some
tools used in this work (Weyl-Hörmander quantization, Wick quantization and the proof of the Theorem \fcolorbox{red}{ white}{\ref{M36}} (Basic Theorem)).

\section{Properties of the Landau operator}\label{z1}

In this section, we first  present the decomposition of the linearized Landau operator $ \mathcal {P} $ then we exhibit a simpler form of this operator. To end up we show that $ \mathcal {P} $ is a generator of a strongly continuous semigroup. Throughout this section, we work with $\gamma \in [{-3,1}] $.

\subsection{Splitting of the linearized operator.} Consider a smooth positive function $\chi\in C_{c}^{\infty}(\mathbb{R}_{v}^{3})$ such that $0\leq \chi(v)\leq 1$, $\chi(v)=1$ for $\vert v \vert\leq 1$ and $\chi(v)=0$ for $\vert v \vert > 2$. For any $R\geq 1$ we define $\chi_{R}(v)=\chi(\frac{v}{R})$ and in the sequel we shall consider the function $M\chi_{R}$, for some constant $M>0$. Then, we introduce the decomposition of the operator $\mathcal{P}$ as $\mathcal{P}=\mathcal{A}+\mathcal{K}$ with
\begin{align}\label{N14}
   \mathcal{A}= - \mathcal{L}_{1} + v\cdot\nabla_{x} + M\chi_{R},\hspace{0.3cm} \mathcal{K}=- \mathcal{L}_{2}-M\chi_{R},
\end{align}
where $M>0$ and $R>0$  will be chosen later.\\
We define the function $F(v)$ as
\begin{align}\label{N0}
F(v)=\Big(\mathbf{A}(v)\frac{v}{2}\cdot \frac{v}{2}\Big) - \nabla_{v}\cdot\Big[ \mathbf{A}(v)\frac{v}{2}\Big]+ M\chi_{R}.
\end{align}
Then, we can rewrite $\mathcal{A}$ as follows
\begin{align}
   \mathcal{A}f=  v\cdot\nabla_{x}f- \nabla_{v}\cdot [\mathbf{A}(v)\nabla_{v}f] +F(v)f.
\end{align}

\subsection{Preliminaries.}  We have the following results concerning the matrix $\mathbf{A}(v)$.
\begin{lem}\label{N6}
The following properties hold:
\begin{itemize}
    \item[a)] For  $ v \in  \mathbb{R}^{3}\backslash\lbrace{0}\rbrace$, the matrix $\mathbf{A}(v)$ has a simple eigenvalue $\ell_{1}(v)>0$ associated with the eigenvector $v$ and a double eigenvalue $\ell_{2}(v)>0$ associated with the eigenspace $v^{\perp}$. Moreover, when $\vert v \vert\rightarrow +\infty$ we have
    $$ \ell_{1}(v)\sim 2{\langle v \rangle}^{\gamma  }\hspace{0.2cm}\text{and}\hspace{0.2cm}  \ell_{2}(v)\sim {\langle v \rangle}^{\gamma+2}. $$
     \item[b)]The function $ \overline{a}_{ij}$ is smooth, for any multi-index
      $ \alpha \in \mathbb{N}^{3}$, there exists $C_{\alpha}>0$ such that for all  $v\in \mathbb{R}^{3}$, we have
$$\vert \partial_{v}^{\alpha} \overline{a}_{ij}(v)\vert +\vert \partial_{v}^{\alpha}(\overline{a}_{ij}(v)v_{j})\vert  \leq C_{\alpha}\hspace{0.05cm}{\langle v \rangle}^{\gamma+2-\vert \alpha \vert },$$
 \item[c)] For  $ v \in  \mathbb{R}^{3}\backslash\lbrace{0}\rbrace$, we have
$$  \overline{a}_{ij}(v)v_{i}v_{j}=\ell_{1}(v)\vert v \vert^{2},$$
  $$\overline{a}_{ii}(v)=\text{tr}(\overline{a}(v))=\ell_{1}(v)+2\ell_{2}(v),$$
 $$\overline{a}_{ij}(v)\eta_{i}\eta_{j}=\ell_{1}(v)\vert P_{v}\eta  \vert^{2}+\ell_{2}(v)\vert (I-P_{v})\eta  \vert^{2},$$
 with $\eta\in \mathbb{R}^{3}$ and $P_{v}$ is the projection on  $v$, i.e. $P_{v}\eta=\big(\eta\cdot\frac{v}{\vert v \vert}\big)\frac{v}{\vert v \vert}.$
   \item[d)] For $\vert v \vert>1$, we have   $$\vert \partial_{v}^{\alpha}\ell_{1}(v)\vert\leq C_{\alpha}\hspace{0.05cm}{\langle v \rangle}^{\gamma-\vert \alpha \vert } \hspace{0.3cm}\text{and}\hspace{0.3cm} \vert\partial_{v}^{\alpha}\ell_{2}(v)\vert \leq C_{\alpha}\hspace{0.05cm}{\langle v \rangle}^{\gamma+2-\vert \alpha \vert }.$$
    \end{itemize}
\end{lem}
\begin{proof}
See for example \cite[Lemma 2.4]{carrapatoso:hal-01143343}, \cite[Lemma 3]{Guo2002} and \cite[Proposition 1]{Wu2013ExponentialTD}.
\end{proof}
\begin{lem}\label{N7}
 For all $v\in \mathbb{R}^{3}$ with $\vert v \vert>1$ , we have
   $$\ell_{1}(v)\gtrsim{\langle v \rangle}^{\gamma} \hspace{0.3cm}\text{and}\hspace{0.3cm} \ell_{2}(v)\gtrsim{\langle v \rangle}^{\gamma+2}.$$ 
\end{lem}
\begin{proof}
Using (a) in Lemma \fcolorbox{red}{ white}{\ref{N6}},  when $\vert v \vert\rightarrow +\infty$, we have
     $\ell_{1}(v)\sim 2{\langle v \rangle}^{\gamma}$. In particular, there is a constant  $N>0$ such that for all $\vert v \vert> N$, we have    $$\ell_{1}(v)\geq {\langle v \rangle}^{\gamma}.$$
  We have that  $ \ell_{1}(v)$ is continuous since $\mathbf{A}(v)$ is a positive definite symmetric matrix and continuous (due to convolution with  $\mu$), hence the existence of a constant  $C>0$ such that for $1\leq\vert v \vert\leq N$ \begin{align*}
\ell_{1}(v)&\geq C \hspace{0.05cm}{\langle v \rangle}^{\gamma },
\end{align*}
and then for all $v\in \mathbb{R}^{3}$, perhaps with changing $C$,
$$ \ell_{1}(v)\geq C\hspace{0.05cm}{\langle v \rangle}^{\gamma}.$$
The proof will be the same for $\ell_{2}(v).$
\end{proof}
\begin{lem}
    Let $F(v)$ be defined in (\fcolorbox{red}{ white}{\ref{N0}}). Then, we can choose  $M$ and $R$ big enough such that for all $v\in \mathbb{R}^{3}$, we have
   $$F(v)\gtrsim{\langle v \rangle}^{\gamma+2}.$$
\end{lem}
\begin{proof}
Since
\begin{align*}
     F(v)\geq\dfrac{1}{4}\ell_{1}(v){\vert v \vert}^{2}-\Big\vert \nabla_{v}\cdot\Big[ \mathbf{A}(v)\dfrac{v}{2}\Big]\Big\vert+ M\chi_{R},
\end{align*}
according to Lemma \fcolorbox{red}{ white}{\ref{N6}}, we have
\begin{align}\label{Z0}
\Big\vert \nabla_{v}\cdot\Big[ \mathbf{A}(v)\dfrac{v}{2}\Big]\Big\vert\lesssim{\langle v \rangle}^{\gamma+1}.
    \end{align}
So, using Lemma \fcolorbox{red}{ white}{\ref{N7}} and (\fcolorbox{red}{ white}{\ref{Z0}}), there exist two positive constants
$C_{1}, C_{2}$ such that   
\begin{align*}
     F(v)\geq C_{1}{\langle v \rangle}^{\gamma +2}-C_{2}{\langle v \rangle}^{\gamma +1}+ M\chi_{R},
     \end{align*}
then  there exist  $M$ and $R$ such that for all  $v\in \mathbb{R}^{3}$ 
   $$F(v)\gtrsim\hspace{0.05cm}{\langle v \rangle}^{\gamma+2}.$$
\end{proof}
\begin{lem}\label{N80}
 For any multi-index
      $ \alpha \in \mathbb{N}^{3}$,  there   exists $C_{\alpha}>0$ such that for all $v\in \mathbb{R}^{3}$, we have  $$\vert \partial_{v}^{\alpha}F(v)\vert\leq C_{\alpha}\hspace{0.05cm}{\langle v \rangle}^{\gamma+2-\vert \alpha \vert }.$$
\end{lem}
\begin{proof}
For  $\vert v \vert>2R$, using Leibniz's formula, we have
\begin{align*}
 \partial_{v}^{\alpha}F(v)=\frac{1}{4}\sum_{\beta\leq \alpha}\binom{\alpha}{\beta}\partial^{\alpha-\beta}\ell_{1}(v)\hspace{0.05cm}\partial^{\beta}{\langle v \rangle}^{2}-\frac{1}{4}\partial^{\alpha}\ell_{1}(v)-\partial^{\alpha}\Big( \nabla_{v}\cdot\Big[ \mathbf{A}(v)\dfrac{v}{2}\Big]\Big),
 \end{align*}
Then, using  Lemma  \fcolorbox{red}{ white}{\ref{N6}} and  Lemma \fcolorbox{red}{ white}{\ref{N7}} we obtain
\begin{align*}
 \vert \partial_{v}^{\alpha}F(v)\vert &\leq  c_{\alpha}\hspace{0.05cm}{\langle v \rangle}^{\gamma+2-\vert \alpha \vert }. 
 \end{align*}
The function $F(v)$ being $C^{\infty}$ on $\vert v \vert\leq 2R$, the estimates for $\vert v \vert\leq 2R$ are immediate.
\end{proof}
\begin{lem}\label{N200}
\begin{itemize}
    \item[i)] $\mathbf{A}(v)$ is written as follows
\begin{align}\label{N8}
\mathbf{A}(v)= {B}^{\text{T}}(v){B}(v),
\end{align}
where ${B}(v)=(b_{ij}(v))_{1\leq  i,j \leq 3}$ is a matrix with real-valued smooth entries.
 \item[ii)] For any multi-index 
      $ \alpha \in \mathbb{N}^{3}$, there exists   $C_{\alpha}>0$ such that for all $v\in \mathbb{R}^{3}$, we have
\begin{equation}\label{N22}
  \vert \partial_{v}^{\alpha}b_{ij}(v)\vert \leq C_{\alpha}\hspace{0.05cm} {\langle v \rangle}^{\frac{\gamma}{2}+1-\vert \alpha \vert }. 
\end{equation}
 \item[iii)] There exists $c,C>0$ such that for all $v\in \mathbb{R}^{3}$,  for all  $ \eta \in  \mathbb{R}^{3}$ we have
 \begin{equation}\label{N501}
   c\hspace{0.05cm}{\langle v \rangle}^{\gamma}(\vert \eta \vert^{2}+ \vert v\wedge\eta\vert^{2})  \leq\mathbf{A}(v)\eta\cdot\eta =\vert B(v)\eta \vert^{2}\leq   C\hspace{0.05cm}{\langle v \rangle}^{\gamma}(\vert \eta \vert^{2}+ \vert v\wedge\eta\vert^{2})
 \end{equation}
\end{itemize}
\end{lem}
\begin{proof}
i) As $\mathbf{A}(v)$ is a positive definite symmetric matrix (denoted $\mathbf{S}_{3}^{++}(\mathbb{R})$) and according to the spectral theorem, there exists  $\mathbf{Q}$ an orthogonal matrix  such that
$$\mathbf{A}(v)=\mathbf{Q}^{\text{T}}(v)\mathbf{D}(v)\mathbf{Q}(v).$$
On $\big\{\vert v \vert>1\big\}$, $v \mapsto \mathbf{Q}(v)$ can be calculated explicitly and can be chosen to be smooth. If we set 
$$ B(v)=\mathbf{Q}^{\text{T}}(v)\sqrt{\mathbf{D}(v)}\mathbf{Q}(v)$$
with $\sqrt{\mathbf{D}(v)}=\text{diag}\big(\sqrt{\ell_{1}(v)}, \sqrt{\ell_{2}(v)}, \sqrt{\ell_{2}(v)}\big)$. We have that  $v \mapsto \sqrt{\mathbf{D}(v)} $ is of class $C^{\infty}$ for $\vert v \vert>1$ (because $\ell_{1}(v),\ell_{2}(v)$ are of class $C^{\infty}$ for  $\vert v \vert>1$). Then  the application 
$v \mapsto B(v)$ is of class $C^{\infty}$ on $\big\{\vert v \vert>1\big\}$.
  Regarding the case where $\vert v \vert\leq 1$, we consider the following two applications:
 $$\begin{array}{ccccc}
\phi & : & \mathbb{R}^{3}& \to & \mathbf{S}_{3}^{++}(\mathbb{R})  \\
 & & v & \mapsto &\mathbf{A}(v)  \\
\end{array},$$  
  $$\begin{array}{ccccc}\psi & : & \mathbf{S}_{3}^{++}(\mathbb{R})& \to & \mathbf{S}_{3}^{++}(\mathbb{R})  \\
 & & M & \mapsto &\sqrt{M}  \\
\end{array},$$ 
we note that $\phi, \psi$ are of class $C^{\infty}$, moreover $B(v)= \psi\circ \phi(v)$. Then we have that the application
$v \mapsto B(v)$ is of class $C^{\infty}.$\\
ii) For $\vert v \vert>1$, we have $B(v)=\sqrt{\ell_{1}(v)}P_{v}+\sqrt{\ell_{2}(v)}(I-P_{v})$, moreover using  Lemma \fcolorbox{red}{ white}{\ref{N6}},  we have
\begin{align}\label{Z2}
 \vert \partial_{v}^{\alpha}\sqrt{\ell_{1}(v)}\vert\lesssim{\langle v \rangle}^{\frac{\gamma}{2}-\vert \alpha \vert },\hspace{0.3cm} \vert\partial_{v}^{\alpha}\sqrt{\ell_{2}(v)}\vert \lesssim{\langle v \rangle}^{\frac{\gamma+2}{2}-\vert \alpha \vert }
 \end{align}
 and the fact that $\vert \partial_{v}^{\alpha}P_{v}\vert\lesssim 1$ ($P_{v}$
 and all its derivatives are bounded), so we get that for  $\vert v \vert>1$, 
 \begin{equation}\label{Z3}
  \vert \partial_{v}^{\alpha}b_{ij}(v)\vert \lesssim{\langle v \rangle}^{\frac{\gamma}{2}+1-\vert \alpha \vert }, 
\end{equation}
where the constants in  (\fcolorbox{red}{ white}{\ref{Z2}}), (\fcolorbox{red}{ white}{\ref{Z3}}) depend on $\alpha$.\\
The function $b_{ij}(v)$ being $C^{\infty}$ on $\vert v \vert\leq 1$,  the estimates for $\vert v \vert\leq 1$ are immediate.
Then, for all 
      $ \alpha \in \mathbb{N}^{3}$, there exists   $C_{\alpha}>0$ such that for all $v\in \mathbb{R}^{3}$, we have
\begin{equation}
  \vert \partial_{v}^{\alpha}b_{ij}(v)\vert \leq C_{\alpha}\hspace{0.05cm} {\langle v \rangle}^{\frac{\gamma}{2}+1-\vert \alpha \vert }. 
\end{equation}
iii) The estimate is immediate on
 $\vert v \vert\leq 1$ because   $\mathbf{A}(v)$ is a positive definite symmetric matrix and $\mathbf{A}= {B}^{\text{T}}{B}$. For $\vert v \vert>1$, using  Lemma \fcolorbox{red}{ white}{\ref{N6}}, we have
\begin{align*}
\overline{a}_{ij}(v)\eta_{i}\eta_{j}&=\ell_{1}(v)\vert P_{v}\eta  \vert^{2}+\ell_{2}(v)\vert (I-P_{v})\eta  \vert^{2}\\&\gtrsim{\langle v \rangle}^{\gamma}\vert \eta \vert^{2}{\cos^{2} (v,\eta)}+ {\langle v \rangle}^{\gamma+2} \frac{{\vert v\wedge\eta\vert}^{2}}{{\vert v\vert}^{2}}\\
&\gtrsim{\langle v \rangle}^{\gamma}\vert \eta \vert^{2}{\cos^{2} (v,\eta)}+{\langle v \rangle}^{\gamma}{\vert v\wedge\eta\vert}^{2}+{\langle v \rangle}^{\gamma} {\vert \eta\vert}^{2}\sin^{2} (v,\eta) \\
&\gtrsim{\langle v \rangle}^{\gamma}({\vert \eta\vert}^{2}+{\vert v\wedge\eta\vert}^{2}),
\end{align*}
on the other hand, we have
\begin{align*}
\overline{a}_{ij}(v)\eta_{i}\eta_{j}&=\ell_{1}(v)\vert P_{v}\eta  \vert^{2}+\ell_{2}(v)\vert (I-P_{v})\eta  \vert^{2}\\&\lesssim {\langle v \rangle}^{\gamma}\vert \eta \vert^{2}+ {\langle v \rangle}^{\gamma+2} \frac{{\vert v\wedge\eta\vert}^{2}}{{\vert v\vert}^{2}}\\
&\lesssim{\langle v \rangle}^{\gamma}({\vert \eta\vert}^{2}+{\vert v\wedge\eta\vert}^{2}).
\end{align*}
Hence the proof of (iii). 
\end{proof}
Using  Lemma \fcolorbox{red}{ white}{\ref{N200}}, 
we can rewrite   $ \mathcal{A}$ in the form
\begin{equation}\label{N9}
   \mathcal{A}=v\cdot\nabla_{x}+(B(v)\nabla_{v})^{*}\cdot B(v)\nabla_{v}+F(v),
   \end{equation}
   where $(B(v)D_{v})^{*}=  D_{v}B(v)^{T}$, is the formal adjoint of  $B(v)D_{v}$.
   \subsection{Study of the operator \texorpdfstring{$\mathcal{P}$}{}.}  
In this part, we will study the following problem:
  \begin{align}
 \left\{
    \begin{array}{ll}
    \partial_{t}f +\mathcal{P}f=0\\
    {f}_{|t=0}=f_{0},
     \end{array}
\right.
    \end{align}
we show that the above problem  is well-posed in the space $L^2(\mathbb{R}_{x}^{3}\times\mathbb{R}_{v}^{3})$ in the sense of semi-groups. By  Hille-Yosida Theorem, it is sufficient to show that $\mathcal{A}$ is maximal accretive in the space $L^2(\mathbb{R}_{x}^{3}\times\mathbb{R}_{v}^{3})$, then using the Bounded Perturbation Theorem in \cite[Theorem 1.3]{book}, we get that the operator $\mathcal{P}$ is a generator of a strongly continuous semigroup  (for more details on the semi-group theory see also \cite{pazy_semigroups_2012}).
First, we  start by recalling the basic definition of hypoellipticity. 
\begin{defi}
Let $P$  be a differential operator with   $C^{\infty}$ coefficients in an open set $\Omega \subset \mathbb{R}^{n}$. We say that $P$ is a hypoelliptic operator on $\Omega$, if,  for any open $\omega \subset \Omega$, any  
$u \in \mathcal{D'}(\Omega)$, such that $Pu \in C^{\infty}(\omega)$ belongs to $ C^{\infty}(\omega)$.
\end{defi}
\begin{lem}\label{N10}
Let $ \mathcal {A} $ be the operator defined in (\fcolorbox{red}{ white}{\ref{N14}}). Then,  $\mathcal{A}$  is a hypoelliptic operator.
\end{lem}
\begin{proof}
Using formula (\fcolorbox{red}{ white}{\ref{N9}}), we can rewrite $ \mathcal {A} $ in the following form:
\begin{align*}
 \mathcal{A}&=\underbrace{v\cdot\partial_{x}}_{X_{0}}+\underbrace{\big(\sum_{j=1}^{3}b_{1,j}(v)\partial_{v_{j}}\big)^{*}\big(\sum_{j=1}^{3}b_{1,j}(v)\partial_{v_{j}}\big)}_{{X_{1}^{*}X_{1}}}+\underbrace{\big(\sum_{j=1}^{3}b_{2,j}(v)\partial_{v_{j}}\big)^{*}\big(\sum_{j=1}^{3}b_{2,j}(v)\partial_{v_{j}}\big)}_{{X_{2}^{*}X_{2}}}\\
 &+\underbrace{\big(\sum_{j=1}^{3}b_{3,j}(v)\partial_{v_{j}}\big)^{*}\big(\sum_{j=1}^{3}b_{3,j}(v)\partial_{v_{j}}\big)}_{{X_{3}^{*}X_{3}}}+ F(v)\\
 &=X_{0}+\sum_{i=1}^{3}{X_{i}^{*}X_{i}}+ F(v).
\end{align*}
In addition, the coefficients  $b_{ij}(v)\hspace{0.1cm} \text{and}\hspace{0.1cm} F(v)$ are of class  $C^{\infty}$, taking the vector field square brackets we get
\begin{align*}
&Y_{k}=\left[ X_{k},X_{0}\right]=\sum_{j=1}^{3}b_{k,j}(v)\partial_{x_{j}}, \hspace{0.5cm} \text{for}  \hspace{0.1cm}  k \in \left \{1,2,3  \right \}. 
\end{align*}
From the above, $\mathcal{A}$  is a “type \MakeUppercase{\romannumeral 2} Hörmander’s  operators” (see for example \cite{article}, \cite{hormander1967hypoelliptic}).    
Moreover, the vector fields $\big\{X_{i},Y_{i},\hspace{0.2cm} i=1,\ldots,3\big\} $ generate all the space tangent to $\mathbb{R}^{6}_{x,v}$
so   $\mathcal{A}$ is hypoelliptic operator.
\end{proof}
\begin{thm}\label{N15}
Let $\gamma \in [{-3,1}]$ and $\mathcal{A}$ be the operator defined in (\fcolorbox{red}{ white}{\ref{N14}}). Then,  its closure $\overline{\mathcal{A}}$  on the space  $\mathcal{S}(\mathbb{R}^{6}_{x,v})$ 
 is maximally accretive.
  \end{thm}
  \begin{proof}
  We adapt here the proof given in  \cite[page 44]{article}. We apply the abstract criterion by taking
  $ \mathcal{H}=L^{2}(\mathbb{R}^{6}_{x,v})$ and the domain of  $\mathcal{A}$ defined by $D(  \mathcal{A})=\mathcal{S}(\mathbb{R}^{6}_{x,v})$.
First, we show the accretivity of the operator  $  \mathcal{A}$.
We want to prove that  $\Re{e}\left(\mathcal{A}u,u\right)_{\mathcal{H}}\geq 0$ for $ u \in D(  \mathcal{A})$. Indeed, from (\fcolorbox{red}{ white}{\ref{N9}})
\begin{align*}
\Re{e}\left(\mathcal{A}u,u\right)_{\mathcal{H}}&=\underbrace{\Re{e}\left(v\cdot\nabla_{x}u,u\right)_{\mathcal{H}}}_{=0\hspace{0.1cm} \text{since}\hspace{0.05cm}v\cdot\nabla_{x}\hspace{0.1cm}\text{is skew-adjoint }}
- \Re{e}\left(\nabla_{v}\cdot (\mathbf{A}(v)\nabla_{v}u),u\right)_{\mathcal{H}}+\Re{e}\left(F(v)u,u\right)_{\mathcal{H}}\\
&={\Vert B(v)\nabla_{v}u \Vert}_{\mathcal{H}}^{2}+{\Vert \sqrt{F(v)}u \Vert}_{\mathcal{H}}^{2}\\
&\geq 0.
\end{align*}
 Since $  \mathcal{A}$ is an accretive operator then its closure   $\overline{\mathcal{A}}$ exists and it is accretive (see \cite[Proposition 5.3]{article}).
Let us now show that there exists ${\lambda}_{1}>0$ such that the operator
$$T=  \mathcal{A}+\lambda_{1}\text{Id}$$
has dense image in $\mathcal{H}$. We take $\lambda_{1}=1$. Let $f\in \mathcal{H}$ satisfy 
\begin{equation}\label{M44}
    {\left(f,Tu\right)}_{\mathcal{H}}=0, \hspace{0.2cm} \forall u\in D(  \mathcal{A}).
\end{equation}
we want's to prove that $f=0$,\\
Since $T$ is a differential operator
then his formal adjoint $T^{\sharp}$ exists (in the sense of distributions). According to  (\fcolorbox{red}{ white}{\ref{M44}}) we obtain
\begin{align}\label{N11}
     T^{\sharp}f=\big(-\nabla_{v}\cdot \mathbf{A}(v)\nabla_{v}+F(v)+1-X_{0}\big)f=0, \hspace{0.2cm}\text{in}\hspace{0.2cm}  \mathcal{D}'(\mathbb{R}^{6}).
     \end{align}
Using Lemma \fcolorbox{red}{ white}{\ref{N10}}, we have  $-\nabla_{v}\cdot \mathbf{A}(v)\nabla_{v}+F(v)+1-X_{0}$ is a hypoelliptic operator, so $f \in C^{\infty}(\mathbb{R}^{6})$ (see  \cite[Chapter 2]{article}).\\
Now we introduce the family of truncation functions $\zeta_{k}$  defined
 by
$$\zeta_{k}(x,v)=\zeta\big(\frac{x}{k_{1}}\big)\zeta\big(\frac{v}{k_{2}}\big),\hspace{0.2cm}\forall k=(k_{1},k_{2})\in (\mathbb{N^{*}})^{2},$$ where $\zeta$ is a $C^{\infty}$ function satisfying  the following conditions:
$$
 \left\{
    \begin{array}{ll}
     0\leq\zeta\leq 1, \\
        \zeta= 1\hspace{0.1cm} \text{on}\hspace{0.2cm}B(0,1),\\
        \text{supp}\hspace{0.1cm}\zeta\subset B(0,2),\\
     \zeta\hspace{0.2cm} \text{is a radial function} .
    \end{array}
\right.
$$
The expression of $ T^{\sharp}(\zeta_{k}f)$ is
\begin{align*}
   T^{\sharp}(\zeta_{k}f)&=-\nabla_{v}\cdot \big(\mathbf{A}(v)\nabla_{v}(\zeta_{k}f)\big)+\big(F(v)+1\big)\zeta_{k}f-X_{0}(\zeta_{k}f)\\
   &=-\nabla_{v}\cdot \big(\mathbf{A}(v)([\nabla_{v}\zeta_{k}]f)\big)-\nabla_{v}\zeta_{k}\cdot \mathbf{A}(v)\nabla_{v}f-X_{0}(\zeta_{k})f-\zeta_{k} T^{\sharp}f,
\end{align*}
by using  (\fcolorbox{red}{ white}{\ref{N11}}), we obtain
\begin{align}
     T^{\sharp}(\zeta_{k}f)=-\nabla_{v}\cdot \big(\mathbf{A}(v)([\nabla_{v}\zeta_{k}]f)\big)-\nabla_{v}\zeta_{k}\cdot \mathbf{A}(v)\nabla_{v}f-X_{0}(\zeta_{k})f.
\end{align}
We note that $T^{\sharp}(\zeta_{k}f)\in\mathcal{H}$, taking the scalar product with $\zeta_{k}f$ we obtain
\begin{align*}
 {\left(T^{\sharp}(\zeta_{k}f),\zeta_{k}f\right)}_{\mathcal{H}}=&-\iint\nabla_{v}\cdot \big(\mathbf{A}(v)([\nabla_{v}\zeta_{k}]f)\big)\zeta_{k}f\hspace{0.1cm}dxdv-\iint X_{0}(\zeta_{k})\zeta_{k}\vert f \vert^{2}\hspace{0.1cm}dxdv\\&-\iint\big(\nabla_{v}\zeta_{k}\cdot \mathbf{A}(v)\nabla_{v}f\big)\zeta_{k}f\hspace{0.1cm}dxdv,
 \end{align*}
 By doing an integration by parts, we obtain 
\begin{align*}
 {\left(T^{\sharp}(\zeta_{k}f),\zeta_{k}f\right)}_{\mathcal{H}}=\iint{\vert B\nabla_{v}(\zeta_{k})f\vert}^{2}\hspace{0.1cm}dxdv-\iint X_{0}(\zeta_{k})\zeta_{k}\vert f \vert^{2}\hspace{0.1cm}dxdv.
 \end{align*}
 On the other hand, using (\fcolorbox{red}{ white}{\ref{N11}}), we obtain
 \begin{align*}
 {\left(T^{\sharp}(\zeta_{k}f),\zeta_{k}f\right)}_{\mathcal{H}}=&\iint{\vert B\nabla_{v}(\zeta_{k}f)\vert}^{2}\hspace{0.1cm}dxdv+\iint\big(F(v)+1\big)\vert \zeta_{k}f\vert^{2}\hspace{0.1cm}dxdv\\
 &-\underbrace{\iint X_{0}(\zeta_{k}f)\zeta_{k} f\hspace{0.1cm}dxdv}_{=0\hspace{0.1cm} \text{since}\hspace{0.1cm}X_{0}\hspace{0.1cm}\text{is skew-adjoint}}\\
 &=\iint{\vert B\nabla_{v}(\zeta_{k}f)\vert}^{2}\hspace{0.1cm}dxdv+\iint\big(F(v)+1\big)\vert \zeta_{k}f\vert^{2}\hspace{0.1cm}dxdv.
 \end{align*}
Using the fact that $$\iint{\vert B\nabla_{v}(\zeta_{k}f)\vert}^{2}\hspace{0.1cm}dxdv\geq0,$$ we obtain 
 \begin{align}\label{Z5}
     \iint\big(F(v)+1\big)\vert \zeta_{k}f\vert^{2}\hspace{0.1cm}dxdv\leq \underbrace{\iint{\vert B\nabla_{v}(\zeta_{k})f\vert}^{2}\hspace{0.1cm}dxdv}_{(i)}-\underbrace{\iint X_{0}(\zeta_{k})\zeta_{k}\vert f \vert^{2}\hspace{0.1cm}dxdv}_{(ii)}.
 \end{align}
 \underline{Estimate of (i):}
Using Lemma \fcolorbox{red}{ white}{\ref{N6}}-(b)-(d) and taking into account that the function $\zeta$ is a radial function we obtain
\begin{align*}
    \iint{\vert B\nabla_{v}(\zeta_{k})f\vert}^{2}\hspace{0.1cm}dxdv&=\iint \ell_{1}(v)\vert P_{v}\nabla_{v}\zeta_{k}  \vert^{2}\vert f\vert^{2}\hspace{0.1cm}dxdv+\underbrace{\iint \ell_{2}(v)\vert  (I-P_{v})\nabla_{v}\zeta_{k} \vert^{2}\vert f\vert^{2}\hspace{0.1cm}dxdv}_{=0\hspace{0.1cm} \text{since}\hspace{0.05cm}\hspace{0.1cm}v\hspace{0.1cm}\text{is parallel to}\hspace{0.1cm}\nabla_{v}\zeta_{k}}\\
    &\leq \frac{C_{0}}{{k}_{2}^{2}}\iint \langle v \rangle^{\gamma}\Phi_{k}\vert f\vert^{2} \hspace{0.1cm}dxdv,
    \end{align*}
where $C_{0}>0$ and $\Phi_{k}=\zeta(\frac{x}{k_{1}})\zeta^{'}(\frac{v}{k_{2}})$.
Using the fact that $\gamma \in [{-3,1}]$ and the fact that $\Phi_{k}$ is a bounded function we have the existence of a constant $C_{1}>0$ such that,
$$ \iint{\vert B\nabla_{v}(\zeta_{k})f\vert}^{2}\hspace{0.1cm}dxdv\leq C_{1}\big( \frac{1}{{k}_{2}^{2}}+ \frac{1}{{k}_{2}}\big)\Vert f \Vert^{2}.$$
 \underline{Esimate of (ii):} we have
 \begin{align*}
 \left|\iint X_{0}(\zeta_{k})\zeta_{k}\vert f \vert^{2}\hspace{0.1cm}dxdv\right|\leq \frac{1}{{k}_{1}}\iint \vert v \vert\Tilde{\Phi}_{k}\zeta_{k}\vert f\vert^{2} \hspace{0.1cm}dxdv,
  \end{align*}
  where $\Tilde{\Phi}_{k}=\zeta^{'}(\frac{x}{k_{1}})\zeta(\frac{v}{k_{2}}).$\\
 Now, taking into account that the functions   $\Tilde{\Phi}_{k}$ and  $\zeta_{k}$ are bounded, we have the existence of a constant $C_{2}>0$ such that,
 \begin{align*}
\left|\iint X_{0}(\zeta_{k})\zeta_{k}\vert f \vert^{2}\hspace{0.1cm}dxdv\right|\leq C_{2} \frac{{k}_{2}}{{k}_{1}}\Vert f \Vert^{2}.
\end{align*}
Finally, coming back to (\fcolorbox{red}{ white}{\ref{Z5}}) we obtained the existence of a constant $C>0$ such that,
\begin{align}\label{N12}
      \iint\vert \zeta_{k}f\vert^{2}\hspace{0.1cm}dxdv\leq  C\big( \frac{1}{{k}_{2}^{2}}+ \frac{1}{{k}_{2} }+\frac{{k}_{2}}{{k}_{1}}\big)\Vert f \Vert^{2},\hspace{0.2cm} \forall k.
      \end{align}
 Taking  $k_{1}\rightarrow +\infty$ in (\fcolorbox{red}{ white}{\ref{N12}}) we obtain
\begin{align}\label{N13}
      \iint\vert \zeta(\frac{v}{k_{2}})f\vert^{2}\hspace{0.1cm}dxdv\leq  C\big( \frac{1}{{k}_{2}^{2}}+ \frac{1}{{k}_{2}}\big)\Vert f \Vert^{2},
\end{align}
and taking $k_{2}\rightarrow +\infty$ in (\fcolorbox{red}{ white}{\ref{N13}}) we obtain
$$  \iint\vert f\vert^{2}\hspace{0.1cm}dxdv=0,$$
then $f=0.$
\end{proof}
 From now on, we write  $\mathcal{A}$ for the closure of the operator $\mathcal{A}$. 
\begin{cor}\label{N3000}
 Let  $\gamma \in [{-3,1}]$ and  $\mathcal{P}$ be the operator defined in (\fcolorbox{red}{ white}{\ref{N4}}). Then, $\mathcal{-P}$ is a generator of a semi-group $\big(S(t)\big)_{t\geq0}$  strongly continuous on $\mathcal{H}=L^{2}(\mathbb{R}^{6}_{x,v})$ verifying
 \begin{align}
     \Vert S(t) \Vert \leq e^{ \Vert\mathcal{K} \Vert t}\hspace{0.2cm}\text{for}\hspace{0.2cm} t\geq 0,
 \end{align}
 where $\mathcal{K}$ the operator defined in (\fcolorbox{red}{ white}{\ref{N14}}).
\end{cor}
\begin{proof}
Using (\fcolorbox{red}{ white}{\ref{N14}}), the operator $\mathcal{-P}$ is written as follows
  $$-\mathcal{P}=-\mathcal{A}-\mathcal{K}.$$
  According to Theorem  \fcolorbox{red}{ white}{\ref{N15}}, $\mathcal{A}$ is a maximally accretive operator. According to the Hille-Yosida Theorem, $-\mathcal{A}$ is a generator of a strongly continuous semi-group of contraction. On the other hand,  $-\mathcal{K}$  is a bounded operator in $\mathcal{H}$ (we have $\mathcal{L}_{2}$ a compact operator and  $M\chi_{R}$ is bounded). Using 
  the Bounded Perturbation Theorem  in \cite[Theorem 1.3]{book}, we have that  $\mathcal{-P}$  is a generator of a semi-group  $\big(S(t)\big)_{t\geq0}$  strongly continuous on  $\mathcal{H}$, moreover we have
 \begin{align}
     \Vert S(t) \Vert \leq e^{ \Vert\mathcal{K} \Vert t}\hspace{0.2cm}\text{for}\hspace{0.2cm} t\geq 0.
     \end{align}
\end{proof}

\section{Hypoelliptic estimates for the operator with parameters}\label{z2}

In the following discussion, we work with $\gamma \in [{0,1}]$. 
In this section, we will study the operator acting on the velocity variable $v$:
\begin{align}
   \mathcal{A}_{\xi}=iv\cdot \xi+(B(v)\nabla_{v})^{*}\cdot B(v)\nabla_{v}+F(v),
        \end{align}
    where $\xi$ is the parameter \text{in}  $\mathbb{R}^{3}.$  The operator $\mathcal{A}_{\xi}$ is obtained by the partial Fourier transformation in $ x $. 
    The goal of studying the operator $\mathcal{A}_{\xi}$ and considering $\xi$ as a parameter, is to obtain estimates of the velocity variables $ v $ uniformly with respect to $ \xi $. Then by using the inverse Fourier transform with respect to $ x $, we can obtain global estimates in all variables.
    We note that the operator $\mathcal{A}_{\xi}$ verifies for all  $u\in \mathcal{S}(\mathbb{R}^{3}_{v})$,  \begin{equation}\label{N502}
{\Vert B(v)\nabla_{v}u \Vert}_{L^{2}}^{2}+{\Vert \sqrt{F(v)}u \Vert}_{L^{2}}^{2}\leq \Re{e}\left(\mathcal{A}_{\xi}u,u\right)_{L^{2}}.   
    \end{equation} 
\textbf{Notations.} Throughout this section, we will use  $\Vert \cdot \Vert_{L^{2}}$ to denote the norm in the space  $L^{2}(\mathbb{R}^{3}_{v})$ and $\xi$ is a parameter.  We use $p^{\text {Wick}}$ to denote  the Wick quantization of  $p$ in the variables $(v,\eta)$ (for more details on Wick quantization see \cite{Lerner2008} and Appendix \fcolorbox{red}{ white}{\ref{P2}}).\\
The main result in this section is  Proposition \fcolorbox{red}{ white}{\ref{N100}} and  Proposition \fcolorbox{red}{ white}{\ref{N105}}.

\subsection{Pseudo-differential parts}\label{W1}
In this part, we will show several lemmas concerning pseudo-differential symbols. We need to build these symbols, who verify assumptions of Theorem \fcolorbox{red}{ white}{\ref{M36}},  such that the pseudo-differential operator associated to these symbols has good properties. These operators play an important role in hypoellipic estimates. The standard concepts on pseudo-differential calculus are explained in Appendix \fcolorbox{red}{ white}{\ref{1001}}. We define for $(v,\eta)\in \mathbb{R}^{6}$  the following symbols, they depend on the parameter $\xi$ but we do not mention it in our notations since $\xi$ is seen as a parameter.
\begin{align}\label{N20}
\lambda(v, \eta)&=\left[\langle v \rangle^\gamma\left(1+\vert v\vert^{2}+\vert \eta\vert^{2}+\vert \xi\vert^{2}+\vert v\wedge\eta\vert^{2}+\vert v\wedge\xi\vert^{2}\right)\right]^{1/2},\\
\mathbf{a}(v,\eta)&=1+\vert v\vert^{2}+\vert \eta\vert^{2}+\vert \xi\vert^{2}+\vert v\wedge\eta\vert^{2}+\vert v\wedge\xi\vert^{2},\\g_{1}(v, \eta)&=1+\langle v \rangle+\langle \eta \rangle,\\
        g_{2}(v, \eta)&=1+{\langle \xi \rangle}^{1/3}+{\langle \eta \rangle},\\
   g_{3}(v, \eta)&=1+{\langle \xi \rangle}^{2/3}+{\langle \eta \rangle}^{2}.
\end{align}
\begin{lem}
The above symbols are  admissible weights in the sense of Definition \fcolorbox{red}{ white}{\ref{S1}}  uniformly with respect to the parameter $\xi$ \hspace{0.1cm}\text{in}\hspace{0.1cm}  $\mathbb{R}^{3}.$
\end{lem}
\begin{proof}
 We have to check $\lambda$ is an admissible weight. It is sufficient to verify that there exists two constants $N$ and $C$, both depending only on $\gamma$, such that for all\\ $Y=(v,\eta),\hspace{0.1cm} Y'=(v', \eta')$, we have
 $$\lambda(Y)\leq C \lambda(Y')\left(1+\Gamma(Y-Y')\right)^{N},$$
 where $\Gamma$ is the metric defined by $\Gamma=dv^{2}+d\eta^{2}$.
We have
\begin{equation}
    \frac{\lambda^{2}(v,\eta)}{\lambda^{2}(v', \eta')}=\frac{\langle v \rangle^\gamma}{\langle v' \rangle^{\gamma}} \left(\frac{1+\vert v\vert^{2}+\vert \eta\vert^{2}+\vert \xi\vert^{2}+\vert v\wedge\eta\vert^{2}+\vert v\wedge\xi\vert^{2}}{1+\vert v'\vert^{2}+\vert \eta'\vert^{2}+\vert \xi\vert^{2}+\vert v'\wedge\eta'\vert^{2}+\vert v'\wedge\xi\vert^{2}}\right).
\end{equation}
We now use Peetre’s inequality
\begin{align}\label{2000}
   {\langle y \rangle}^\tau\leq 2^{\frac{\vert\tau\vert}{2}}{\langle y' \rangle}^\tau{\langle y-y' \rangle}^{\vert\tau\vert},\hspace{0.2cm} \tau\in \mathbb{R},
\end{align}
to get 
$$\frac{\langle v \rangle^\gamma}{\langle v' \rangle^{\gamma}}\leq 2^{\gamma}{\langle v-v' \rangle}^{\gamma}.$$ 
Using (\fcolorbox{red}{ white}{\ref{2000}}), we obtain
$$\frac{1+\vert v\vert^{2}+\vert \eta\vert^{2}}{1+\vert v'\vert^{2}+\vert \eta'\vert^{2}+\vert \xi\vert^{2}+\vert v'\wedge\eta'\vert^{2}+\vert v'\wedge\xi\vert^{2}}\leq \frac{\langle v \rangle^2}{\langle v' \rangle^{2}}+\frac{\langle \eta \rangle^{2}}{\langle \eta' \rangle^{2}}\leq 4\left(\langle v-v' \rangle+\langle \eta-\eta' \rangle \right)^{2}$$ 
and
$$\frac{\vert \xi\vert^{2}}{1+\vert v'\vert^{2}+\vert \eta'\vert^{2}+\vert \xi\vert^{2}+\vert v'\wedge\eta'\vert^{2}+\vert v'\wedge\xi\vert^{2}}\leq 1\leq \left(\langle v-v' \rangle+\langle \eta-\eta' \rangle \right)^{2}.$$ 
Using the relation
$$ v\wedge\xi=(v-v')\wedge\xi+v'\wedge\xi,$$
we obtain
$$\frac{\vert v\wedge\xi\vert^{2}}{1+\vert v'\vert^{2}+\vert \eta'\vert^{2}+\vert \xi\vert^{2}+\vert v'\wedge\eta'\vert^{2}+\vert v'\wedge\xi\vert^{2}}\leq 1+\frac{\langle v-v' \rangle^2\langle \xi \rangle^{2}}{\langle \xi \rangle^{2}}\leq 2\left(\langle v-v' \rangle+\langle \eta-\eta' \rangle \right)^{2}.$$ 
Moreover using the relation
$$ v\wedge\eta=(v-v')\wedge(\eta-\eta') + (v-v')\wedge\eta'+v'\wedge(\eta-\eta')+v'\wedge\eta',$$
we compute
\begin{align*}
    &\frac{\vert v\wedge\eta\vert^{2}}{1+\vert v'\vert^{2}+\vert \eta'\vert^{2}+\vert \xi\vert^{2}+\vert v'\wedge\eta'\vert^{2}+\vert v'\wedge\xi\vert^{2}}\\
    &\leq\frac{4\vert v-v'\vert^{2}\vert \eta-\eta'\vert^{2}+4\vert v-v'\vert^{2}\vert \eta'\vert^{2}+4\vert v'\vert^{2}\vert \eta-\eta'\vert^{2}+4\vert v'\wedge\eta'\vert^{2}}{1+\vert v'\vert^{2}+\vert \eta'\vert^{2}+\vert \xi\vert^{2}+\vert v'\wedge\eta'\vert^{2}+\vert v'\wedge\xi\vert^{2}}\\
    &\leq 4\vert v-v'\vert^{2}\vert \eta-\eta'\vert^{2}+4\vert v-v'\vert^{2}+4\vert \eta-\eta'\vert^{2}+4\\
    &\leq 10\left(\langle v-v' \rangle+\langle \eta-\eta' \rangle \right)^{4}.
\end{align*}
Combining the above inequalities, we get 
\begin{equation*}
    \frac{\lambda(Y)}{\lambda(Y')}\leq C_{\gamma}\left(1+\Gamma(Y-Y')\right)^{\frac{4+\gamma}2},
\end{equation*}
so $\lambda$ is an admissible weight. The proof will be the same for $a, g_{1}, g_{2}$ and $ g_{3}.$
\end{proof}
\begin{lem}\label{N26}
For $m\in \mathbb{R}$, 
 $$\lambda^{m}\in S(\lambda^{m},\Gamma),$$
uniformly with respect to the parameter $\xi$ \hspace{0.1cm}\text{in}\hspace{0.1cm}  $\mathbb{R}^{3}$, where $S(\lambda^{m},\Gamma)$ is defined in Definition \fcolorbox{red}{ white}{\ref{S2}} in Appendix \fcolorbox{red}{ white}{\ref{S3}}.
\end{lem}
\begin{proof}
We can rewrite  $\lambda$ as follows 
 $$\lambda(v,\eta)=\langle v \rangle^{\frac{\gamma}{2}} {\mathbf{a}^{\frac{1}{2}}(v,\eta)}.$$
To prove the wanted result, we use induction on $\vert\alpha+\beta\vert$ to prove that for any $k\in\mathbb{R}$ and any $\vert\alpha+\beta\vert\geq 0$,
 \begin{align}\label{N16}
 \vert\partial_{v}^{\alpha}\partial_{\eta}^{\beta}{\mathbf{a}^{k}(v, \eta)} \vert\lesssim {\mathbf{a}^{k}(v, \eta)},
  \end{align}
which obviously holds for  $\vert\alpha+\beta\vert=0$. Now suppose $\vert\alpha+\beta\vert\geq 1$, then we have either $\vert\alpha\vert\geq 1$ or $\vert\beta\vert\geq 1$, and suppose $\vert\beta\vert\geq 1$ without loss of generality. So we can write $\partial_{\eta}^{\beta}=\partial_{\eta}^{\Tilde{\beta}}\partial_{\eta_{j}}$ with $\vert\Tilde{\beta}\vert=\vert\beta\vert-1$ and thus
 \begin{align*}
     \partial_{v}^{\alpha}\partial_{\eta}^{\beta}\mathbf{a}^{k}(v, \eta)= \partial_{v}^{\alpha}\partial_{\eta}^{\Tilde{\beta}}\left[k \mathbf{a}^{k-1}\big(2\eta_{j}+2(v\wedge\eta)\partial_{\eta_{j}}(v\wedge\eta)\big)\right],
 \end{align*}
 which along with Leibniz’ formula and the induction assumption yields
  \begin{align*}
  \vert\partial_{v}^{\alpha}\partial_{\eta}^{\beta}\mathbf{a}^{k}(v, \eta) \vert&\lesssim \mathbf{a}^{k-1}\big(1+\vert \eta\vert+\vert v\vert\vert \eta\vert+\vert v\vert\vert v\wedge\eta\vert+\vert v\vert^{2}\big)\\
  &\lesssim {\mathbf{a}^{k}(v, \eta)},
  \end{align*}
  on the other hand, we have for all $k\in\mathbb{R}$,
  \begin{align}\label{N17}
  \langle v \rangle^k \in S({\langle v \rangle}^{k},\Gamma).
    \end{align}
 Finally, using (\fcolorbox{red}{ white}{\ref{N16}}),  (\fcolorbox{red}{ white}{\ref{N17}}) and Leibniz’ formula we conclude  for all $\vert\alpha+\beta\vert\geq 0$,
 \begin{align}
 \vert\partial_{v}^{\alpha}\partial_{\eta}^{\beta}\lambda^{m}(v, \eta) \vert\leq C_{\alpha,\beta}\hspace{0.05cm}\lambda^{m}(v, \eta).
  \end{align}
\end{proof}
\begin{lem}\label{N27}
For $m\in \mathbb{R}$, 
 $$\vert\xi\cdot\partial_{\eta}\lambda^{m}\vert \lesssim \lambda^{m},$$
uniformly with respect to the parameter $\xi$ \text{in}  $\mathbb{R}^{3}.$
\end{lem}
\begin{proof}
Lemma \fcolorbox{red}{ white}{\ref{N27}} follows directly from (\fcolorbox{red}{ white}{\ref{N20}}) and the fact that
$$\vert\xi\cdot\partial_{\eta}\lambda^{2}\vert=\lambda{\langle v \rangle}^{\frac{\gamma}{2}} \mathbf{a}^{-1/2}\vert\xi\cdot\eta+v\wedge\xi\cdot v\wedge\eta\vert\lesssim \lambda^{2},$$
we can conclude that for all $m\in \mathbb{R}$, 
 $$\vert\xi\cdot\partial_{\eta}\lambda^{m}\vert \lesssim \lambda^{m},$$
 uniformly with respect to the parameter $\xi$ \text{in} $\mathbb{R}^{3}.$
 \end{proof}
\begin{lem}
We have
$$g_{i}\in S(g_{i},\Gamma),\hspace{0.1cm} \text{for}\hspace{0.2cm}i=1,\ldots,3, $$
uniformly with respect to the parameter $\xi$ \text{in}  $\mathbb{R}^{3}.$
\end{lem}

\begin{proof}
Using the fact that for all $m\in \mathbb{R}$, ${\langle v \rangle}^{m}\in S({\langle v \rangle}^{m},\Gamma)$, ${\langle \eta \rangle}^{m}\in S({\langle \eta \rangle}^{m}, \Gamma)$ we obtain
$$\forall \alpha,\beta \in \mathbb{N}^{3},\hspace{0.2cm}\vert \partial_{v}^{\alpha}\partial_{\eta}^{\beta}g_{i}(v, \eta)\vert \leq C_{\alpha,\beta}\hspace{0.05cm}g_{i},$$
uniformly with respect to the parameter $\xi$ \text{in}  $\mathbb{R}^{3}.$
\end{proof}
\begin{lem}\label{M80}
For all $\varepsilon > 0$,
\begin{tabenum}
[{i)}]\tabenumitem $\partial_{\eta}\lambda\in S(\varepsilon\lambda +\varepsilon^{-1}{\langle v \rangle}^{\frac{\gamma}{2}+1},\Gamma),$ \tabenumitem $\partial_{\eta}g_{1}\in S(\varepsilon g_{1} +\varepsilon^{-1},\Gamma),$\\\tabenumitem $\partial_{\eta}g_{2}\in S(\varepsilon g_{2} +\varepsilon^{-1}{\langle v \rangle}^{\frac{\gamma}{6}+1},\Gamma),$\tabenumitem $\partial_{\eta}g_{3}\in S(\varepsilon g_{3} +\varepsilon^{-1}{\langle v \rangle}^{\frac{\gamma}{3}+1},\Gamma);$\end{tabenum} 
uniformly with respect to the parameter $\xi$ \hspace{0.1cm}\text{in}\hspace{0.1cm}  $\mathbb{R}^{3}.$
\end{lem}
\begin{proof}
We have
\begin{align*}
    \vert \partial_{\eta}g_{1}(v, \eta) \vert&=\vert \partial_{\eta}\big(1+\langle v \rangle+\langle \eta \rangle\big) \vert
  \lesssim \vert\eta \vert \langle \eta \rangle^{-1}
     \lesssim g_{1}^{1/2},\\
    \vert \partial_{\eta}g_{2}(v, \eta) \vert&=\vert \partial_{\eta}\big(1+{\langle \xi \rangle}^{1/3}+\langle \eta \rangle \big) \vert
   \lesssim \vert\eta \vert \langle \eta \rangle^{-1}
    \lesssim g_{2}^{1/2}{\langle v \rangle}^{\frac{\gamma}{12}+\frac{1}{2}},\\
   \vert \partial_{\eta}g_{3}(v, \eta) \vert&=\vert \partial_{\eta}\big(1+{\langle \xi \rangle}^{2/3}+{\langle \eta \rangle}^{2} \big) \vert \lesssim
     \langle \eta \rangle
    \lesssim g_{3}^{1/2}{\langle v \rangle}^{\frac{\gamma}{6}+\frac{1}{2}}.
  \end{align*}
  Using Young's inequality we get for all $\varepsilon > 0$, 
   \begin{align*}
    \vert \partial_{\eta}g_{1}(v, \eta) \vert\lesssim  \varepsilon g_{1} +\varepsilon^{-1},\hspace{0.1cm} 
    \vert \partial_{\eta}g_{2}(v, \eta) \vert\lesssim  \varepsilon g_{2} +\varepsilon^{-1}{\langle v \rangle}^{\frac{\gamma}{6}+1},\hspace{0.1cm} 
    \vert \partial_{\eta}g_{3}(v, \eta) \vert\lesssim  \varepsilon g_{3} +\varepsilon^{-1}{\langle v \rangle}^{\frac{\gamma}{3}+1},
    \end{align*}
    then arguing as above we can use induction on $ \vert \alpha \vert +\vert \beta \vert$ to obtain, for $ \vert \alpha \vert +\vert \beta \vert \geq 0$,
    $$\vert \partial_{v}^{\alpha}\partial_{\eta}^{\beta}\partial_{\eta}g_{1}(v, \eta)\vert \lesssim \varepsilon  g_{1} +\varepsilon^{-1},\hspace{0.4cm}\vert \partial_{v}^{\alpha}\partial_{\eta}^{\beta}\partial_{\eta}g_{2}(v, \eta)\vert \lesssim \varepsilon  g_{2} +\varepsilon^{-1}{\langle v \rangle}^{\frac{\gamma}{6}+1},$$
    $$\text{and}\hspace{0.2cm}\vert \partial_{v}^{\alpha}\partial_{\eta}^{\beta}\partial_{\eta}g_{3}(v, \eta)\vert \lesssim \varepsilon  g_{3} +\varepsilon^{-1}{\langle v \rangle}^{\frac{\gamma}{3}+1}.$$
  Regarding the symbol $\lambda$,  we have
    \begin{align*}\label{M75}
    \vert \partial_{\eta}\lambda(v, \eta) \vert={\langle v \rangle}^{\frac{\gamma}{2}}\mathbf{a}^{-1/2} \big\vert\eta+(v\wedge\eta)\partial_{\eta}(v\wedge\eta)\big\vert
    \lesssim {\langle v \rangle}^{\frac{\gamma}{2}}\mathbf{a}^{-1/2} \big(\vert \eta\vert+\vert v\vert\vert v\wedge\eta\vert\big)
  \lesssim  {\langle v \rangle}^{\frac{\gamma}{4}+\frac{1}{2}}\lambda^{1/2}.
  \end{align*}
     Using Young’s inequality we get for all $\varepsilon > 0$,
   \begin{align*}
    \vert \partial_{\eta}\lambda(v, \eta) \vert\lesssim  \varepsilon \lambda +\varepsilon^{-1} \langle v \rangle^{\frac{\gamma}{2}+1},
     \end{align*}
     then arguing as above we can use induction on $ \vert \alpha \vert +\vert \beta \vert$ to obtain, for $ \vert \alpha \vert +\vert \beta \vert \geq 0$,
    $$\vert \partial_{v}^{\alpha}\partial_{\eta}^{\beta}\partial_{\eta}\lambda(v, \eta) \vert\lesssim  \varepsilon \lambda +\varepsilon^{-1} \langle v \rangle^{\frac{\gamma}{2}+1}.$$
    \end{proof}
    From the above, we have shown that the symbols $ g_{1},\hspace{0.1cm}  g_{2},\hspace{0.1cm}   g_{3}$ and $\lambda$ verify the hypotheses of Theorem \fcolorbox{red}{ white}{\ref{M36}}, so we can apply the results of Theorem \fcolorbox{red}{ white}{\ref{M36}}  to the following operators
    \begin{align}
      g_{1,K}^{w}&=\big (g_{1}+K\big )^{w},\\ 
         \label{N70}  
      g_{2,K}^{w}&=\big (g_{2}+K{\langle v \rangle}^{\frac{\gamma}{6}+1}\big )^{w},\\
      \label{N91}
      g_{3,K}^{w}&=\big (g_{3}+K{\langle v \rangle}^{\frac{\gamma}{3}+1}\big )^{w},\\
      \label{N52}   
       \lambda_{K}^{w}&=\big (\lambda+K\langle v \rangle^{\frac{\gamma}{2}+1}\big )^{w},
    \end{align}
    where $K$ the fixed constant given by Theorem  \fcolorbox{red}{ white}{\ref{M36}}. In  Section \ref{2001}, we will apply the results of Theorem \fcolorbox {red} {white} {\ref{M36}} on the operators above.\\
    
 Let $\psi$ be a $  C_{0}^{\infty}(\mathbb{R},[{0,1}])$ function such that
    \begin{align}\label{N21}
          \psi= 1\hspace{0.1cm} \text{on}\hspace{0.2cm}[{-1,1}],
        \hspace{0.2cm}\text{supp}\hspace{0.1cm} \psi\subset [{-2,2}].
         \end{align}
         \begin{defi}
     Define the real-valued symbol
        \begin{align}\label{N36}
        g=-\frac{ B(v)\xi \cdot B(v)\eta }{ \lambda^{4/3}} \psi\bigg(\frac{{\vert B(v)\eta \vert}^{2}+F(v)}{\lambda^{2/3}}\bigg),
        \end{align}
       where $\lambda$ is the symbol defined in (\fcolorbox{red}{ white}{\ref{N20}}).
          \end{defi}
    
\begin{lem}[\textbf{Lemma 3.3 in \cite{herau_anisotropic_2011}}]\label{N25}
We have
$$ \psi\bigg(\frac{{\vert B(v)\eta \vert}^{2}+F(v)}{\lambda^{2/3}}\bigg) \in S(1,\Gamma),$$
uniformly with respect to the parameter   $\xi$ \text{in} $\mathbb{R}^{3}$.
\end{lem}
\begin{lem}\label{P5}
The symbol $g$  belongs to the class
  $S(1,\Gamma)$ uniformly with respect to the parameter   $\xi$ \text{in} $\mathbb{R}^{3}$.
\end{lem}
\begin{proof}
Notice from (\fcolorbox{red}{ white}{\ref{N21}}) that
\begin{align}
  {\vert B(v)\eta \vert}^{2}+F(v)\leq 2{\lambda^{2/3}},
\end{align}
on the support of the function
\begin{align}
    \psi\bigg(\frac{{\vert B(v)\eta \vert}^{2}+F(v)}{\lambda^{2/3}}\bigg).
\end{align}
By recalling (\fcolorbox{red}{ white}{\ref{N20}}) and using (\fcolorbox{red}{ white}{\ref{N501}}), we obtain
\begin{align}
\vert B(v)\xi \vert \lesssim{\langle v \rangle}^\frac{\gamma}{2}\vert \xi\vert +{\langle v \rangle}^\frac{\gamma}{2}\vert v\wedge\xi\vert\lesssim\lambda.
\end{align}
We deduce from the Cauchy-Schwarz inequality that
one can estimate
\begin{align}\label{N37}
\vert B(v)\xi \cdot B(v)\eta \vert \lesssim \vert B(v)\xi\vert  \vert B(v)\eta\vert  \lesssim \lambda^{4/3},
\end{align}
on the support of $\psi$. The symbol  $g$  is therefore a bounded function uniformly with respect to the parameter   $\xi$ \text{in}  $\mathbb{R}^{3}$. Using (\fcolorbox{red}{ white}{\ref{N22}}) and (\fcolorbox{red}{ white}{\ref{N501}}), we can estimate 
\begin{align}\label{N23}
    \vert \partial_{v}^{ \alpha }{ B(v)\xi }\vert \lesssim {\langle v \rangle}^\frac{\gamma}{2}\vert \xi\vert \lesssim\vert B(v)\xi\vert \lesssim \lambda 
    \end{align}
    where $ \alpha \in \mathbb{N}^{3}$ with $\vert \alpha \vert \geq 1$. Using again (\fcolorbox{red}{ white}{\ref{N22}}) and (\fcolorbox{red}{ white}{\ref{N501}}), we can estimate 
    \begin{align}\label{N24}
     \vert \partial_{v}^{ \alpha }{ B(v)\eta }\vert \lesssim {\langle v \rangle}^\frac{\gamma}{2}\vert \eta\vert \lesssim\vert B(v)\eta\vert \lesssim \lambda^{1/3}
\end{align}
 on the support of $\psi$. Morever, one can estimate from above the modulus of all the derivatives of the term $B(v)\eta$ by a constant times $\lambda^{1/3}$ on the support of the function
$$\psi\bigg(\frac{{\vert B(v)\eta \vert}^{2}+F(v)}{\lambda^{2/3}}\bigg).$$
Using Leibniz's formula, Cauchy-Schwarz inequality (\fcolorbox{red}{ white}{\ref{N23}}), (\fcolorbox{red}{ white}{\ref{N24}})  one can estimate from above the
modulus of all the derivatives of the term $B(v)\xi \cdot B(v)\eta$ by a constant times  $\lambda^{4/3}$ on the support of the fonction $\psi$. According to Lemma
\fcolorbox{red}{ white}{\ref{N26}} and Lemma 
\fcolorbox{red}{ white}{\ref{N25}}, this proves that the symbol $g$ belongs to the class  $S(1,\Gamma)$  uniformly with respect to the parameter   $\xi$ \text{in} $\mathbb{R}^{3}$.
\end{proof}
\begin{lem}\label{N28}
We have
$$\left| \xi\cdot\partial_{\eta}\left[\psi\bigg(\frac{{\vert B(v)\eta \vert}^{2}+F(v)}{\lambda^{2/3}}\bigg) \right]\right|\lesssim 1+{\vert B(v)\eta \vert}^{2}+F(v),$$
 uniformly with respect to the parameter   $\xi$ \text{in}  $\mathbb{R}^{3}$.
\end{lem}
\begin{proof}
Let $\omega=\frac{{\vert B(v)\eta \vert}^{2}+F(v)}{\lambda^{2/3}}$. We may write
\begin{align*}
    \xi\cdot\partial_{\eta}\left[\psi(\omega) \right]=\psi'(\omega)&\left[\frac{2 B(v)\xi \cdot B(v)\eta }{ \lambda^{2/3}}
    +\big({\vert B(v)\eta \vert}^{2}+F(v)\big)(\xi\cdot\partial_{\eta})(\lambda^{-2/3})\right].
    \end{align*}
    Notice from (\fcolorbox{red}{ white}{\ref{N20}}) and (\fcolorbox{red}{ white}{\ref{N21}}) 
    $$\left|\frac{2 B(v)\xi \cdot B(v)\eta }{ \lambda^{2/3}}\right| \lesssim \frac{ \vert B(v)\xi \vert\vert B(v)\eta\vert}{ \lambda^{2/3}}\lesssim\frac{\lambda\lambda^{1/3}}{\lambda^{2/3}}\lesssim \lambda^{2/3}\lesssim {\vert B(v)\eta \vert}^{2}+F(v),$$
    on the support of the function $\psi'(\omega)$. One can then deduce Lemma \fcolorbox{red}{ white}{\ref{N28}}  from \\ Lemma \fcolorbox{red}{ white}{\ref{N27}}. 
\end{proof}
\subsection{Hypoelliptic estimates}\label{2001}
 We shall consider the multiplier $G= g^{\text {Wick}}$  defined by the Wick quantization of the symbol $g$.
We refer the reader to Appendix \fcolorbox{red}{ white}{\ref{P2}} on Wick calculus. We begin by noticing from  (\fcolorbox{red}{ white}{\ref{P3}}) that there exists a real-valued symbol  $\Tilde{g}$ belonging
to the class $S(1,\Gamma)$ uniformly with respect to the parameter $\xi$ \text{in} $\mathbb{R}^{3}$ such that
\begin{align}\label{N32}
    G= g^{\text {Wick}}=\Tilde{g}^{w};
    \end{align}
where $\Tilde{g}^{w}$ denotes the operator obtained by the Weyl quantization of the symbol $\Tilde{g}$ given by
\begin{align}\label{N90}
(\Tilde{g}^{w}u)(v)=\dfrac{1}{(2\pi)^{3}}\int_{\mathbb{R}^{6}}e^{i(v-v')\cdot\eta}\Tilde{g}\bigg(\dfrac{v+v'}{2},\eta\bigg)u(v')\hspace{0.1cm}dv'd\eta.
\end{align}
We shall sometimes closely follow \cite{herau_anisotropic_2011} and refer to appendix (Section \fcolorbox{red}{ white}{\ref{1001}}) for the main features of the Wick and the Weyl quantizations.\\
We begin by a series of Lemmas whose proof is exactly the same as the one in \cite{herau_anisotropic_2011}.
\begin{lem}\label{N34}
There exists  $c_{1}>0$ such that for all $u\in \mathcal{S}(\mathbb{R}^{3}_{v})$,
\begin{align*}
\vert\left(F(v)u,Gu\right)_{L^{2}}\vert+\vert\left(\nabla_{v}\cdot (\mathbf{A}(v)\nabla_{v}u),Gu\right)_{L^{2}}\vert\leq c_{1}\hspace{0.1cm} \Re{e}\left(\mathcal{A}_{\xi}u,u\right)_{L^{2}},
\end{align*}
uniformly with respect to the parameter $\xi$ \text{in} $\mathbb{R}^{3}$.
\end{lem}
\begin{proof}
See Lemma 3.7 in \cite{herau_anisotropic_2011}.
\end{proof}
\begin{lem}\label{N41}
There exists  $c_{2}>0$  such that for all $u\in \mathcal{S}(\mathbb{R}^{3}_{v})$,
$$
\left| \Vert B(v)\nabla_{v}u \Vert_{L^{2}}^{2}+\Vert \sqrt{F(v)}u \Vert_{L^{2}}^{2}-
\left({[4\pi^{2}{\vert B(v)\eta \vert}^{2}+F(v)]}^{\text{Wick}}u,u\right)_{L^{2}}\right|\leq c_{2}\hspace{0.1cm} \Re{e}\left(\mathcal{A}_{\xi}u,u\right)_{L^{2}},$$
uniformly with respect to the parameter $\xi$ \text{in} $\mathbb{R}^{3}$.
\end{lem}
\begin{proof}
See Lemma 3.11 in \cite{herau_anisotropic_2011}.
\end{proof}
Let $\delta$ be a positive parameter such that $0<\delta\leq1$. We use a multiplier method and write that
\begin{align}\label{N30}
  \begin{split}
    \Re{e}\left(\mathcal{A}_{\xi}u,(1-\delta G)u\right)_{L^{2}}=&\Vert B(v)\nabla_{v}u \Vert_{L^{2}}^{2}+\Vert \sqrt{F(v)}u \Vert_{L^{2}}^{2}-\delta\Re{e}\left(iv\cdot{\xi}u, Gu\right)_{L^{2}}\\&+\delta\Re{e}\left(\nabla_{v}\cdot (\mathbf{A}(v)\nabla_{v}u),Gu\right)_{L^{2}}-\delta\Re{e}\left(F(v)u,Gu\right)_{L^{2}}.
      \end{split}
\end{align}
\begin{lem}\label{P8}
 We have for any $s\in \mathbb{R}$
$$\Vert{\langle B(v)\xi\rangle}^{s}(1-\delta G) u\Vert_{L^{2}}
\lesssim \Vert{\langle B(v)\xi\rangle}^{s} u\Vert_{L^{2}},$$
uniformly with respect to the parameter $\xi$ \text{in} $\mathbb{R}^{3}$.
\end{lem}
\begin{proof}
See Lemma 3.8 in \cite{herau_anisotropic_2011}.
\end{proof}
\begin{prop}\label{N51}
There exists  $C>0$ such that for all $u\in \mathcal{S}(\mathbb{R}^{3}_{v})$,
$$
((\lambda^{2/3})^{\text{Wick}}u,u)_{L^{2}}\leq C\big( \Re{e}(\mathcal{A}_{\xi}u,u)_{L^{2}}+ \Re{e}(\mathcal{A}_{\xi}u,(1-\delta G)u)_{L^{2}}\big),
$$
uniformly with respect to the parameter $\xi$ \text{in} $\mathbb{R}^{3}$.
\end{prop}
\begin{proof}
Using (\fcolorbox{red}{ white}{\ref{N30}}), we have for $0<\delta\leq1$ and $u\in \mathcal{S}(\mathbb{R}^{3}_{v})$,
\begin{multline}\label{N31}
 \Vert B(v)\nabla_{v}u \Vert_{L^{2}}^{2}+\Vert \sqrt{F(v)}u \Vert_{L^{2}}^{2}-\delta\Re{e}\left(iv\cdot{\xi}u, Gu\right)_{L^{2}}\\=   \Re{e}\left(\mathcal{A}_{\xi}u,(1-\delta G)u\right)_{L^{2}}-\delta\Re{e}\left(\nabla_{v}\cdot (\mathbf{A}(v)\nabla_{v}u),Gu\right)_{L^{2}}+\delta\Re{e}\left(F(v)u,Gu\right)_{L^{2}},
 \end{multline}
uniformly with respect to the parameter $\xi$ \text{in} $\mathbb{R}^{3}$. Recalling (\fcolorbox{red}{ white}{\ref{N32}}) and noticing from (\fcolorbox{red}{ white}{\ref{P3}}) and (\fcolorbox{red}{ white}{\ref{P4}}) that $v^{\text{Wick}}=v$, we may rewrite  (\fcolorbox{red}{ white}{\ref{N31}}) as
\begin{multline}\label{N33}
 \Vert B(v)\nabla_{v}u \Vert_{L^{2}}^{2}+\Vert \sqrt{F(v)}u \Vert_{L^{2}}^{2}-\delta\Re{e}\left(i{\xi}\cdot v^{\text{Wick}} u, Gu\right)_{L^{2}}\\=   \Re{e}\left(\mathcal{A}_{\xi}u,(1-\delta G)u\right)_{L^{2}}-\delta\Re{e}\left(\nabla_{v}\cdot (\mathbf{A}(v)\nabla_{v}u),Gu\right)_{L^{2}}+\delta\Re{e}\left(F(v)u,Gu\right)_{L^{2}}.
 \end{multline}
Using Lemma \fcolorbox{red}{ white}{\ref{N34}}, we deduce that there is a constant $c_{3}>0$ such that for all $0<\delta\leq1$ and $u\in \mathcal{S}(\mathbb{R}^{3}_{v})$, 
\begin{multline}\label{N35}
 \Vert B(v)\nabla_{v}u \Vert_{L^{2}}^{2}+\Vert \sqrt{F(v)}u \Vert_{L^{2}}^{2}-\delta\Re{e}\left(i{\xi}\cdot v^{\text{Wick}} u, Gu\right)_{L^{2}}\\\leq c_{3}\big( \Re{e}\left(\mathcal{A}_{\xi}u,u\right)_{L^{2}}+  \Re{e}\left(\mathcal{A}_{\xi}u,(1-\delta G)u\right)_{L^{2}}\big),
      \end{multline}

uniformly with respect to the parameter $\xi$ \text{in} $\mathbb{R}^{3}$. We deduce from Lemma
 \fcolorbox{red}{ white}{\ref{P5}} and (\fcolorbox{red}{ white}{\ref{P6}}) that
\begin{align}\label{N39}
  \begin{split}
 -\delta\Re{e}\left(i{\xi}\cdot v^{\text{Wick}} u, Gu\right)_{L^{2}}&= -\delta\Re{e}\left( g^{\text{Wick}}{(i{\xi}\cdot v)}^{\text{Wick}} u,u\right)_{L^{2}}\\
 &=\delta\frac{1}{4\pi}\left({\left\{{\xi}\cdot v, g\right\}}^{\text{Wick}}u,u\right)_{L^{2}}.
  \end{split}
\end{align}
  Using (\fcolorbox{red}{ white}{\ref{N36}}) with a direct computation of the Poisson bracket gives that
  \begin{multline}\label{N38}
  \left\{{\xi}\cdot v, g\right\}=(B(v)\xi \cdot B(v)\eta )(\xi\cdot\partial_{\eta}(\lambda^{-4/3}))\psi\bigg(\frac{{\vert B(v)\eta \vert}^{2}+F(v)}{\lambda^{2/3}}\bigg)\\
       +\frac{{\vert B(v)\xi  \vert}^{2}}{ \lambda^{4/3}} \psi\bigg(\frac{{\vert B(v)\eta \vert}^{2}+F(v)}{\lambda^{2/3}}\bigg)+\frac{ B(v)\xi \cdot B(v)\eta }{ \lambda^{4/3}} \xi\cdot\partial_{\eta}\left[\psi\bigg(\frac{{\vert B(v)\eta \vert}^{2}+F(v)}{\lambda^{2/3}}\bigg) \right].
       \end{multline}
  
  We notice from Lemma \fcolorbox{red}{ white}{\ref{N27}}, Lemma \fcolorbox{red}{ white}{\ref{N28}},  (\fcolorbox{red}{ white}{\ref{N36}}) and (\fcolorbox{red}{ white}{\ref{N37}}) that
  \begin{align}
      \left| \left\{{\xi}\cdot v, g\right\}-\frac{{\vert B(v)\xi  \vert}^{2}}{ \lambda^{4/3}} \psi\bigg(\frac{{\vert B(v)\eta \vert}^{2}+F(v)}{\lambda^{2/3}}\bigg)\right|\lesssim 1+{\vert B(v)\eta \vert}^{2}+F(v),
      \end{align}
uniformly with respect to the parameter $\xi$ \text{in} $\mathbb{R}^{3}$.
 It follows from (\fcolorbox{red}{ white}{\ref{N39}}),   (\fcolorbox{red}{ white}{\ref{N35}}) and the fact that
the Wick quantization is a positive quantization (\fcolorbox{red}{ white}{\ref{P7}}) that there exists constante $c_{4}>0$ such that for all $0<\delta\leq1$ and $u\in \mathcal{S}(\mathbb{R}^{3}_{v})$, 
\begin{multline}\label{N40}
 \frac{\delta}{4\pi}\left(\left[\frac{{\vert B(v)\xi  \vert}^{2}}{ \lambda^{4/3}} \psi\bigg(\frac{{\vert B(v)\eta \vert}^{2}+F(v)}{\lambda^{2/3}}\bigg)\right]^{\text{Wick}}u,u\right)_{L^{2}}+\Vert B(v)\nabla_{v}u \Vert_{L^{2}}^{2}+\Vert \sqrt{F(v)}u \Vert_{L^{2}}^{2}\\\leq c_{3}\big( \Re{e}\left(\mathcal{A}_{\xi}u,u\right)_{L^{2}}+  \Re{e}\left(\mathcal{A}_{\xi}u,(1-\delta G)u\right)_{L^{2}}\big)\hspace{7cm}\\+\delta c_{4}\left(\left[1+{\vert B(v)\eta \vert}^{2}+F(v)\right]^{\text{Wick}}u,u\right)_{L^{2}},
 \end{multline}

uniformly with respect to the parameter $\xi$ \text{in} $\mathbb{R}^{3}$. It follows from Lemma \fcolorbox{red}{ white}{\ref{N41}}  that there exists  $c_{5}>0$ such that for all $0<\delta\leq1$ and $u\in \mathcal{S}(\mathbb{R}^{3}_{v})$
\begin{multline}\label{N42}
 \delta\left(\left[\frac{{\vert B(v)\xi  \vert}^{2}}{ \lambda^{4/3}} \psi\bigg(\frac{{\vert B(v)\eta \vert}^{2}+F(v)}{\lambda^{2/3}}\bigg)\right]^{\text{Wick}}u,u\right)_{L^{2}}+ (4\pi^{2}\left[{\vert B(v)\eta \vert}^{2}+F(v)\right]^{\text{Wick}}u,u)_{L^{2}}\\+\Vert u \Vert_{L^{2}}^{2}
 \leq c_{5}\big( \Re{e}\left(\mathcal{A}_{\xi}u,u\right)_{L^{2}}+  \Re{e}\left(\mathcal{A}_{\xi}u,(1-\delta G)u\right)_{L^{2}}\big),
 \end{multline}
uniformly with respect to the parameter $\xi$ \text{in} $\mathbb{R}^{3}$.\\
Notice from (\fcolorbox{red}{ white}{\ref{N20}}), (\fcolorbox{red}{ white}{\ref{N21}}) and (\fcolorbox{red}{ white}{\ref{N501}}) that
\begin{align*}
    &\delta\frac{{\vert B(v)\xi  \vert}^{2}}{ \lambda^{4/3}} \psi\bigg(\frac{{\vert B(v)\eta \vert}^{2}+F(v)}{\lambda^{2/3}}\bigg)+4\pi^2{\vert B(v)\eta \vert}^{2}+F(v)+1\geq\\
    &\delta\frac{{\vert B(v)\xi  \vert}^{2}+{\vert B(v)\eta \vert}^{2}+F(v)+1}{ \lambda^{4/3}} \psi\bigg(\frac{{\vert B(v)\eta \vert}^{2}+F(v)}{\lambda^{2/3}}\bigg)\\
    &+({\vert B(v)\eta \vert}^{2}+F(v)+1)\left[1- \psi\bigg(\frac{{\vert B(v)\eta \vert}^{2}+F(v)}{\lambda^{2/3}}\bigg)\right]\gtrsim\\
    &\delta\frac{ \lambda^{2}}{ \lambda^{4/3}} \psi\bigg(\frac{{\vert B(v)\eta \vert}^{2}+F(v)}{\lambda^{2/3}}\bigg)+\delta\lambda^{2/3}\left[1- \psi\bigg(\frac{{\vert B(v)\eta \vert}^{2}+F(v)}{\lambda^{2/3}}\bigg)\right]\gtrsim\delta \lambda^{2/3},
\end{align*}
when  $0<\frac{\delta}{\lambda^{4/3}}\leq 1$; since 
$${\vert B(v)\eta \vert}^{2}+F(v)\geq  \lambda^{2/3},$$
on the support of the function 
$$1- \psi\bigg(\frac{{\vert B(v)\eta \vert}^{2}+F(v)}{\lambda^{2/3}}\bigg).$$
By using again that the Wick quantization is a positive quantization (\fcolorbox{red}{ white}{\ref{P7}}), we deduce that there
exists $C>0$ such that for all $0<\delta\leq1$ and $u\in \mathcal{S}(\mathbb{R}^{3}_{v})$, 
$$
((\lambda^{2/3})^{\text{Wick}}u,u)_{L^{2}}\leq C\big( \Re{e}(\mathcal{A}_{\xi}u,u)_{L^{2}}+ \Re{e}(\mathcal{A}_{\xi}u,(1-\delta G)u)_{L^{2}}\big),
$$
uniformly with respect to the parameter $\xi$ \text{in}  $\mathbb{R}^{3}$.
\end{proof}
\begin{lem}\label{P1}
There exists $C>0$ such that for all $v\in\mathbb{R}^{3}$ and $\xi\in\mathbb{R}^{3}$,
$${\langle B(v)\xi\rangle}^{1/3}\leq C m(v,\xi),$$
with $$m(v,\xi)={\big(\int_{\mathbb{R}^{3}}{\langle B(v+\tilde{v})\xi\rangle}^{2/3}\pi^{-3}e^{-\vert \tilde{v}\vert^{2}}\hspace{0.1cm}d\tilde{v}\big)}^{1/2}.$$
\end{lem}
\begin{proof}
We have that
\begin{align*}
    m(v,\xi)^{2}\gtrsim\int_{\mathbb{R}^{3}}(1+{\vert B(v+\tilde{v})\xi\vert}^{2/3})\pi^{-3}e^{-\vert \tilde{v}\vert^{2}}\hspace{0.1cm}d\tilde{v},
    \end{align*}
 therefore by using (\fcolorbox{red}{ white}{\ref{N501}}), we obtain that
    \begin{align*}
      m(v,\xi)^{2}&\gtrsim\int_{\mathbb{R}^{3}}(1+{\langle v+\tilde{v} \rangle}^{\gamma/3}\vert \xi\vert^{2/3}+ {\langle v+\tilde{v} \rangle}^{\gamma/3}\vert  (v+\tilde{v})\wedge\xi\vert^{2/3})\pi^{-3}e^{-\vert \tilde{v}\vert^{2}}\hspace{0.1cm}d\tilde{v},
\end{align*}
and using Peetre's inequality (\fcolorbox{red}{ white}{\ref{2000}}), we have 
$$\frac{\langle v \rangle^{\gamma/3}}{\langle \tilde{v}  \rangle^{\gamma/3}}\lesssim {\langle v+\tilde{v} \rangle}^{{\gamma/3}},$$ 
so we get 
\begin{align*}
     m(v,\xi)^{2}& \gtrsim 1+{\langle v \rangle}^{\gamma/3}\vert \xi\vert^{2/3}+\int_{\mathbb{R}^{3}} \frac{\langle v \rangle^{\gamma/3}}{\langle \tilde{v} \rangle^{\gamma/3}}\vert  (v+\tilde{v})\wedge\xi\vert^{2/3}\pi^{-3}e^{-\vert \tilde{v}\vert^{2}}\hspace{0.1cm}d\tilde{v}\\
     &  \gtrsim 1+{\langle v \rangle}^{\gamma/3}\vert \xi\vert^{2/3}+\int_{\overline{B(0,1)}} \frac{\langle v \rangle^{\gamma/3}}{\langle \tilde{v} \rangle^{\gamma/3}}\vert  (v+\tilde{v})\wedge\xi\vert^{2/3}\pi^{-3}e^{-\vert \tilde{v}\vert^{2}}\hspace{0.1cm}d\tilde{v}
\end{align*}
where $\overline{B(0,1)}$ stands for the closed unit ball in $\mathbb{R}^{3}$. By noticing that we have
$$\vert  (v+\tilde{v})\wedge\xi\vert\geq \vert  v\wedge\xi\vert-\vert \tilde{v}\wedge\xi\vert\geq \vert  v\wedge\xi\vert-\vert\xi\vert\geq \frac{1}{2}\vert  v\wedge\xi\vert,$$
when $\vert \tilde{v} \vert\leq 1$ and $2\vert  \xi\vert\leq \vert  v\wedge\xi\vert$, it follows that
$$\int_{\mathbb{R}^{3}} \frac{\langle v \rangle^{\gamma/3}}{\langle \tilde{v} \rangle^{\gamma/3}}\vert  (v+\tilde{v})\wedge\xi\vert^{2/3}\pi^{-3}e^{-\vert \tilde{v}\vert^{2}}\hspace{0.1cm}d\tilde{v} \gtrsim {\langle v \rangle^{\gamma/3}}\vert  v\wedge\xi\vert^{2/3},$$
when $2\vert  \xi\vert\leq \vert  v\wedge\xi\vert$. Since 
$${\langle v \rangle^{\gamma/3}}\vert  \xi\vert^{2/3} \gtrsim {\langle v \rangle^{\gamma/3}}\vert  v\wedge\xi\vert^{2/3}$$ when $ 2\vert  \xi\vert\geq \vert  v\wedge\xi\vert$. Based on the above and using  (\fcolorbox{red}{ white}{\ref{N501}}), we obtain 
$$m(v,\xi)^{2} \gtrsim 1+{\langle v \rangle}^{\gamma/3}\vert \xi\vert^{2/3}+{\langle v \rangle^{\gamma/3}}\vert  v\wedge\xi\vert^{2/3}\gtrsim 1+{\vert B(v)\xi\vert}^{2/3}\gtrsim {\langle B(v)\xi\rangle}^{2/3}.$$
\end{proof}
\begin{lem}\label{N50}
For all  $u\in \mathcal{S}(\mathbb{R}^{3}_{v})$,
\begin{align}
  \Vert {\langle v \rangle}^{\frac{5\gamma}{6}+2}u \Vert_{L^{2}}^{2}+  \Vert{\langle v \rangle}^{\frac{\gamma}{3}+1} B(v)D_{v}u \Vert_{L^{2}}^{2}\lesssim \Vert{\mathcal{A}_{\xi}}u\Vert_{L^{2}}^{2} +\Vert u \Vert_{L^{2}}^{2}.
\end{align}
\end{lem}
\begin{proof}
See Lemma 3.3 in \cite{herau:hal-00637834}.
\end{proof}
For the rest, we need to improve this estimate. We have the following Lemma. 
\begin{lem}\label{N62}
For all  $u\in \mathcal{S}(\mathbb{R}^{3}_{v})$,
\begin{align}
  \Vert {\langle v \rangle}^{\gamma+2}u \Vert_{L^{2}}^{2}+  \Vert{\langle v \rangle}^{\frac{\gamma}{2}+1} B(v)D_{v}u \Vert_{L^{2}}^{2}\lesssim \Vert{\mathcal{A}_{\xi}}u\Vert_{L^{2}}^{2} +\Vert u \Vert_{L^{2}}^{2}.
\end{align}
\end{lem}
\begin{proof}
We will start by estimating the term
$$\left(\left[ {\mathcal{A}_{\xi}},{\langle v \rangle}^{\frac{\gamma}{2}+1}\right]u,{\langle v \rangle}^{\frac{\gamma}{2}+1}u\right)_{L^{2}}.$$
Let $u\in \mathcal{S}(\mathbb{R}^{3}_{v})$, we have
\begin{align*}
    \left|\left(\left[ {\mathcal{A}_{\xi}},{\langle v \rangle}^{\frac{\gamma}{2}+1}\right]u,{\langle v \rangle}^{\frac{\gamma}{2}+1}u\right)_{L^{2}}\right|\leq   &\left|\left( B(v)\left[D_{v} ,{\langle v \rangle}^{\frac{\gamma}{2}+1}\right]u,B(v)D_{v}{\langle v \rangle}^{\frac{\gamma}{2}+1}u\right)_{L^{2}}\right|\\
    &+\left| \left( B(v)D_{v}u, B(v)\left[D_{v} ,{\langle v \rangle}^{\frac{\gamma}{2}+1}\right]{\langle v \rangle}^{\frac{\gamma}{2}+1}u\right)_{L^{2}}\right|\\\leq   &\left|\left( B(v)\left[D_{v} ,{\langle v \rangle}^{\frac{\gamma}{2}+1}\right]u, B(v)\left[D_{v} ,{\langle v \rangle}^{\frac{\gamma}{2}+1}\right]u\right)_{L^{2}}\right|\\&+ \left|\left( B(v)\left[D_{v} ,{\langle v \rangle}^{\frac{\gamma}{2}+1}\right]u, {\langle v \rangle}^{\frac{\gamma}{2}+1}B(v)D_{v} u\right)_{L^{2}}\right|\\ &+\left| \left( B(v)D_{v}u, B(v)\left[D_{v} ,{\langle v \rangle}^{\frac{\gamma}{2}+1}\right]{\langle v \rangle}^{\frac{\gamma}{2}+1}u\right)_{L^{2}}\right|,
\end{align*}
using the fact that the symbol of the operator $\left[D_{v} ,{\langle v \rangle}^{\frac{\gamma}{2}+1}\right]$ belongs to  the class
$ S({{\langle v \rangle}^{\frac{\gamma}{2}}},\Gamma)$
uniformly with respect to the parameter $\xi$ and using the following notations
$$H_{1}=\underbrace{B(v)\left[D_{v} ,{\langle v \rangle}^{\frac{\gamma}{2}+1}\right]{\langle v \rangle}^{-\gamma-1}}_{\in\mathcal{B}(L^{2})},$$
 $$H_{2}=\underbrace{{\langle v \rangle}^{-\frac{\gamma}{2}-1}B(v)\left[D_{v} ,{\langle v \rangle}^{\frac{\gamma}{2}+1}\right]{\langle v \rangle}^{-\frac{\gamma}{2}}}_{\in\mathcal{B}(L^{2})},$$
we obtain 
\begin{align*}
    \left|\left(\left[ {\mathcal{A}_{\xi}},{\langle v \rangle}^{\frac{\gamma}{2}+1}\right]u,{\langle v \rangle}^{\frac{\gamma}{2}+1}u\right)_{L^{2}}\right|\leq   &\left|\left( H_{1}{\langle v \rangle}^{\gamma+1}u,H_{1}{\langle v \rangle}^{\gamma+1}u\right)_{L^{2}}\right|\\
   &+ \left|\left(H_{1}{\langle v \rangle}^{\gamma+1}u, {\langle v \rangle}^{\frac{\gamma}{2}+1}B(v)D_{v} u\right)_{L^{2}}\right|\\ &+\left| \left( {\langle v \rangle}^{\frac{\gamma}{2}+1} B(v)D_{v}u, H_{2}{\langle v \rangle}^{\frac{\gamma}{2}+1} u\right)_{L^{2}}\right|,
\end{align*}
using the fact that $\gamma+1\leq 2+\frac{5\gamma}{6}$ and Lemma \fcolorbox{red}{ white}{\ref{N50}}, we obtain for all $\varepsilon>0$
\begin{align*}
    \left|\left(\left[ {\mathcal{A}_{\xi}},{\langle v \rangle}^{\frac{\gamma}{2}+1}\right]u,{\langle v \rangle}^{\frac{\gamma}{2}+1}u\right)_{L^{2}}\right|\lesssim \varepsilon( \Vert {\langle v \rangle}^{\gamma+2}u \Vert_{L^{2}}^{2}+  &\Vert{\langle v \rangle}^{\frac{\gamma}{2}+1} B(v)D_{v}u \Vert_{L^{2}}^{2})\\&+C_{\varepsilon}(\Vert{\mathcal{A}_{\xi}}u\Vert_{L^{2}}^{2} +\Vert u \Vert_{L^{2}}^{2}).
    \end{align*}
   On the other hand, we have
    \begin{align*}
     \left|\left({\mathcal{A}_{\xi}}{\langle v \rangle}^{\frac{\gamma}{2}+1}u,{\langle v \rangle}^{\frac{\gamma}{2}+1}u\right)_{L^{2}}\right|\leq  \left|\left({\mathcal{A}_{\xi}}u,{\langle v \rangle}^{\gamma+2}u\right)_{L^{2}}\right|+    \left|\left(\left[ {\mathcal{A}_{\xi}},{\langle v \rangle}^{\frac{\gamma}{2}+1}\right]u,{\langle v \rangle}^{\frac{\gamma}{2}+1}u\right)_{L^{2}}\right|
        \end{align*}
        and
          \begin{align*}
     \left|\left({\mathcal{A}_{\xi}}u,{\langle v \rangle}^{\gamma+2}u\right)_{L^{2}}\right|\leq \varepsilon \Vert {\langle v \rangle}^{\gamma+2}u \Vert_{L^{2}}^{2}+C_{\varepsilon}\Vert{\mathcal{A}_{\xi}}u\Vert_{L^{2}}^{2}.
        \end{align*}
        Now using (\fcolorbox{red}{ white}{\ref{N502}}), we can write
        \begin{align*}
            &\Vert {\langle v \rangle}^{\gamma+2}u \Vert_{L^{2}}^{2}+  \Vert{\langle v \rangle}^{\frac{\gamma}{2}+1} B(v)D_{v}u \Vert_{L^{2}}^{2}\\
            &\leq \Vert {\langle v \rangle}^{\frac{\gamma}{2}+1}{\langle v \rangle}^{\frac{\gamma}{2}+1}u \Vert_{L^{2}}^{2}+  \Vert B(v)D_{v}{\langle v \rangle}^{\frac{\gamma}{2}+1}u \Vert_{L^{2}}^{2}+\Vert B(v)\left[D_{v},{\langle v \rangle}^{\frac{\gamma}{2}+1}\right]u \Vert_{L^{2}}^{2}\\
            &\leq   \left|\left({\mathcal{A}_{\xi}}{\langle v \rangle}^{\frac{\gamma}{2}+1}u,{\langle v \rangle}^{\frac{\gamma}{2}+1}u\right)_{L^{2}}\right|
            +\Vert B(v)\left[D_{v},{\langle v \rangle}^{\frac{\gamma}{2}+1}\right]u \Vert_{L^{2}}^{2}.
        \end{align*}
      For the last term, we have
        \begin{align*}
        \Vert B(v)\left[D_{v},{\langle v \rangle}^{\frac{\gamma}{2}+1}\right]u \Vert_{L^{2}}^{2}\leq   \Vert H_{1}{\langle v \rangle}^{\gamma+1}u \Vert_{L^{2}}^{2}\lesssim \varepsilon \Vert {\langle v \rangle}^{\gamma+2}u \Vert_{L^{2}}^{2}+C_{\varepsilon}(\Vert{\mathcal{A}_{\xi}}u\Vert_{L^{2}}^{2}+\Vert u\Vert_{L^{2}}^{2}).
          \end{align*}
          Finally, taking $\varepsilon $ small enough, we obtain
          for all $u\in \mathcal{S}(\mathbb{R}^{3}_{v})$,
\begin{align}
  \Vert {\langle v \rangle}^{\gamma+2}u \Vert_{L^{2}}^{2}+  \Vert{\langle v \rangle}^{\frac{\gamma}{2}+1} B(v)D_{v}u \Vert_{L^{2}}^{2}\lesssim \Vert{\mathcal{A}_{\xi}}u\Vert_{L^{2}}^{2} +\Vert u \Vert_{L^{2}}^{2}.
\end{align}
\end{proof}
\begin{prop}\label{N71}
There exists $C>0$ such that for all $u\in \mathcal{S}(\mathbb{R}^{3}_{v})$,
$$
\Vert{\langle v \rangle}^{\frac{\gamma}{3}}{\vert \xi \vert}^{2/3}u\Vert_{L^{2}}^{2}\leq C\big( \Vert{\mathcal{A}_{\xi}}u\Vert_{L^{2}}^{2} +\Vert u \Vert_{L^{2}}^{2}\big),
$$
uniformly with respect to the parameter   $\xi$ \text{in} $\mathbb{R}^{3}$.
\end{prop}
\begin{proof}
We deduce from Proposition \fcolorbox{red}{ white}{\ref{N51}}, Lemma  \fcolorbox{red}{ white}{\ref{P1}} and Lemma  \fcolorbox{red}{ white}{\ref{P8}} that
$$\Vert{\langle B(v)\xi\rangle}^{1/3} u\Vert_{L^{2}}
\lesssim{\Vert{\langle B(v)\xi\rangle}^{-1/3}\mathcal{A}_{\xi}}u\Vert_{L^{2}}^{2}\Vert{\langle B(v)\xi\rangle}^{1/3} u \Vert_{L^{2}} ,
$$
uniformly with respect to the parameter   $\xi$ \text{in} $\mathbb{R}^{3}$. By
substituting ${\langle B(v)\xi\rangle}^{1/3}u$ to $u$ in this estimate, we obtain that
\begin{align}\label{N60}
\Vert{\langle B(v)\xi\rangle}^{2/3} u\Vert_{L^{2}}^{2}
\lesssim{\Vert{\langle B(v)\xi\rangle}^{-1/3}\mathcal{A}_{\xi}}{\langle B(v)\xi\rangle}^{1/3}u\Vert_{L^{2}}\Vert{\langle B(v)\xi\rangle}^{2/3} u \Vert_{L^{2}},
\end{align}
uniformly with respect to the parameter   $\xi$ \text{in} $\mathbb{R}^{3}$.
First, we have
\begin{align*}
&{\langle B(v)\xi\rangle}^{-1/3}\left[D_{v}\cdot \mathbf{A}(v)D_{v},{\langle B(v)\xi\rangle}^{1/3}\right]\\
&={\langle B(v)\xi\rangle}^{-1/3} \sum_{j,k=1}^{3}D_{v_{j}}\cdot\overline{a}_{jk}(v) \left[D_{v_{k}},{\langle B(v)\xi\rangle}^{1/3} \right]\\&\hspace{6.5cm}+{\langle B(v)\xi\rangle}^{-1/3} \sum_{j,k=1}^{3}\left[D_{v_{j}},{\langle B(v)\xi\rangle}^{1/3} \right]\cdot\overline{a}_{jk}(v) D_{v_{k}}
\\&={\langle B(v)\xi\rangle}^{-1/3} \sum_{j,k=1}^{3}D_{v_{j}}\overline{a}_{jk}(v) \left[D_{v_{k}},{\langle B(v)\xi\rangle}^{1/3} \right]\\&\hspace{6.5cm}+{\langle B(v)\xi\rangle}^{-1/3} \sum_{j,k=1}^{3}\overline{a}_{jk}(v)\left[D_{v_{k}},{\langle B(v)\xi\rangle}^{1/3} \right]D_{v_{j}}\\
&+{\langle B(v)\xi\rangle}^{-1/3} \sum_{j,k=1}^{3}\overline{a}_{jk}(v)D_{v_{j}}\big(\left[D_{v_{k}},{\langle B(v)\xi\rangle}^{1/3} \right]\big)\\
&\hspace{6.5cm}+{\langle B(v)\xi\rangle}^{-1/3} \sum_{j,k=1}^{3}\left[D_{v_{j}},{\langle B(v)\xi\rangle}^{1/3} \right]\cdot\overline{a}_{jk}(v) D_{v_{k}}.
\end{align*}
Symbolic calculus shows that $$\left[D_{v} ,{\langle B(v)\xi\rangle}^{1/3}\right]=\frac{1}{i}\nabla_{v}({\langle B(v)\xi\rangle}^{1/3}),$$
and using the following notations
$$H_{3}=\underbrace{{\langle B(v)\xi\rangle}^{-1/3} \partial_{v_{j}}\overline{a}_{jk}(v) \partial_{v_{k}}{\langle B(v)\xi\rangle}^{1/3} {\langle v \rangle}^{-\gamma-1}}_{\in\mathcal{B}(L^{2})},$$ 
$$H_{4}=\underbrace{{\langle B(v)\xi\rangle}^{-1/3} \overline{a}_{jk}(v) \partial_{v_{j}}\partial_{v_{k}}{\langle B(v)\xi\rangle}^{1/3} {\langle v \rangle}^{-\gamma-2}}_{\in\mathcal{B}(L^{2})},$$
$$H_{5}=\underbrace{{\langle B(v)\xi\rangle}^{-1/3} {b}_{ik}(v) \partial_{v_{k}}{\langle B(v)\xi\rangle}^{1/3} {\langle v \rangle}^{-\frac{\gamma}{2}-1}}_{\in\mathcal{B}(L^{2})},$$
$$H_{6}=\underbrace{{\langle B(v)\xi\rangle}^{-1/3} {b}_{ij}(v) \partial_{v_{j}}{\langle B(v)\xi\rangle}^{1/3} {\langle v \rangle}^{-\frac{\gamma}{2}-1}}_{\in\mathcal{B}(L^{2})},$$
where we used the fact that
$\mathbf{A}(v)= {B}^{\text{T}}(v){B}(v)$.
Then, going back to (\fcolorbox{red}{ white}{\ref{N60}}), we have
\begin{align*}
&\Vert{\langle B(v)\xi\rangle}^{2/3} u\Vert_{L^{2}}^{2}\\
\lesssim\hspace{0.1cm} &{\Vert{\langle B(v)\xi\rangle}^{-1/3}\mathcal{A}_{\xi}}{\langle B(v)\xi\rangle}^{1/3}u\Vert_{L^{2}}\Vert{\langle B(v)\xi\rangle}^{2/3} u \Vert_{L^{2}},\\
\lesssim \hspace{0.1cm} &\Vert{\langle B(v)\xi\rangle}^{-1/3}\left[\mathcal{A}_{\xi},{\langle B(v)\xi \rangle}^{1/3}\right]u\Vert_{L^{2}}^{2}+\Vert{\mathcal{A}_{\xi}}u \Vert_{L^{2}}^{2},\\
\lesssim \hspace{0.1cm} & \Vert{\mathcal{A}_{\xi}}u \Vert_{L^{2}}^{2}+\Vert H_{3}{\langle v \rangle}^{\gamma+1}u \Vert_{L^{2}}^{2}+\Vert H_{4}{\langle v \rangle}^{\gamma+2}u \Vert_{L^{2}}^{2}\\+&\sum_{i=1}^{3}\Vert H_{5}\big(\sum_{j=1}^{3}{\langle v \rangle}^{\frac{\gamma}{2}+1}{b}_{ij}(v)D_{v_{j}}u\big) \Vert_{L^{2}}^{2}+\sum_{i=1}^{3}\Vert H_{6}\big(\sum_{k=1}^{3}{\langle v \rangle}^{\frac{\gamma}{2}+1}{b}_{ik}(v)D_{v_{k}}u\big) \Vert_{L^{2}}^{2},\\
\lesssim \hspace{0.1cm} &\Vert{\mathcal{A}_{\xi}}u \Vert_{L^{2}}^{2}+\Vert{\langle v \rangle}^{\gamma+1}u \Vert_{L^{2}}^{2}+\Vert {\langle v \rangle}^{\gamma+2}u \Vert_{L^{2}}^{2}+\Vert {\langle v \rangle}^{\frac{\gamma}{2}+1}B(v)D_{v}u \Vert_{L^{2}}^{2}
\end{align*}
We finally conclude from Lemma \fcolorbox{red}{ white}{\ref{N62}} that for  all $u\in \mathcal{S}(\mathbb{R}^{3}_{v})$,
\begin{align*}
\Vert{\langle B(v)\xi\rangle}^{2/3} u\Vert_{L^{2}}^{2}
\lesssim \Vert{\mathcal{A}_{\xi}}u\Vert_{L^{2}}^{2} +\Vert u \Vert_{L^{2}}^{2}.
\end{align*}
This ends the proof of Proposition  \fcolorbox{red}{ white}{\ref{N71}}.
\end{proof}
 \begin{lem}
 Let  $\lambda_{K}$ be the symbol defined in (\fcolorbox{red}{ white}{\ref{N52}}). Then for any  $\Tilde{\varepsilon}>0$ there exists a constant $C_{\Tilde{\varepsilon}}$,  such that for all $u\in \mathcal{S}(\mathbb{R}^{3}_{v})$,
 \begin{align}\label{N83}
    \begin{split}
    &\Re{e}\left(\mathcal{A}_{\xi}\big(\lambda_{K}^{1/3}\big)^{w}u,\big(\lambda_{K}^{1/3}\big)^{w}u\right)_{L^{2}}+\Re{e}\left(\mathcal{A}_{\xi}\big(\lambda_{K}^{1/3}\big)^{w}u,(1-\delta G)\big(\lambda_{K}^{1/3}\big)^{w}u\right)_{L^{2}}\\&\lesssim \Tilde{\varepsilon}\Vert\big(\lambda_{K}^{2/3}\big)^{w}u \Vert_{L^{2}}^{2}
     +C_{\Tilde{\varepsilon}}\big(\Vert\mathcal{A}_{\xi}u\Vert_{L^{2}}^{2} +\Vert u \Vert_{L^{2}}^{2}\big).
     \end{split}
 \end{align}
 \end{lem}
 \begin{proof}
 As a preliminary step we firstly show that for any $\varepsilon,\Tilde{\varepsilon}>0$ there exists a constant $C_{\varepsilon,\Tilde{\varepsilon}}$, such that
   \begin{align}\label{N84}
   \begin{split}
    &\Re{e}\left(\left[ \mathcal{A}_{\xi},\big(\lambda_{K}^{1/3}\big)^{w}\right]u,a^{w}\big(\lambda_{K}^{1/3}\big)^{w}u\right)_{L^{2}}\\&\lesssim \Tilde{\varepsilon}\Vert\big(\lambda_{K}^{2/3}\big)^{w}u \Vert_{L^{2}}^{2}+C_{\varepsilon,\Tilde{\varepsilon}}\big(\Vert\mathcal{A}_{\xi}u\Vert_{L^{2}}^{2} +\Vert u \Vert_{L^{2}}^{2}\big)\\ &+\varepsilon\left\{\Re{e}\left(\mathcal{A}_{\xi}\big(\lambda_{K}^{1/3}\big)^{w}u,\big(\lambda_{K}^{1/3}\big)^{w}u\right)_{L^{2}}+\Re{e}\left(\mathcal{A}_{\xi}\big(\lambda_{K}^{1/3}\big)^{w}u,(1-\delta G)\big(\lambda_{K}^{1/3}\big)^{w}u\right)_{L^{2}}\right\} ,
     \end{split}\end{align}
  where $a$ is an arbitrary symbol belonging to $S(1,\Gamma)$ uniformly with respect to the parameter $\xi$.
   Using the notation
  $$\mathcal{Z}_{1}=\left(\left[ D_{v}\cdot\mathbf{A}(v)D_{v},\big(\lambda_{K}^{1/3}\big)^{w}\right]u, a^{w}\big(\lambda_{K}^{1/3}\big)^{w}u\right)_{L^{2}}.$$
We have
  \begin{align*}
     \mathcal{Z}_{1}&= \sum_{i,j,k=1}^{3}\left(\left[ {b}_{ik}(v)D_{v_{k}},\big(\lambda_{K}^{1/3}\big)^{w}\right]u, {b}_{ij}(v)D_{v_{j}}a^{w}\big(\lambda_{K}^{1/3}\big)^{w}u\right)_{L^{2}}\\
     &+ \sum_{i,j,k=1}^{3}\left({b}_{ik}D_{v_{k}}u, \left[ \big(\lambda_{K}^{1/3}\big)^{w},{b}_{ij}(v)D_{v_{j}}\right]a^{w}\big(\lambda_{K}^{1/3}\big)^{w}u\right)_{L^{2}}.
  \end{align*}
 Using the fact that
  $$ \left[ B(v)D_{v},\big(\lambda_{K}^{1/3}\big)^{w}\right]u=B(v)\left[ D_{v},\big(\lambda_{K}^{1/3}\big)^{w}\right]u-\left[ \big(\lambda_{K}^{1/3}\big)^{w},B(v)\right]D_{v}u,$$
 we obtain 
  \begin{align*}
  \vert \mathcal{Z}_{1} \vert &\leq 
  \left|\sum_{i,j,k=1}^{3} \left( {b}_{ik}(v)\left[D_{v_{k}},\big(\lambda_{K}^{1/3}\big)^{w}\right]u, {b}_{ij}(v)D_{v_{j}}a^{w}\big(\lambda_{K}^{1/3}\big)^{w}u\right)_{L^{2}}\right|\\
  &+ \left|\sum_{i,j,k=1}^{3} \left( \left[\big(\lambda_{K}^{1/3}\big)^{w},  {b}_{ik}(v)\right]D_{v_{k}}u, {b}_{ij}(v)D_{v_{j}}a^{w}\big(\lambda_{K}^{1/3}\big)^{w}u\right)_{L^{2}}\right|\\
  &+\left|\sum_{i,j,k=1}^{3} \left( {b}_{ik}(v)D_{v_{k}} u, {b}_{ij}(v)\left[D_{v_{j}},\big(\lambda_{K}^{1/3}\big)^{w}\right]a^{w}\big(\lambda_{K}^{1/3}\big)^{w}u\right)_{L^{2}}\right|\\
  &+\left| \sum_{i,j,k=1}^{3}\left({b}_{ik}(v)D_{v_{k}} u, \left[\big(\lambda_{K}^{1/3}\big)^{w}, {b}_{ij}(v) \right]D_{v_{j}}a^{w}\big(\lambda_{K}^{1/3}\big)^{w}u\right)_{L^{2}}\right|\\
  &\leq \vert \mathcal{Z}_{1,1} \vert + \vert \mathcal{Z}_{1,2} \vert+ \vert \mathcal{Z}_{1,3} \vert+ \vert \mathcal{Z}_{1,4} \vert.
   \end{align*}
   \underline{Estimate of  $\mathcal{Z}_{1,1}$:}
   Observing  $a\in S(1,\Gamma)$,\hspace{0.1cm} $\partial_{v}\lambda_{K}^{1/3}\in S({\langle v \rangle}^{\gamma/6}g_{2,K},\Gamma)$ and using (\fcolorbox{red}{ white}{\ref{N22}}) with symbolic calculus  shows that  
$$\left[D_{v_{j}} ,a^{w}\right]\in \Psi(1,\Gamma),\hspace{0.2cm}\left [{b}_{ij}(v),a^{w}\right] 
\in \Psi({\langle v \rangle}^{\frac{\gamma}{2}},\Gamma)\hspace{0.2cm}\text{and}\hspace{0.2cm} \left[D_{v_{k}}, \big(\lambda_{K}^{1/3}\big)^{w}\right]
\in \Psi({{\langle v \rangle}^{\frac{\gamma}{6}}}g_{2,K},\Gamma)$$
uniformly with respect to the parameter $\xi$, where $g_{2,K}^{w}$  the operator defined in (\fcolorbox{red}{ white}{\ref{N70}}).\\  Now  using the following notation
\begin{align*}
H_{7}=\underbrace{{b}_{ik}(v)\left[D_{v_{k}} ,\big(\lambda_{K}^{1/3}\big)^{w}\right]\left(g_{2,K}^{w}\right)^{-1}{\langle v \rangle}^{-\gamma/6}{\langle v \rangle}^{-\frac{\gamma}{2}-1}}_{\in\mathcal{B}(L^{2})},
\end{align*}
 we obtain
\begin{align*}
    \vert \mathcal{Z}_{1,1} \vert &\leq 
  \left| \sum_{i,j,k=1}^{3}\left(H_{7}{\langle v \rangle}^{\frac{2\gamma}{3}+1} g_{2,K}^{w}u, {b}_{ij}(v)\left[D_{v_{j}} ,a^{w}\right]{\langle v \rangle}^{-\frac{\gamma}{2}-1}{\langle v \rangle}^{\frac{\gamma}{2}+1}\big(\lambda_{K}^{1/3}\big)^{w}u\right)_{L^{2}}\right|\\
  &+ \left| \sum_{i,j,k=1}^{3}\left(H_{7}{\langle v \rangle}^{\frac{2\gamma}{3}+1}g_{2,K}^{w} u, \left[{b}_{ij}(v),a^{w}\right]{\langle v \rangle}^{-\frac{\gamma}{2}}{\langle v \rangle}^{\frac{\gamma}{2}}D_{v_{j}}\big(\lambda_{K}^{1/3}\big)^{w}u\right)_{L^{2}}\right|\\
  &+ \left|\sum_{i,j,k=1}^{3} \left(H_{7}{\langle v \rangle}^{\frac{2\gamma}{3}+1}g_{2,K}^{w} u, a^{w}{b}_{ij}(v)D_{v_{j}}\big(\lambda_{K}^{1/3}\big)^{w}u\right)_{L^{2}}\right|\\
  &\lesssim C_{\varepsilon} \left(\Vert{\langle v \rangle}^{{\gamma +2}}u \Vert_{L^{2}}^{2}+\Vert {\langle v \rangle}^{{\gamma+1}}D_{v}u \Vert_{L^{2}}^{2}+\Vert{\langle v \rangle}^{\frac{\gamma}{3}}{\vert \xi \vert}^{2/3}u\Vert_{L^{2}}^{2}\right)\\
  &+\varepsilon\Re{e}\left(\mathcal{A}_{\xi}\big(\lambda_{K}^{1/3}\big)^{w}u,\big(\lambda_{K}^{1/3}\big)^{w}u\right)_{L^{2}}.
\end{align*}
Moreover using Lemma \fcolorbox{red}{ white}{\ref{N62}} and Proposition \fcolorbox{red}{ white}{\ref{N71}}, we obtain 
\begin{align*}
        \vert \mathcal{Z}_{1,1} \vert &\lesssim \varepsilon\Re{e}\left(\mathcal{A}_{\xi}\big(\lambda_{K}^{1/3}\big)^{w}u,\big(\lambda_{K}^{1/3}\big)^{w}u\right)_{L^{2}}+C_{\varepsilon}\big(\Vert{\mathcal{A}_{\xi}}u\Vert_{L^{2}}^{2} +\Vert u \Vert_{L^{2}}^{2}\big).
\end{align*}
  \underline{Estimate of $\mathcal{Z}_{1,2}$:}
  Observing $\partial_{\eta}\lambda_{K}^{1/3}\in S({\langle v \rangle}^{\frac{\gamma}{6}+\frac{1}{3}},\Gamma)$ and using  (\fcolorbox{red}{ white}{\ref{N22}}) with symbolic calculus shows that the symbol of the commutator
  $\left[\big(\lambda_{K}^{1/3}\big)^{w},{b}_{ik}(v)\right]$\hspace{0.2cm}
 belongs to
 $S({\langle v \rangle}^{\frac{2\gamma}{3}+\frac{1}{3}},\Gamma)$ uniformly with respect to the parameter $\xi$.
 Now using the following notation
\begin{align*}
H_{8}=\underbrace{\left[\big(\lambda_{K}^{1/3}\big)^{w},{b}_{ik}(v)\right]{\langle v \rangle}^{-\frac{2\gamma}{3}-\frac{1}{3}}}_{\in\mathcal{B}(L^{2})},
\end{align*}
we obtain
\begin{align*}
    \vert \mathcal{Z}_{1,2} \vert &\leq 
 \left| \sum_{i,j,k=1}^{3}\left(H_{8}{\langle v \rangle}^{\frac{2\gamma}{3}+\frac{1}{3}}D_{v_{k}} u, {b}_{ij}(v)\left[D_{v_{j}},a^{w}\right]{\langle v \rangle}^{-\frac{\gamma}{2}-1}{\langle v \rangle}^{\frac{\gamma}{2}+1}\big(\lambda_{K}^{1/3}\big)^{w}u\right)_{L^{2}}\right|\\
  &+ \left| \sum_{i,j,k=1}^{3}\left(H_{8}{\langle v \rangle}^{\frac{2\gamma}{3}+\frac{1}{3}}D_{v_{k}} u, \left[{b}_{ij}(v),a^{w}\right]{\langle v \rangle}^{-\frac{\gamma}{2}}{\langle v \rangle}^{\frac{\gamma}{2}}D_{v_{j}}\big(\lambda_{K}^{1/3}\big)^{w}u\right)_{L^{2}}\right|\\
  &+ \left|\sum_{i,j,k=1}^{3} \left(H_{8}{\langle v \rangle}^{\frac{2\gamma}{3}+\frac{1}{3}} D_{v_{k}}u, a^{w}{b}_{ij}(v)D_{v_{j}}\big(\lambda_{K}^{1/3}\big)^{w}u\right)_{L^{2}}\right|\\
  &\lesssim C_{\varepsilon} \Vert {\langle v \rangle}^{{\gamma+1}}D_{v}u \Vert_{L^{2}}^{2}+\varepsilon\Re{e}\left(\mathcal{A}_{\xi}\big(\lambda_{K}^{1/3}\big)^{w}u,\big(\lambda_{K}^{1/3}\big)^{w}u\right)_{L^{2}},
\end{align*}
moreover using  Proposition \fcolorbox{red}{ white}{\ref{N71}}, we obtain
\begin{align*}
        \vert \mathcal{Z}_{1,2} \vert &\lesssim \varepsilon\Re{e}\left(\mathcal{A}_{\xi}\big(\lambda_{K}^{1/3}\big)^{w}u,\big(\lambda_{K}^{1/3}\big)^{w}u\right)_{L^{2}}+C_{\varepsilon}\big(\Vert\mathcal{A}_{\xi}u\Vert_{L^{2}}^{2} +\Vert u \Vert_{L^{2}}^{2}\big).
\end{align*}
 \underline{Estimate of  $\mathcal{Z}_{1,3}$:}
 Using the following notation
\begin{align*}
H_{9}=\underbrace{{\langle v \rangle}^{-\frac{\gamma}{2}-1}{b}_{ij}(v)\left[D_{v_{j}} ,\big(\lambda_{K}^{1/3}\big)^{w}\right]a^{w}\left(g_{2,K}^{w}\right)^{-1}{\langle v \rangle}^{-\gamma/6}}_{\in\mathcal{B}(L^{2})},
\end{align*}
we obtain 
\begin{align*}
    \vert \mathcal{Z}_{1,3} \vert &\leq
    \left|\sum_{i,j,k=1}^{3} \left({\langle v \rangle}^{\frac{\gamma}{2}+1} {b}_{ik}D_{v_{k}}u, H_{9}{\langle v \rangle}^{\frac{\gamma}{6}}g_{2,K}^{w}\big(\lambda_{K}^{1/3}\big)^{w}u\right)_{L^{2}}\right|\\
   &\lesssim\left|\sum_{i,j,k=1}^{3} \left({\langle v \rangle}^{\frac{\gamma}{2}+1} {b}_{ik}D_{v_{k}}u, {\langle v \rangle}^{\frac{\gamma}{6}}g_{2,K}^{w}\big(\lambda_{K}^{1/3}\big)^{w}u\right)_{L^{2}}\right|\\
  &\lesssim \varepsilon \left(\Vert{\langle v \rangle}^{{\frac{\gamma}{6} }}\big(\lambda_{K}^{1/3}\big)^{w}u \Vert_{L^{2}}^{2}+\Vert {\langle v \rangle}^{{\frac{\gamma}{6} }}D_{v}\big(\lambda_{K}^{1/3}\big)^{w}u \Vert_{L^{2}}^{2}+\Vert{\langle v \rangle}^{\frac{\gamma}{6}}{\vert \xi \vert}^{1/3}\big(\lambda_{K}^{1/3}\big)^{w}u\Vert_{L^{2}}^{2}\right)\\
  &+C_{\varepsilon} \Vert{\langle v \rangle}^{\frac{\gamma}{2}+1} B(v)D_{v}u \Vert_{L^{2}}^{2},
\end{align*}
moreover using Lemma \fcolorbox{red}{ white}{\ref{N62}} and Proposition \fcolorbox{red}{ white}{\ref{N51}}, we obtain 
\begin{align*}
        \vert \mathcal{Z}_{1,3} \vert &\lesssim \varepsilon\Re{e}\left(\mathcal{A}_{\xi}\big(\lambda_{K}^{1/3}\big)^{w}u,\big(\lambda_{K}^{1/3}\big)^{w}u\right)_{L^{2}}+\varepsilon\Re{e}\left(\mathcal{A}_{\xi}(\lambda_{K}^{1/3})^{w}u,(1-\delta G)(\lambda_{K}^{1/3})^{w}u\right)_{L^{2}}\\
        &+C_{\varepsilon}\big(\Vert\mathcal{A}_{\xi}u\Vert_{L^{2}}^{2} +\Vert u \Vert_{L^{2}}^{2}\big).
\end{align*}
 \underline{Estimate of $\mathcal{Z}_{1,4}$:}
Using the following notations
 \begin{align*}
H_{10}=\underbrace{{\langle v \rangle}^{-\frac{\gamma}{3}-1}\left[\big(\lambda_{K}^{1/3}\big)^{w},{b}_{ij}(v)\right]\left[D_{v_{j}} ,a^{w}\right]{\langle v \rangle}^{-\frac{\gamma}{3}}}_{\in\mathcal{B}(L^{2})},
\end{align*}
 \begin{align*}
H_{11}=\underbrace{{\langle v \rangle}^{-\frac{\gamma}{3}-1}\left[\big(\lambda_{K}^{1/3}\big)^{w},{b}_{ij}(v)\right]a^{w}{\langle v \rangle}^{-\frac{\gamma}{3}}}_{\in\mathcal{B}(L^{2})},
\end{align*}
we obtain
\begin{align*}
    \vert \mathcal{Z}_{1,4} \vert &\leq
    \left|\sum_{i,j,k=1}^{3} \left({\langle v \rangle}^{\frac{\gamma}{3}+1} {b}_{ik}D_{v_{k}}u, H_{10}{\langle v \rangle}^{\frac{\gamma}{3}}\big(\lambda_{K}^{1/3}\big)^{w}u\right)_{L^{2}}\right|\\
    &+   \left|\sum_{i,j,k=1}^{3} \left({\langle v \rangle}^{\frac{\gamma}{3}+1} {b}_{ik}D_{v_{k}}u, H_{11}{\langle v \rangle}^{\frac{\gamma}{3}}D_{v_{j}}\big(\lambda_{K}^{1/3}\big)^{w}u\right)_{L^{2}}\right|\\
  &\lesssim \varepsilon \left(\Vert{\langle v \rangle}^{{\frac{\gamma}{3} }}\big(\lambda_{K}^{1/3}\big)^{w}u \Vert_{L^{2}}^{2}+\Vert {\langle v \rangle}^{{\frac{\gamma}{3} }}D_{v}\big(\lambda_{K}^{1/3}\big)^{w}u \Vert_{L^{2}}\right)\\
  &+C_{\varepsilon} \Vert{\langle v \rangle}^{\frac{\gamma}{3}+1} B(v)D_{v}u \Vert_{L^{2}}^{2},
\end{align*}
Moreover using Lemma \fcolorbox{red}{ white}{\ref{N50}}, we obtain
\begin{align*}
        \vert \mathcal{Z}_{1,4} \vert &\lesssim \varepsilon\Re{e}\left(\mathcal{A}_{\xi}\big(\lambda_{K}^{1/3}\big)^{w}u,\big(\lambda_{K}^{1/3}\big)^{w}u\right)_{L^{2}}+C_{\varepsilon}\big(\Vert\mathcal{A}_{\xi}u\Vert_{L^{2}}^{2} +\Vert u \Vert_{L^{2}}^{2}\big),
\end{align*}
so using the estimates of $\mathcal{Z}_{1,n}$ for $n=1,\ldots,4, $ we obtain
\begin{align*}
        \vert \mathcal{Z}_{1} \vert &\lesssim \varepsilon\Re{e}\left(\mathcal{A}_{\xi}\big(\lambda_{K}^{1/3}\big)^{w}u,\big(\lambda_{K}^{1/3}\big)^{w}u\right)_{L^{2}}+C_{\varepsilon}\big(\Vert\mathcal{A}_{\xi}u\Vert_{L^{2}}^{2} +\Vert u \Vert_{L^{2}}^{2}\big).
\end{align*}
Let's look now
  $$\mathcal{Z}_{2}=\left(\left[ F(v),\big(\lambda_{K}^{1/3}\big)^{w}\right]u, a^{w}\big(\lambda_{K}^{1/3}\big)^{w}u\right)_{L^{2}},$$
   observing the symbol $\partial_{\eta}\lambda_{K}^{1/3}\in S({\langle v \rangle}^{\frac{\gamma}{6}+\frac{1}{3}},\Gamma)$, and using  Lemma \fcolorbox{red}{ white}{\ref{N80}} with symbolic calculus shows that the symbol of the commutator 
$\left[ F(v),\big(\lambda_{K}^{1/3}\big)^{w}\right]$\hspace{0.2cm}
 belongs to $S({\langle v \rangle}^{\frac{3\gamma}{2}+2},\Gamma)$ uniformly with respect to the parameter $\xi$. Now using the following notation
  $$H_{12}=\underbrace{{\langle v \rangle}^{-\frac{\gamma}{2}-1}{\left[F(v) ,\big(\lambda_{K}^{1/3}\big)^{w}\right]} {\langle v \rangle}^{-{\gamma}-1}}_{\in\mathcal{B}(L^{2})},$$
  we obtain
      \begin{align*}
    \vert\mathcal{Z}_{2}\vert
    &\leq\left|\left({\langle v \rangle}^{\frac{\gamma}{2}+1}H_{12}{\langle v \rangle}^{-\gamma-1}{\langle v \rangle}^{\gamma+1} u, a^{w}\big(\lambda_{K}^{1/3}\big)^{w}u\right)_{L^{2}}\right|\\
    &\leq\left|\left(H_{12}{\langle v \rangle}^{\gamma+1} u, a^{w}{\langle v \rangle}^{\frac{\gamma}{2}+1}\big(\lambda_{K}^{1/3}\big)^{w}u\right)_{L^{2}}\right|\\&+ \left|\left(H_{12}{\langle v \rangle}^{\gamma+1} u, \left[{\langle v \rangle}^{\frac{\gamma}{2}+1} ,a^{w}\right]{\langle v \rangle}^{-\frac{\gamma}{2}-1}{\langle v \rangle}^{\frac{\gamma}{2}+1}\big(\lambda_{K}^{1/3}\big)^{w}u\right)_{L^{2}}\right|,
    \end{align*}
    in addition, we have that the symbol of the commutator  $ \left[{\langle v \rangle}^{\frac{\gamma}{2}+1} ,a^{w}\right]$ belongs to $S({{\langle v \rangle}^{\frac{\gamma}{2}+1}},\Gamma)$
    and using  Lemma \fcolorbox{red}{ white}{\ref{N62}}, we obtain 
      \begin{align}
    \vert\mathcal{Z}_{2}\vert
    \lesssim \varepsilon\Re{e}\left(\mathcal{A}_{\xi}\big(\lambda_{K}^{1/3}\big)^{w}u,\big(\lambda_{K}^{1/3}\big)^{w}u\right)_{L^{2}}+C_{\varepsilon}(\Vert \mathcal{A}_{\xi}u\Vert_{L^{2}}^{2} +\Vert u \Vert_{L^{2}}^{2}).
    \end{align}
      Let’s look now
    $$\mathcal{Z}_{3}=\left(\left[ iv\cdot \xi,\big(\lambda_{K}^{1/3}\big)^{w}\right]u, a^{w}\big(\lambda_{K}^{1/3}\big)^{w}u\right)_{L^{2}},$$
    using Lemma \fcolorbox{red}{ white}{\ref{N27}} with symbolic calculus shows that the symbol of the commutator
   $\left[ iv\cdot \xi,\big(\lambda_{K}^{1/3}\big)^{w}\right]$ belongs to $S(\lambda_{K}^{1/3},\Gamma)$  uniformly with respect to the parameter $\xi$.
    Now using the following notation
 $$H_{13}=\underbrace{\big(\lambda_{K}^{1/3}\big)^{w}(a^{w})^{*}{\left[iv\cdot \xi ,\big(\lambda_{K}^{1/3}\big)^{w}\right]} \left(\big(\lambda_{K}^{2/3}\big)^{w}\right)^{-1}}_{\in\mathcal{B}(L^{2})},$$
 we obtain 
  \begin{align}
  \vert\mathcal{Z}_{3}\vert &\leq\left|\left(H_{13}\big(\lambda_{K}^{2/3}\big)^{w}u,u\right)_{L^{2}}\right|\lesssim \Tilde{\varepsilon}\Vert\big(\lambda_{K}^{2/3}\big)^{w}u \Vert_{L^{2}}^{2}
     +C_{\Tilde{\varepsilon}}\big(\Vert\mathcal{A}_{\xi}u\Vert_{L^{2}}^{2} +\Vert u \Vert_{L^{2}}^{2}\big).
      \end{align}
        From the above,  using the estimates of $\mathcal{Z}_{l}$ for $l=1,\ldots,3$, we obtain (\fcolorbox{red}{ white}{\ref{N84}}). 
Next we prove  (\fcolorbox{red}{ white}{\ref{N83}}), we have the following relation
\begin{align*}
  &\Re{e}\left(\mathcal{A}_{\xi}\big(\lambda_{K}^{1/3}\big)^{w}u,\big(\lambda_{K}^{1/3}\big)^{w}u\right)_{L^{2}}+\Re{e}\left(\mathcal{A}_{\xi}\big(\lambda_{K}^{1/3}\big)^{w}u,(1-\delta G)\big(\lambda_{K}^{1/3}\big)^{w}u\right)_{L^{2}}\\
=\hspace{0.1cm}&\Re{e}\left(\mathcal{A}_{\xi}u,\big(\lambda_{K}^{1/3}\big)^{w}\big(\lambda_{K}^{1/3}\big)^{w}u\right)_{L^{2}}+  \Re{e}\left(\left[ \mathcal{A}_{\xi},\big(\lambda_{K}^{1/3}\big)^{w}\right]u,\big(\lambda_{K}^{1/3}\big)^{w}u\right)_{L^{2}}\\
+\hspace{0.1cm}&\Re{e}\left(\mathcal{A}_{\xi}u,\big(\lambda_{K}^{1/3}\big)^{w}(1-\delta G)\big(\lambda_{K}^{1/3}\big)^{w}u\right)_{L^{2}}+ \Re{e}\left(\left[ \mathcal{A}_{\xi},\big(\lambda_{K}^{1/3}\big)^{w}\right]u,(1-\delta G)\big(\lambda_{K}^{1/3}\big)^{w}u\right)_{L^{2}}\\
=\hspace{0.1cm}&\Re{e}\left(\mathcal{A}_{\xi}u,\big(\lambda_{K}^{1/3}\big)^{w}\big(\lambda_{K}^{1/3}\big)^{w}\left(\big(\lambda_{K}^{2/3}\big)^{w}\right)^{-1}\big(\lambda_{K}^{2/3}\big)^{w}u\right)_{L^{2}}+\Re{e}\left(\left[ \mathcal{A}_{\xi},\big(\lambda_{K}^{1/3}\big)^{w}\right]u,\big(\lambda_{K}^{1/3}\big)^{w}u\right)_{L^{2}}\\
+\hspace{0.1cm}&\Re{e}\left(\mathcal{A}_{\xi}u,\big(\lambda_{K}^{1/3}\big)^{w}(1-\delta G)\big(\lambda_{K}^{1/3}\big)^{w}\left(\big(\lambda_{K}^{2/3}\big)^{w}\right)^{-1}\big(\lambda_{K}^{2/3}\big)^{w}u\right)_{L^{2}}\\
+\hspace{0.1cm}&\Re{e}\left(\left[ \mathcal{A}_{\xi},\big(\lambda_{K}^{1/3}\big)^{w}\right]u,(1-\delta G)\big(\lambda_{K}^{1/3}\big)^{w}u\right)_{L^{2}}
 \end{align*}
 gives, with  $\Tilde{\varepsilon}>0$ arbitrary,
 \begin{align*}
 &\Re{e}\left(\mathcal{A}_{\xi}\big(\lambda_{K}^{1/3}\big)^{w}u,\big(\lambda_{K}^{1/3}\big)^{w}u\right)_{L^{2}}+\Re{e}\left(\mathcal{A}_{\xi}\big(\lambda_{K}^{1/3}\big)^{w}u,(1-\delta G)\big(\lambda_{K}^{1/3}\big)^{w}u\right)_{L^{2}}\\
&\lesssim \Tilde{\varepsilon}\Vert\big(\lambda_{K}^{2/3}\big)^{w}u \Vert_{L^{2}}^{2}
     +C_{\Tilde{\varepsilon}}\big(\Vert\mathcal{A}_{\xi}u\Vert_{L^{2}}^{2} +\Vert u \Vert_{L^{2}}^{2}\big)+  \Re{e}\left(\left[ \mathcal{A}_{\xi},\big(\lambda_{K}^{1/3}\big)^{w}\right]u,\big(\lambda_{K}^{1/3}\big)^{w}u\right)_{L^{2}}\\&+\Re{e}\left(\left[ \mathcal{A}_{\xi},\big(\lambda_{K}^{1/3}\big)^{w}\right]u,(1-\delta G)\big(\lambda_{K}^{1/3}\big)^{w}u\right)_{L^{2}}.
  \end{align*}
 We could apply (\fcolorbox{red}{ white}{\ref{N84}}) with $a=1$ and $a=1-\delta\Tilde{g}$  to control the last term in the above inequality; this gives,
with   $\varepsilon,\Tilde{\varepsilon}>0$ arbitrarily small,
   \begin{align*}
   &\Re{e}\left(\mathcal{A}_{\xi}\big(\lambda_{K}^{1/3}\big)^{w}u,\big(\lambda_{K}^{1/3}\big)^{w}u\right)_{L^{2}}+\Re{e}\left(\mathcal{A}_{\xi}\big(\lambda_{K}^{1/3}\big)^{w}u,(1-\delta G)\big(\lambda_{K}^{1/3}\big)^{w}u\right)_{L^{2}}\\
   &\lesssim \Tilde{\varepsilon}\Vert\big(\lambda_{K}^{2/3}\big)^{w}u \Vert_{L^{2}}^{2}+ \varepsilon\Re{e}\left(\mathcal{A}_{\xi}\big(\lambda_{K}^{1/3}\big)^{w}u,\big(\lambda_{K}^{1/3}\big)^{w}u\right)_{L^{2}}+C_{\varepsilon,\Tilde{\varepsilon}}(\Vert\mathcal{A}_{\xi}u\Vert_{L^{2}}^{2} +\Vert u \Vert_{L^{2}}^{2})\\&+\varepsilon\Re{e}\left(\mathcal{A}_{\xi}\big(\lambda_{K}^{1/3}\big)^{w}u,(1-\delta G)\big(\lambda_{K}^{1/3}\big)^{w}u\right)_{L^{2}}.
     \end{align*}
     Letting $\varepsilon$ small enough yields the desired estimate  (\fcolorbox{red}{ white}{\ref{N83}}).
\end{proof}
\begin{prop}\label{N100}
 Let  $\lambda_{K}$ be the symbol defined in   (\fcolorbox{red}{ white}{\ref{N52}}). Then there exists $C_{0}>0$ such that for all $u\in \mathcal{S}(\mathbb{R}^{3}_{v})$,
 \begin{align}
    \Vert\big(\lambda_{K}^{2/3}\big)^{w}u \Vert_{L^{2}}^{2}\leq 
     C_{0}\big(\Vert\mathcal{A}_{\xi}u\Vert_{L^{2}}^{2} +\Vert u \Vert_{L^{2}}^{2}\big).
 \end{align}
\end{prop}
\begin{proof}
 
   Using  Proposition \fcolorbox{red}{ white}{\ref{N51}}, we have for all $u\in \mathcal{S}(\mathbb{R}^{3}_{v})$,
    \begin{align}
((\lambda_{K}^{2/3})^{\text{Wick}}u,u)_{L^{2}}\lesssim  \Re{e}(\mathcal{A}_{\xi}u,u)_{L^{2}}+ \Re{e}(\mathcal{A}_{\xi}u,(1-\delta G)u)_{L^{2}},
\end{align}
uniformly with respect to the parameter   $\xi$ \text{in} $\mathbb{R}^{3}$.
   By substituting $\left(\lambda_{K}^{1/3}\right)^{w}u$ to $u$ in the  above estimate, we obtain that for all $\Tilde{\varepsilon}>0,$
     \begin{align}\label{M102}
        \left(\big(\lambda_{K}^{2/3}\big)^{\text {Wick}}\big(\lambda_{K}^{1/3}\big)^{w}u,\big(\lambda_{K}^{1/3}\big)^{w}u\right)_{L^{2}}\lesssim \Tilde{\varepsilon}\Vert \big(\lambda_{K}^{2/3}\big)^{w}u \Vert_{L^{2}}^{2}
     +C_{\Tilde{\varepsilon}}\big(\Vert\mathcal{A}_{\xi}u\Vert_{L^{2}}^{2} +\Vert u \Vert_{L^{2}}^{2}\big) .
    \end{align}
   Notice from (\fcolorbox{red}{ white}{\ref{P3}}) that we may write
   \begin{align}
(\lambda_{K}^{2/3})^{\text {Wick}}=(\lambda_{K}^{2/3})^{w}+r^{w},
\end{align}
   with
   $$r(v,\eta)=\int_{0}^{1}\int_{\mathbb{R}^{6}} (1-\theta)(\lambda_{K}^{2/3})^{"}(Y+\theta Y_{1}) Y_{1}\cdot Y_{1} e^{-\vert Y_{1}\vert^{2}}\mathrm{d}Y_{1} \, \mathrm{d}\theta,$$
   where $Y,Y_{1}\in \mathbb{R}^{6}$ and $(\lambda_{K}^{2/3})^{"}(Y)$ is the Hessian of $\lambda_{K}^{2/3}$ at the point $Y$.
  Define
  $$r_{1}=\pi^{-3}\int_{0}^{1}\int_{\mathbb{R}^{6}} (1-\theta)\nabla^{2}_{\eta}(\lambda_{K}^{2/3})(Y+\theta Y_{1})\eta_{1}\cdot\eta_{1} e^{-\vert Y_{1}\vert^{2}}\mathrm{d}Y_{1} \, \mathrm{d}\theta,$$
   $$r_{2}=\pi^{-3}\sum_{j,k=1}^{3}\int_{0}^{1}\int_{\mathbb{R}^{6}} (1-\theta)\partial_{v_{j}}\partial_{\eta_{k}}(\lambda_{K}^{2/3})(Y+\theta Y_{1})({v_{1}}_{j}{\eta_{1}}_{k}+{v_{1}}_{k}{\eta_{1}}_{j}) e^{-\vert Y_{1}\vert^{2}}\mathrm{d}Y_{1} \, \mathrm{d}\theta,$$
 and
    $$r_{3}=\pi^{-3}\int_{0}^{1}\int_{\mathbb{R}^{6}} (1-\theta)\nabla^{2}_{v}(\lambda_{K}^{2/3})(Y+\theta Y_{1})v_{1}\cdot v_{1} e^{-\vert Y_{1}\vert^{2}}\mathrm{d}Y_{1} \, \mathrm{d}\theta.$$
    Using Lemma \fcolorbox{red}{ white}{\ref{M80}} with symbolic calculus shows that the symbol
$$\nabla^{2}_{\eta}(\lambda_{K}^{2/3}), \nabla_{v}\nabla_{\eta}(\lambda_{K}^{2/3})
\in S({{\langle v \rangle}^{\gamma+2}},\Gamma),$$
uniformly with respect to the parameter $\xi$, then
$r_{1}, r_{2}$
belong to   $S({{\langle v \rangle}^{\gamma+2}},\Gamma)$.
Using the following notations
 $$H_{14}=\underbrace{\big(\big(\lambda_{K}^{2/3}\big)^{w}\big)^{-1}\big(\lambda_{K}^{1/3}\big)^{w}r^{w}_{1}\big(\lambda_{K}^{1/3}\big)^{w}{\langle v \rangle}^{-{\gamma}-2}}_{\in\mathcal{B}(L^{2})},$$ $$H_{15}=\underbrace{\big(\big(\lambda_{K}^{2/3}\big)^{w}\big)^{-1}\big(\lambda_{K}^{1/3}\big)^{w}r^{w}_{1}\big(\lambda_{K}^{1/3}\big)^{w}{\langle v \rangle}^{-{\gamma}-2}}_{\in\mathcal{B}(L^{2})},$$
 we obtain for all  $\Tilde{\Tilde{\varepsilon}}>0$,
 \begin{align*}
     &\left|\left(r_{1}^{w}\big(\lambda_{K}^{1/3}\big)^{w}u,\big(\lambda_{K}^{1/3}\big)^{w}u\right)_{L^{2}}\right|+ \left|\left(r_{2}^{w}\big(\lambda_{K}^{1/3}\big)^{w}u,\big(\lambda_{K}^{1/3}\big)^{w}u\right)_{L^{2}}\right|\\
    &\leq\left|\left(H_{14}{\langle v \rangle}^{{\gamma}+2}u,\big(\lambda_{K}^{2/3}\big)^{w}u\right)_{L^{2}}\right|+\left|\left(H_{15}{\langle v \rangle}^{{\gamma}+2}u,\big(\lambda_{K}^{2/3}\big)^{w}u\right)_{L^{2}}\right|\\
    &\lesssim \Tilde{\Tilde{\varepsilon}}\Vert\big(\lambda_{K}^{2/3}\big)^{w}u \Vert_{L^{2}}^{2}
     +C_{\Tilde{\Tilde{\varepsilon}}}\big(\Vert\mathcal{A}_{\xi}u\Vert_{L^{2}}^{2} +\Vert u \Vert_{L^{2}}^{2}\big).
\end{align*}
 Taking into account that the symbol
  $\nabla^{2}_{v}(\lambda_{K}^{2/3})$,
  belongs to $S({{\langle v \rangle}^{\frac{\gamma}{3}}}g_{3,K},\Gamma)$
  uniformly with respect to the parameter $\xi$, where $g_{3,K}$ the operator defined in (\fcolorbox{red}{ white}{\ref{N91}}), then the symbol  $r_{3}$ belongs to   $S({{\langle v \rangle}^{\frac{\gamma}{3}}}g_{3,K},\Gamma).$ Using the following notation
 $$H_{16}=\underbrace{\big(\big(\lambda_{K}^{2/3}\big)^{w}\big)^{-1}\big(\lambda_{K}^{1/3}\big)^{w}r^{w}_{3}\big(\lambda_{K}^{1/3}\big)^{w}\big(g_{3,K}^{w}\big)^{-1}{\langle v \rangle}^{-\frac{\gamma}{3}}}_{\in\mathcal{B}(L^{2})},$$
 we obtain
 \begin{align*}
     \left|\left(r_{3}^{w}\big(\lambda_{K}^{1/3}\big)^{w}u,\big(\lambda_{K}^{1/3}\big)^{w}u\right)_{L^{2}}\right|
    &\leq\left|\left(H_{16}{\langle v \rangle}^{\frac{\gamma}{3}}g_{3,K}^{w}u,\big(\lambda_{K}^{2/3}\big)^{w}u\right)_{L^{2}}\right|\\
    &\lesssim \Tilde{\Tilde{\varepsilon}}\Vert\big(\lambda_{K}^{2/3}\big)^{w}u \Vert_{L^{2}}^{2}
     +C_{\Tilde{\Tilde{\varepsilon}}}\big(\Vert\mathcal{A}_{\xi}u\Vert_{L^{2}}^{2} +\Vert u \Vert_{L^{2}}^{2}\big).
\end{align*}
From the above, using the estimates of $r_{l}$ for $l=1,\ldots,3$, we obtain  
\begin{align}\label{N95}
     \left|\left(r^{w}\big(\lambda_{K}^{1/3}\big)^{w}u,\big(\lambda_{K}^{1/3}\big)^{w}u\right)_{L^{2}}\right|
    \lesssim  \Tilde{\Tilde{\varepsilon}}\Vert\big(\lambda_{K}^{2/3}\big)^{w}u \Vert_{L^{2}}^{2}
     +C_{\Tilde{\Tilde{\varepsilon}}}\big(\Vert\mathcal{A}_{\xi}u\Vert_{L^{2}}^{2} +\Vert u \Vert_{L^{2}}^{2}\big).
\end{align}
By applying Theorem  \fcolorbox{red}{ white}{\ref{M36}} with $p=\lambda$, we obtain  
\begin{align}
\left(\big(\lambda_{K}^{2/3}\big)^{ w}\big(\lambda_{K}^{1/3}\big)^{w}u,\big(\lambda_{K}^{1/3}\big)^{w}u\right)_{L^{2}}\sim \Vert{\big[ \big(\lambda_{K}^{1/3}\big)^{w}\big]}^{2}u \Vert_{L^{2}}^{2}\sim\Vert\big(\lambda_{K}^{2/3}\big)^{w}u \Vert_{L^{2}}^{2},
\end{align}
uniformly with respect to the parameter $\xi$.
 Using (\fcolorbox{red}{ white}{\ref{N95}}), then taking $\Tilde{\Tilde{\varepsilon}}$ small enough, we get for all $\Tilde{\varepsilon}>0$,  
\begin{align*}
    \Vert\big(\lambda_{K}^{2/3}\big)^{w}u \Vert_{L^{2}}^{2}\lesssim \Tilde{\varepsilon}\Vert\big(\lambda_{K}^{2/3}\big)^{w}u \Vert_{L^{2}}^{2}
     +\Tilde{C}_{\Tilde{\varepsilon}}\big(\Vert\mathcal{A}_{\xi}u\Vert_{L^{2}}^{2} +\Vert u \Vert_{L^{2}}^{2}\big),
\end{align*}
now taking $\Tilde{\varepsilon}$ small enough, we obtain  that there is a constant $C_{0}>0$
  such that for all $u\in \mathcal{S}(\mathbb{R}^{3}_{v})$,
 \begin{align}
    \Vert\big(\lambda_{K}^{2/3}\big)^{w}u \Vert_{L^{2}}^{2}\leq 
     C_{0}\big(\Vert\mathcal{A}_{\xi}u\Vert_{L^{2}}^{2} +\Vert u \Vert_{L^{2}}^{2}\big).
 \end{align}
 \end{proof}
 \begin{prop}\label{N105}
   Let  $\lambda_{K}$ be the symbol defined in   (\fcolorbox{red}{ white}{\ref{N52}}). Then there exists $\mathbf{C}_{0} >0$ such that for all $u\in \mathcal{S}(\mathbb{R}^{3}_{v})$,
    \begin{align}
\Vert \mathcal{A}_{\xi}u \Vert_{L^{2}}^{2} \leq \mathbf{C}_{0} \Vert \big({\lambda^{2}_{K}}\big)^{w}u \Vert_{L^{2}}^{2}.
 \end{align}
 \end{prop}
 \begin{proof}
We denote by $\Tilde{\sigma}$  the symbol of the operator $\mathcal{A}_{\xi}$.  We will show that $\Tilde{\sigma}$ belongs to $ S(\lambda^{2},\Gamma)$ uniformly with respect to the parameter $\xi$. \\Using Lemma 3.11 in \cite{herau_anisotropic_2011}, we can write $\Tilde{\sigma}$ as follows
  $$\Tilde{\sigma}=iv\cdot\xi + {\vert B(v)\eta \vert}^{2}+F(v)+R_{1}+R_{2},$$
where $R_{1}$ (resp $R_{2}$) is a symbol belongs to $ S( {\langle v \rangle}^{\gamma+1} \langle\eta \rangle,\Gamma)$ (resp  $S({\langle v \rangle}^{\gamma},\Gamma)$) uniformly with respect to the parameter $\xi$.\\ 
   \underline{\text{For}\hspace{0.2cm}$iv\cdot\xi$:}
  Taking into account the fact that  $\gamma \geq 0$,  we have
 \begin{align*}
 \text{For}\hspace{0.2cm} \vert \alpha \vert=  0,\hspace{0.2cm} &\vert iv\cdot\xi \vert \lesssim  \vert v\vert^{2}+\vert \xi \vert^{2}
    \lesssim  {\langle v \rangle}^{\gamma+2}+ {\langle v \rangle}^{\gamma}\vert \xi \vert^{2}\lesssim  {\lambda}^{2},\\
\text{for}\hspace{0.2cm} \vert \alpha \vert=  1,\hspace{0.2cm} &\vert \partial_{v}^{ \alpha  }(iv\cdot\xi) \vert=\vert \xi \vert\lesssim  {\langle v \rangle}^{\gamma/2}\vert \xi \vert \lesssim {\lambda}^{2},\\
 \text{for}\hspace{0.2cm}  \vert \alpha \vert  \geq 2,\hspace{0.2cm}  &\vert \partial_{v}^{\alpha }(iv\cdot\xi) \vert=0.
  \end{align*}
   \underline{\text{For}\hspace{0.2cm}${\vert B(v)\eta \vert}^{2}$:}
  Using   (\fcolorbox{red}{ white}{\ref{N22}}), we have
    \begin{align}
          \text{for}\hspace{0.2cm} \vert \alpha \vert\geq 1,\hspace{0.2cm}  {\vert \partial_{v}^{ \alpha } B(v)\eta \vert}^{2}
  \lesssim {\langle v \rangle}^{\gamma}\vert \eta \vert^{2}\lesssim {\vert B(v)\eta \vert}^{2},
   \end{align}
then, using Cauchy-Schwarz we get 
   \begin{align}
        &\forall \alpha\in \mathbb{N}^{3},\hspace{0.2cm}  \vert \partial_{v}^{ \alpha }({\vert B(v)\eta \vert}^{2})\vert \lesssim {\vert B(v)\eta \vert}^{2}
  \lesssim {\lambda}^{2}.
   \end{align}
 On the other hand, also using  (\fcolorbox{red}{ white}{\ref{N22}}), we have
   \begin{align*}
  \text{pour}\hspace{0.2cm} \vert \alpha \vert=  0,\hspace{0.2cm}   &{\vert B(v)\eta \vert}^{2}  \lesssim   {\lambda}^{2},\\
\text{pour}\hspace{0.2cm} \vert \alpha \vert=  1,\hspace{0.2cm} &\vert \partial_{\eta}^{ \alpha }({\vert B(v)\eta \vert}^{2}) \vert \lesssim \vert B(v) \vert \vert B(v)\eta \vert\lesssim {\lambda}^{2},\\
\text{pour}\hspace{0.2cm} \vert \alpha \vert=  2,\hspace{0.2cm} &\vert \partial_{\eta}^{ \alpha }({\vert B(v)\eta \vert}^{2})\vert \lesssim \vert B^{T}(v)B(v)\vert\lesssim {\lambda}^{2},\\
 \text{pour}\hspace{0.2cm}  \vert \alpha \vert  \geq 3,\hspace{0.2cm}  &\vert \partial_{\eta}^{ \alpha } ({\vert B(v)\eta \vert}^{2}) \vert=0,
  \end{align*}
  so we get
  \begin{align}\label{M131}
  \forall \alpha \in \mathbb{N}^{3} ,\hspace{0.2cm}  \vert \partial_{\eta}^{ \alpha } ({\vert B(v)\eta \vert}^{2}) \vert
  \lesssim {\lambda}^{2}.
  \end{align}
Moreover,  one can estimate from above the modulus of all the derivatives of the term ${\vert B(v)\eta \vert}^{2}$  by a constant times $\lambda^{2}$.

\underline{\text{For}\hspace{0.2cm}$F(v)$:}
Using Lemma \fcolorbox{red}{ white}{\ref{N80}}, we have
  \begin{align}
  \forall \alpha \in \mathbb{N}^{3} ,\hspace{0.2cm}
  \vert
 \partial_{v}^{ \alpha  }F(v)\vert \lesssim{\langle v \rangle}^{\gamma+2-\vert \alpha \vert } \lesssim {\lambda}^{2},
 \end{align}
 which gives that $F(v)\in S(\lambda^{2},\Gamma)$ uniformly with respect to the parameter $\xi$.\\
In addition, we have that  $ R_{1}, R_{2}\in S(\lambda^{2},\Gamma)$. From the above, we can deduce that  $\Tilde{\sigma}$ belongs to  $ S(\lambda^{2},\Gamma)$ uniformly with respect to the parameter  $\xi$.
  Using the following notation
 $$H_{17}=\underbrace{ \mathcal{A}_{\xi}\big(\big(\lambda_{K}^{2}\big)^{w}\big)^{-1}}_{\in\mathcal{B}(L^{2})},$$
Then there exists $\mathbf{C}_{0}>0$ such that
 \begin{align}
      \Vert H_{17}\varphi \Vert_{L^{2}}^{2} \leq \mathbf{C}_{0} \Vert \varphi \Vert_{L^{2}}^{2}\hspace{0.2cm} \forall \varphi \in L^{2}(\mathbb{R}^{3}_{v}),
 \end{align}
 uniformly with respect to the parameter $\xi$,  which implies for all  $u\in \mathcal{S}(\mathbb{R}^{3}_{v})$,
 \begin{align}
\Vert  \mathcal{A}_{\xi}u \Vert_{L^{2}}^{2} \leq \mathbf{C}_{0} \Vert \big({\lambda^{2}_{K}}\big)^{w}u \Vert_{L^{2}}^{2}.
 \end{align}
 \end{proof}
 \section{Hypoelliptic estimates for the whole linearized Landau operator}\label{z3}
In this section, we show hypoelliptic estimates with respect to the velocity and position variables for the Landau operator ${\mathcal{P}}$. These estimates allow us to locate the spectrum and estimate the resolvent of the Landau operator. We denote by  $\Lambda_{K}$ the operator associated to the symbol $\lambda_{K}$ by considering the  inverse Fourier transform with respect to the variable  $x$.

 \begin{thm}\label{N102}
There exists $C>0$ such that for all $u\in \mathcal{S}(\mathbb{R}^{6}_{x,v})$,
 \begin{align}\label{W4}
 \Vert \Lambda^{2/3}_{K}u \Vert_{{L^{2}_{x,v}}}^{2}\leq C\left(\Vert {\mathcal{P}}u \Vert_{L^{2}_{x,v}}^{2} +\Vert u \Vert_{L^{2}_{x,v}}^{2}\right).
  \end{align}
 \end{thm}
 \begin{proof}
Using Theorem  \fcolorbox{red}{ white}{\ref{M36}},  we have for all  $u\in \mathcal{S}(\mathbb{R}^{3}_{v})$,
 $$   \Vert\big(\lambda_{K}^{2/3}\big)^{w}u \Vert_{L^{2}_{v}}^{2}\sim  \Vert\big(\lambda_{K}^{w}\big)^{2/3}u \Vert_{L^{2}_{v}}^{2},$$
so using the Proposition \fcolorbox{red}{ white}{\ref{N100}}, there exists a constant $C_{1}>0$ 
such that for all  $u\in \mathcal{S}(\mathbb{R}^{3}_{v})$,  
 \begin{align}
\Vert \big(\lambda_{K}^{w}\big)^{2/3}u \Vert_{L^{2}_{v}}^{2}\leq C_{1}\left(\Vert\mathcal{A}_{\xi}u \Vert_{L^{2}_{v}}^{2} +\Vert u \Vert_{L^{2}_{v}}^{2}\right),
 \end{align}
  uniformly with respect to the parameter $\xi$.
  By integrating the previous inequality with respect to the parameter $\xi$ \text{in} $\mathbb{R}^{3}$ and  considering the inverse Fourier transform with respect to the variable  $x$, we obtain for all  $u\in \mathcal{S}(\mathbb{R}^{6}_{x,v})$,  
 \begin{align}\label{N101}
\Vert \Lambda^{2/3}_{K}u \Vert_{L^{2}_{x,v}}^{2}\leq C_{1} \left(\Vert {\mathcal{A}}u \Vert_{L^{2}_{x,v}}^{2} +\Vert u \Vert_{L^{2}_{x,v}}^{2}\right).
  \end{align}
  Using (\fcolorbox{red}{ white}{\ref{N14}}), the operator $\mathcal{P}$ is written as follows
  $$\mathcal{P}=\mathcal{A}+\mathcal{K}.$$ 
 Consequently, using  (\fcolorbox{red}{ white}{\ref{N101}}), we have 
  \begin{align*}
\Vert \Lambda^{2/3}_{K}u \Vert_{L^{2}_{x,v}}^{2}&\leq C_{1} \left(\Vert {(\mathcal{A+K-K}})u \Vert_{L^{2}_{x,v}}^{2} +\Vert u \Vert_{L^{2}_{x,v}}^{2}\right)\\
&\leq C_{1} \left(\Vert {\mathcal{P}}u \Vert_{L^{2}_{x,v}}^{2} +\Vert {\mathcal{K}}u \Vert_{L^{2}_{x,v}}^{2}+\Vert u \Vert_{L^{2}_{x,v}}^{2}\right)
  \end{align*}
  and using the fact that $\mathcal{K}$ is a bounded operator,  we obtain that there exists a constant  $C>0$ such that for all  $u\in \mathcal{S}(\mathbb{R}^{6}_{x,v})$,  
  \begin{align} 
 \Vert \Lambda^{2/3}_{K}u \Vert_{{L^{2}_{x,v}}}^{2}\leq C\left(\Vert {\mathcal{P}}u \Vert_{L^{2}_{x,v}}^{2} +\Vert u \Vert_{L^{2}_{x,v}}^{2}\right).
  \end{align}
   \end{proof}
 By adding a term $i\kappa$ to $\mathcal{P}$ with $\kappa\in\mathbb{R}$, the proof doesn't change due to the never changing of the real part of $\mathcal{P}$ as mentionned crucially in Remark 2.2 in \cite{herau_isotropic_2004}. So Theorem \fcolorbox{red}{ white}{\ref{N102}} admits the following extension.
 \begin{thm}
There exists $C>0$ such that for all $u\in \mathcal{S}(\mathbb{R}^{6}_{x,v})$,
 \begin{align}\label{M142}
 \forall \kappa \in \mathbb{R},\hspace{0.2cm}\Vert \Lambda^{2/3}_{K}u \Vert_{L^{2}_{x,v}}^{2}\leq C\left(\Vert (\mathcal{P}-i\kappa) u \Vert_{L^{2}_{x,v}}^{2} +\Vert u \Vert_{L^{2}_{x,v}}^{2}\right).
 \end{align}
\end{thm}

 \begin{thm}
  There exists  $\mathbf{C}>0$  such that for all $u\in \mathcal{S}(\mathbb{R}^{6}_{x,v})$,
 \begin{align}\label{N111}
\Vert {\mathcal{P}}u \Vert_{L^{2}_{x,v}}^{2} \leq \mathbf{C} \Vert \Lambda^{2}_{K}u \Vert_{L^{2}_{x,v}}^{2}.
 \end{align}
 \end{thm}
 \begin{proof}
Using Theorem  \fcolorbox{red}{ white}{\ref{M36}} , we have for all $u\in \mathcal{S}(\mathbb{R}^{3}_{v})$, 
$$   \Vert\big(\lambda_{K}^{2}\big)^{w}u \Vert_{L^{2}_{v}}^{2}\sim  \Vert\big(\lambda_{K}^{w}\big)^{2}u \Vert_{L^{2}_{v}}^{2},$$
so using the Proposition  \fcolorbox{red}{ white}{\ref{N105}}, there exists a constant $\mathbf{C}_{1}>0$ 
such that for all $u\in \mathcal{S}(\mathbb{R}^{3}_{v})$, 
 \begin{align}
\Vert\mathcal{A}_{\xi}u \Vert_{L^{2}_{v}}^{2}\leq  \mathbf{C}_{1} \Vert\big(\lambda_{K}^{w}\big)^{2}u \Vert_{L^{2}_{v}}^{2},
 \end{align}
  uniformly with respect to the parameter $\xi$.
    By integrating the previous inequality withrespect to the parameter $\xi$ \text{in} $\mathbb{R}^{3}$ and considering the inverse Fourier transform with respect to the variable $x$,  we obtain for all $u\in \mathcal{S}(\mathbb{R}^{6}_{x,v})$,
     \begin{align}\label{N110}
\Vert {\mathcal{A}}u \Vert_{L^{2}_{x,v}}^{2} \leq \mathbf{C}_{1} \Vert \Lambda^{2}_{K}u \Vert_{L^{2}_{x,v}}^{2}.
  \end{align}
Using (\fcolorbox{red}{ white}{\ref{N14}}), the operator $\mathcal{P}$ is written as follows
  $$\mathcal{P}=\mathcal{A}+\mathcal{K}.$$ 
  Consequently, using  (\fcolorbox{red}{ white}{\ref{N110}}), we have
  \begin{align*}
  \Vert {\mathcal{P}}u \Vert_{L^{2}_{x,v}}^{2} &\leq\Vert {\mathcal{A}}u \Vert_{L^{2}_{x,v}}^{2} +\Vert {\mathcal{K}}u \Vert_{L^{2}_{x,v}}^{2}
  \\&\leq  \mathbf{C}_{1} \Vert \Lambda^{2}_{K}u \Vert_{L^{2}_{x,v}}^{2}+\mathbf{C}_{2}\Vert u \Vert_{L^{2}_{x,v}}^{2},
  \end{align*}
with  $\mathbf{C}_{2}>0$. Using the fact that  the operator $\Lambda^{2}_{K}\geq \text{Id}$ (see for instance Theorem \fcolorbox{red}{ white}{\ref{M36}} in the Appendix \fcolorbox{red}{ white}{\ref{N130}}), we obtain that there exists a constant
   $\mathbf{C}>0$  such that for all $u\in \mathcal{S}(\mathbb{R}^{6}_{x,v})$,
 \begin{align}
\Vert {\mathcal{P}}u \Vert_{L^{2}_{x,v}}^{2} \leq \mathbf{C} \Vert \Lambda^{2}_{K}u \Vert_{L^{2}_{x,v}}^{2}.
 \end{align}
 \end{proof}
 We srongly think the  estimate in (\fcolorbox{red}{ white}{\ref{W4}}) is  optimal in term of the index $2/3$ appearing in the left hand side, although we don't prove that. This index  $2/3$ is classical for kinetic model. We refer for example to Proposition 5.22 in \cite{article}, partial results in this direction in the Fokker-Planck case.
 \section{Localisation of the spectrum for the Landau operator}\label{z4}
 In all that follows, we denote by  $\sigma(\mathcal{P})$ the spectrum of the operator $\mathcal{P}$, $\rho(\mathcal{P})$ The resolvent set of $\mathcal{P}$ and  $\Vert \cdot \Vert_{L^{2}}$ to denote the norm in the space $L^{2}(\mathbb{R}^{6}_{x,v})$.
The following lemma holds for any maximally accretive operator.
\begin{lem}\label{M164}
Let $(\mathcal{A},D(\mathcal{A}))$ be a maximally accretive operator in the Hilbert space $\mathcal{H}$. For
any  $\eta\in\mathopen{]}0\,,1\mathclose{[}$ , the estimate 
$$\vert z+1 \vert^{2\eta}\Vert u \Vert^{2}\leq 4 \big( ((\mathcal{A}+1)^{*}(\mathcal{A}+1))^{\eta}u,u \big)_{\mathcal{H}} +4\Vert (\mathcal{A}-z)u \Vert^{2}$$
holds for all $u\in D(\mathcal{A})$ and $z\in \mathbb{C}$ with $\Re{e}\hspace{0.05cm} z\geq -1$ .
\end{lem}
\begin{proof}
See Proposition B.1 in \cite{herau_isotropic_2004}.
\end{proof}
\textbf{Proof of Theorem \ref{N5}.}
\begin{proof}
The proof will be divided into two steps. \\
\underline{\textbf{First step:}}\\
 Using (\fcolorbox{red}{ white}{\ref{M142}}) and the triangle inequality we therefore get for all  $z=\nu + i\kappa \in \mathbb{C} ,$ with \hspace{0.1cm} $\nu=\Re{e}\hspace{0.05cm} z\geq -1/2$ and $u\in \mathcal{S}(\mathbb{R}^{6}_{x,v})$, 
\begin{align*}
    \Vert \Lambda^{2/3}_{K}u \Vert_{L^{2}}^{2}&\leq C(\Vert (\mathcal{P}-i\kappa) u \Vert_{L^{2}}^{2} +\Vert u \Vert_{L^{2}}^{2}).\\
     &\leq C(\Vert (\mathcal{P}-i\kappa-\nu+\nu)u \Vert_{L^{2}}^{2} +\Vert u \Vert_{L^{2}}^{2})\\
     &\leq C(2\Vert (\mathcal{P}-z)u \Vert_{L^{2}}^{2} +(2\nu^{2}+1)\Vert u \Vert_{L^{2}}^{2}).
\end{align*}
But we have $\nu\geq -1/2$  implies $2\nu+2\geq 1$, so we get
\begin{align*}
 \Vert \Lambda^{2/3}_{K}u \Vert_{L^{2}}^{2}
   &\leq C(6\Vert (\mathcal{P}-z)u \Vert_{L^{2}}^{2} +6(\nu+1)^{2}\Vert u \Vert_{L^{2}}^{2})\\
   &\leq 6C(\Vert (\mathcal{P}-z)u \Vert_{L^{2}}^{2} +(\nu+1)^{2}\Vert u \Vert_{L^{2}}^{2}),
\end{align*}
By taking $\Tilde{C}=6C$, we finally obtain the following estimate
\begin{equation}\label{M150}
     \Vert \Lambda^{2/3}_{K}u \Vert_{L^{2}}^{2}\leq \Tilde{C} \big(\Vert (\mathcal{P}-z)u \Vert_{L^{2}}^{2} +(\Re{e}\hspace{0.05cm} z+1)^{2}\Vert u \Vert_{L^{2}}^{2}\big),
\end{equation}
for all $u \in \mathcal{S}(\mathbb{R}^{6}_{x,v})$ and $ z\in \mathbb{C}$ with $\Re{e}\hspace{0.05cm} z\geq -\frac{1}{2}$.\\
\underline{\textbf{Second step:}}\\
First, we will show that there exists a constant  $\Tilde{C}_{1}$ such that for all  $u\in \mathcal{S}(\mathbb{R}^{6}_{x,v})$, we have
\begin{align}\label{M161}
    0\leq  \big((\mathcal{A}+1)^{*}(\mathcal{A}+1)u,u \big)_{L^{2}}\leq \Tilde{C}_{1}\big(\Lambda^{4}_{K}u,u\big)_{L^{2}},
\end{align}
where  $\mathcal{A}$ the operator defined in (\fcolorbox{red}{ white}{\ref{N14}}).
Indeed, let  $u\in \mathcal{S}(\mathbb{R}^{6}_{x,v})$, we have
$$
 \big((1+\mathcal{A})^{*}(1+\mathcal{A})u ,u \big)_{L^{2}}
 = \big((1+\mathcal{A})u ,(1+\mathcal{A})u\big)_{L^{2}}
 =\Vert (1+\mathcal{A})u \Vert_{L^{2}}^{2}\geq 0.$$ 
On the other hand, we have
$$\big((1+\mathcal{A})^{*}(1+\mathcal{A})u ,u \big)_{L^{2}}=\Vert (1+\mathcal{A})u \Vert_{L^{2}}^{2}
     \leq\Vert \mathcal{A}u \Vert_{L^{2}}^{2} +\Vert u \Vert_{L^{2}}^{2}.$$
    Using (\fcolorbox{red}{ white}{\ref{N14}}), the operator $\mathcal{P}$ is written as follows
  $$\mathcal{P}=\mathcal{A}+\mathcal{K},$$ 
  we obtain 
\begin{align*}
\big((1+\mathcal{A})^{*}(1+\mathcal{A})u ,u \big)_{L^{2}}&\leq\Vert {\mathcal{P}}u \Vert_{L^{2}}^{2} +\Vert {\mathcal{K}}u \Vert_{L^{2}}^{2}+\Vert u \Vert_{L^{2}}^{2},
\end{align*}
  finally, using the fact that  $\mathcal{K}$ is a bounded operator, $\Lambda_{K}\geq \text{Id}$  and the estimate  (\fcolorbox{red}{ white}{\ref{N111}}), we obtain 
     \begin{align*}
\big((1+\mathcal{A})^{*}(1+\mathcal{A})u ,u \big)_{L^{2}}&\leq (1+\nu_{0})\Vert \Lambda_{K}^{2}u \Vert_{L^{2}}^{2} + \mathbf{C}  \Vert \Lambda_{K}^{2}u \Vert_{L^{2}}^{2}\\
  & \leq (1+\nu_{0}+\mathbf{C} )  \Vert \Lambda_{K}^{2}u \Vert_{L^{2}}^{2},
        \end{align*}
where  $\nu_{0}>0$. By taking  $\Tilde{C}_{1}=(1+\nu_{0}+\mathbf{C} )$, we obtain the estimate (\fcolorbox{red}{ white}{\ref{M161}}).\\
According to the monotonicity of the operator functional  $\mathcal{A}\longrightarrow \mathcal{A}^{\alpha}$ for $\alpha \in\mathopen{[}0\,,1\mathclose{]},$  in particular with  $\alpha=\dfrac{1}{3}$, we obtain 
$$0\leq\big(((1+\mathcal{A})^{*}(1+\mathcal{A}))^{1/3}u,u\big)_{L^{2}}\leq {\Tilde{C}_{1}}^{2/3} \big(\Lambda_{K}^{4/3}u,u\big)_{L^{2}}.$$
According to Theorem \fcolorbox{red}{ white}{\ref{N15}}, $\mathcal{A}$ is maximally accretive, then by applying  Lemma  \fcolorbox{red}{ white}{\ref{M164}} with $\eta=\dfrac{1}{3}$, for $\Re{e}\hspace{0.05cm} z\geq -1/2 \hspace{0.2cm}\text{and}\hspace{0.2cm} u\in \mathcal{S}(\mathbb{R}^{6}_{x,v})$, we obtain
\begin{align*}
    \vert z+1 \vert^{2/3}\Vert u \Vert_{L^{2}}^{2}&\leq 4 \big(((1+\mathcal{A})^{*}(1+\mathcal{A}))^{1/3}u,u\big)_{L^{2}}+4\Vert (\mathcal{A}-z)u \Vert_{L^{2}}^{2}\\
    &\leq4{\Tilde{C}_{1}}^{2/3}\big(\Lambda_{K}^{4/3}u,u\big)_{L^{2}} +4\Vert (\mathcal{A}-z)u \Vert_{L^{2}}^{2}\\
    &\leq 4{\Tilde{C}_{1}}^{2/3}\Vert \Lambda_{K}^{2/3}u \Vert_{L^{2}}^{2} +4\Vert (\mathcal{A}-z)u \Vert_{L^{2}}^{2}\\
    &\leq 4{\Tilde{C}_{1}}^{2/3}\Vert \Lambda_{K}^{2/3}u \Vert_{L^{2}}^{2} +4\Vert (\mathcal{P}-z)u \Vert_{L^{2}}^{2}+4\nu_{1}\Vert u \Vert_{L^{2}}^{2}.
\end{align*}
With the inequality (\fcolorbox{red}{ white}{\ref{M150}}), we obtain
$$\vert z+1 \vert^{2/3}\Vert u \Vert_{L^{2}}^{2}\leq(4{\Tilde{C}_{1}}^{2/3}\Tilde{C}+4)\Vert (\mathcal{P}-z)u \Vert_{L^{2}}^{2}+
(4{\Tilde{C}_{1}}^{2/3}\Tilde{C}+16\nu_{1})(\Re{e}\hspace{0.05cm} z +1)^{2}\Vert u \Vert_{L^{2}}^{2}.$$
By taking $ C_{\mathcal{P}}=\sqrt{8(4\nu_{1}+{\Tilde{C}_{1}}^{2/3}\Tilde{C})},\hspace{0.1cm} Q_{\mathcal{P}}=\sqrt{8(1+{\Tilde{C}_{1}}^{2/3}\Tilde{C})}$ 
we finally get \\
$\forall z\in \mathbb{C},\hspace{0.1cm}\Re{e}\hspace{0.05cm} z\geq -1/2,\hspace{0.1cm}\forall u \in \mathcal{S}(\mathbb{R}^{6}_{x,v})$,
\begin{equation}\label{M165}
\vert z+1 \vert^{2/3}\Vert u \Vert_{L^{2}}^{2}\leq \frac{{Q^{2}_{\mathcal{P}}}}{2}\Vert (\mathcal{P}-z)u \Vert_{L^{2}}^{2}+
\frac{{C^{2}_{\mathcal{P}}}}{2}(\Re{e}\hspace{0.05cm} z +1)^{2}\Vert u \Vert_{L^{2}}^{2}.
\end{equation}

Now let  $z \notin S_{\mathcal{P}}$ such that  $\Re{e}\hspace{0.05cm}z\geq -1/2$, so according to the definition of $S_{\mathcal{P}}$ given in  (\fcolorbox{red}{ white}{\ref{N112}}) we have
$$(\Re{e}\hspace{0.05cm} z +1)^{2}\leq  \dfrac{1}{C_{\mathcal{P}}^{2}}\vert z+1 \vert^{2/3},$$
then for all  $u\in \mathcal{S}(\mathbb{R}^{6}_{x,v})$, inequality  (\fcolorbox{red}{ white}{\ref{M165}}) implies 
\begin{align}\label{M170}
    \vert z+1 \vert^{2/3}\Vert u \Vert_{L^{2}}^{2}&\leq  {Q}_{\mathcal{P}}^{2}\Vert (\mathcal{P}-z)u \Vert_{L^{2}}^{2},
\end{align}
we deduce that ${\mathcal{P}}-z$ is injective, moreover we can replace ${\mathcal{P}}-z $ by $({\mathcal{P}}-z)^{*}$ in (\fcolorbox{red}{white}{\ref{M170}}), which gives $({\mathcal{P}}-z)^{*}$ is injective and consequently ${\mathcal{P}}-z$ is bijective with dense image in $L^{2}$, therefore $z\in\rho(\mathcal{P})$.
By taking $v=(\mathcal{P}-z)u$ in the estimate (\fcolorbox{red}{ white}{\ref{M170}}) we get
\begin{align*}
\Vert {(\mathcal{P}-z)}^{-1}v \Vert_{L^{2}}^{2}&\leq  Q^{2}_{\mathcal{P}}\vert z+1 \vert^{-2/3}\Vert v \Vert_{L^{2}}^{2}\hspace{0.2cm}\forall v \in L^{2},
\end{align*}
then we obtain  the resolvent estimate (\fcolorbox{red}{ white}{\ref{N113}}).\\ 
On the other hand, we have that if $z\in \sigma(\mathcal{P})$ then $z \in S_{\mathcal{P}}$ and taking into account  that the numerical range of the operator  $\mathcal{P}$  is the half plan $\left\lbrace \Re{e}\hspace{0.05cm} z \geq 0\right\rbrace$, we deduce that the spectrum $\sigma(\mathcal{P})$ satisfies
$$\sigma(\mathcal{P})\subset S_{\mathcal{P}}\cap \lbrace{\Re{e}\hspace{0.05cm} z \geq 0\rbrace}.$$
Concerning  estimate (\fcolorbox{red}{ white}{\ref{W3}}), we have 
\begin{align*}
    \Re{e}\left((\mathcal{P}-z)u,u\right)_{L^{2}}\geq -\Re{e}\hspace{0.05cm}z\Vert u \Vert_{L^{2}}^{2}\hspace{0.3cm}\text{with}\hspace{0.3cm} \Re{e}\hspace{0.05cm}z\leq-\frac{1}{2}<0,
\end{align*}
which implies that
    \begin{align*}
    {\Vert (z-\mathcal{P})^{-1} \Vert}_{\mathcal{B}(L^{2}_{x,v})} \leq \vert \Re{e}\hspace{0.05cm}z \vert^{-1}.
    \end{align*}
The proof is then complete.
\end{proof}
\begin{appendix}
\section{Appendix}\label{1001}
\subsection{Weyl-Hörmander calculus}\label{S3}
We recall here some notations and basic facts of symbolic calculus, and refer to  \cite{lerner_metrics_2011} and  \cite{hormander_analysis_2007} for detailed discussions on the pseudo-differential calculus.\\ 
We introduce on  $\mathbb{R}^{2n}$ the following metric
$$\Gamma=dv^{2}+d\eta^{2}.$$
\begin{defi}\label{S1}
Let $m\geq 1$ be a $C^{\infty}$ function on $\mathbb{R}^{2n}$. We say that  $m$ is an admissible weight for  $\Gamma$  if there exist two constants $C>0$ and $N>0$ such that
\begin{equation}\label{M61}
\forall X,Y\in \mathbb{R}^{2n},\quad m(X)\leq C \left\langle X-Y\right\rangle^{N}m(Y).
\end{equation}
\end{defi}
\begin{defi}\label{S2}
Let $m$  be an admissible function. We denote by $S(m,\Gamma)$ the symbol class of all smooth functions  $p(v,\eta)$ (possibly depending on parameter  $\xi$) satisfying 
 $$ \forall \alpha,\beta \in \mathbb{N}^{n},\exists\hspace{0.05cm} C_{\alpha,\beta}>0 ;\hspace{0.05cm}\forall (v,\eta)\in \mathbb{R}^{2n},\hspace{0.05cm}\vert \partial_{v}^{\alpha}\partial_{\eta}^{\beta}p(v,\eta)\vert \leq C_{\alpha,\beta}m(v,\eta).$$
\end{defi}
The space of symbols  $S(m,\Gamma)$ endowed with the semi-norms
 \begin{equation}\label{M1}
 \Vert p \Vert_{k;S(m,\Gamma)}=\underset{\vert \alpha+\beta \vert \leq k} {\text{sup}}\hspace{0.1cm}\underset{(v,\eta)\in \mathbb{R}^{2n}}{\text{sup}}\vert m(v,\eta)^{-1} \partial_{v}^{\alpha}\partial_{\eta}^{\beta}p(v,\eta)\vert\hspace{0.05cm};\hspace{0.1cm} k\in \mathbb{N}
  \end{equation}
 becomes a Fréchet space.\\
 For such a symbol $p$ in  $S(m,\Gamma)$ we may define its Weyl
quantization  $p^{w}$ by  
\begin{equation}\label{M27}
 \forall u\in \mathcal{S}(\mathbb{R}^{n}),\hspace{0.3cm} (p^{w}u)(v)=\dfrac{1}{(2\pi)^{n}}\int_{\mathbb{R}^{2n}}e^{i(v-v')\cdot\eta}p\bigg(\dfrac{v+v'}{2},\eta\bigg)u(v')\hspace{0.1cm}dv'd\eta.
\end{equation}
The Weyl quantization of $S(m,\Gamma)$ is denoted by $\Psi(m,\Gamma)$.
 \begin{thm}\textbf{(Calderon-Vaillancourt)}\label{M2}\\
 Let $p^{w}$ be an operator in   $\Psi(1,\Gamma)$. We have $p^{w}$ a continuous operator on $L^{2}(\mathbb{R}^{n})$ and
 \begin{equation}\label{M4}
  \forall u \in L^{2}(\mathbb{R}^{n}),\quad  \Vert p^{w}u \Vert_{L^{2}} \leq C\Vert p \Vert_{N;S(1,\Gamma)}\Vert u \Vert_{L^{2}},
\end{equation}
where $C>0$ and a positive integer $N$ depending only on the dimension.
\end{thm}
\begin{proof}
See Section 18 in \cite{hormander_analysis_2007}.
\end{proof}
\begin{thm}\label{M10}
Let $p^{w}$ be an invertible operator in $\Psi(m,\Gamma)$, then its inverse  $[p^{w}]^{-1}$ belongs to  $\Psi(m^{-1},\Gamma)$.  
\end{thm}
\begin{proof}
See Lemma A.2 in  \cite{herau_isotropic_2004}.
\end{proof}
\begin{defi}
Let $p^{w} \in \Psi(m,\Gamma)$. We say that $p^{w}$ is an  elliptic operator if there exists $C>0$ such that
$$\vert p \vert \geq C m .$$
\end{defi}
Let us also recall here the composition formula of Weyl quantization.
Let $a\in S(m_{1},\Gamma)$ and $b\in S(m_{2},\Gamma)$,
the compositions of the pseudo-differential operators $a^{w}$ and $b^{w}$ are pseudo-differential operators whose symbol, denoted $a \sharp b$,  belongs to  $S(m_{1}m_{2},\Gamma)$ and has the following development:
\begin{equation}\label{M3}
a \sharp b=ab+\int_{0}^{1} \iint e^{-i\sigma(Y-Y_{1},Y-Y_{2})/(2\theta)}\frac{i}{2}\sigma(\partial_{Y_{1}},\partial_{ Y_{2}})a(Y_{1})b(Y_{2})\mathrm{d}Y_{1}\mathrm{d}Y_{2}\mathrm{d}\theta/(\theta)^{2n},
\end{equation}
 where $\sigma$ is the symplectic form in $T^{*}\mathbb{R}^{n}=\mathbb{R}^{2n}$ given by 
 \begin{eqnarray}\label{M6}
  \sigma(Z,Z')=\sum_{j=0}^{n}\zeta_{j}z'_{j}-z_{j}\zeta'_{j}.
   \end{eqnarray}
   As we work with pseudo-differential operators which belong to classes associated to the metric  $\Gamma$, all operators will be defined as continuous operators of $ \mathcal{S}(\mathbb{R}^{n})$ to $ \mathcal{S}(\mathbb{R}^{n})$ or from $ \mathcal{S}^{'}(\mathbb{R}^{n})$ to $ \mathcal{S}^{'}(\mathbb{R}^{n})$.
 \subsection{Basic theorem}\label{N130}
This theorem aims at giving a uniform statement for Weyl Hörmander tools with a large parameter $K$. This theorem gives very important results which can use these results to establish estimates which requires pseudo-differential operators. Part (I) in the theorem below has been shown in \cite{alexandre_global_2012} with $\tau=1$, but here we have improved this result for all $\tau\in \mathbb{R}$. In general, these results give a general and robust framework to techniques already used as well for work that requires these kinds of properties.
\begin{thm}\label{M36}
Let $p,q\geq 1$ be two symbols which verify the following hypotheses:
\begin{itemize}
    \item[i)] $p$ is an admissible weight.
    \item[ii)] $p\sim q$.
      \item[iii)] $p,\hspace{0.1cm}q \in S(p,\Gamma)$.
      \item[iv)] There exists  $M\in \mathbb{R}$ such that for all       $\varepsilon > 0$ we have  $\partial_{\eta}p,\hspace{0.1cm}\partial_{\eta}q \in S(\varepsilon p +\varepsilon^{-1}{\langle v \rangle}^{M} ,\Gamma)$.
\end{itemize}
 We define the symbols $p_{K}$, $q_{K}$ as  $p_{K}=p+K{\langle v \rangle}^{M}$    and  
$ q_{K}=q+K{\langle v \rangle}^{M}$.
 Then there exists $K_{0}$  such that for all $K\geq K_{0}:$ 
\begin{itemize}
    \item[I)] For all $\tau \in \mathbb{R}$,  $(p_{K}^{\tau})^{w}$ and $ (q_{K}^{\tau})^{w}$ are invertible.
    \item[II)] For all $\tau \in \mathbb{R}$,
   ${[ (p_{K}^{\tau})^{w}]}^{-1}$ and ${[ (q_{K}^{\tau})^{w}]}^{-1}$ are pseudo-differential operators that belong to $\Psi(p_{K}^{-\tau},\Gamma)$ uniformly in $K$.
      \item[III)] For all $\tau \in \mathbb{R}$ and for all  $\kappa\geq 0$,  $ {[({p_{K}^{\kappa}})^{w}]}^{\tau}$ and ${[ ({q_{K}^{\kappa}})^{w}]}^{\tau}$ are pseudo-differential operators belong to $\Psi(p_{K}^{\kappa\tau},\Gamma)$  uniformly in $K$.
       \item[IV)] For all $\tau \in \mathbb{R}$ and for all  $\kappa\geq 0$,  $\forall u \in  \mathcal{S}(\mathbb{R}^{n})$, we have $$\Vert  {[({p_{K}^{\kappa}})^{w}]}^{\tau}u \Vert_{L^{2}}^{2}\sim \Vert (p_{K}^{\kappa\tau})^{w}u \Vert_{L^{2}}^{2}, $$
   uniformly in $K$.
      \item[V)]  For all $\tau \in \mathbb{R}$, $\forall u \in  \mathcal{S}(\mathbb{R}^{n})$, we have $$\Vert  (p_{K}^{w})^{\tau}u \Vert_{L^{2}}^{2}\sim \Vert (p_{K}^{\tau})^{w}u \Vert_{L^{2}}^{2}\sim \Vert (q_{K}^{\tau})^{w}u \Vert_{L^{2}}^{2} \sim \Vert  (q_{K}^{w})^{\tau}u \Vert_{L^{2}}^{2} ,$$
       uniformly in $K$.
      \item[VI)]$\forall u \in  \mathcal{S}(\mathbb{R}^{n})$, we have $$(p_{K}^{w}u,u)_{L^{2}}\sim (q_{K}^{w}u,u)_{L^{2}},$$
        uniformly in $K$.
       \item[VII)] If $M\geq 0$, we have $p_{K}^{w}$, $q_{K}^{w} \geq$ \text{Id}. 
      \end{itemize}
\end{thm}
\begin{proof}
Let $K>0$ and let $p$ and $q$ two symbols and $M\in \mathbb{R}$ such that the hypotheses from  (i) to (iv) are verified.
For simplification, we will prove 
 (I, II, III) just for the operator  $(p_{K}^{\tau})^{w}$, and we reason in the same way for $(q_{K}^{\tau})^{w}$.  
 \item[I.] We will show that $ (p_{K}^{\tau})^{w}$  is invertible for all $\tau \in \mathbb{R}$. We note that $p_{K}$   is an admissible weight  ($p$ is an admissible weight). We have  $p_{K}$    $ \in S(p_{K} ,\Gamma)$
 uniformly in  $K$. Indeed, using (iii) we have for all  $ \alpha,\beta \in \mathbb{N}^{n}$,
  \begin{align*}
      \vert \partial_{v}^{\alpha}\partial_{\eta}^{\beta}p_{K}\vert&\leq \vert \partial_{v}^{\alpha}\partial_{\eta}^{\beta}p\vert+\vert \partial_{v}^{\alpha}\partial_{\eta}^{\beta}K{\langle v \rangle}^{M}\vert\\
      &\leq C_{\alpha,\beta}\hspace{0.05cm}p+K C_{\alpha,\beta,M}{\langle v \rangle}^{M}\\
      &\leq {\widetilde C_{\alpha,\beta,M}}\hspace{0.05cm}p_{K}.
  \end{align*}
  More generally we can show,  by induction on $\vert\alpha\vert$ and Leibnitz's formula that for $\tau\in \mathbb{R}$,
  $$\forall \alpha \in \mathbb{N}^{2n},\hspace{0.2cm} \vert \partial_{v,\eta}^{\alpha}p_{K}^{\tau}\vert \leq  C_{\alpha,\beta,\tau}\hspace{0.05cm}p_{K}^{\tau },$$
  which gives $p_{K}^{\tau} \in S(p_{K}^{\tau} ,\Gamma)$  uniformly in $K$.\\
 Using formula (\fcolorbox{red}{ white}{\ref{M3}}), we may write
  \begin{align}\label{M5}
 (p_{K}^{\tau})^{w}(p_{K}^{-\tau})^{w}=\text{Id}- R_{K}^{w},
\end{align}
  where 
  $$ R_{K}=-\int_{0}^{1} (\partial_{\eta}p_{K}^{\tau})\sharp_{\theta}\big( \partial_{v}(p_{K}^{-\tau})\big)\, \mathrm{d}\theta + \int_{0}^{1} (\partial_{v}p_{K}^{\tau})\sharp_{\theta}\big( \partial_{\eta}(p_{K}^{-\tau})\big)\, \mathrm{d}\theta$$
 with $g\sharp_{\theta} h$ defined by
  \begin{equation}\label{M12}
  g\sharp_{\theta} h(Y)=\iint e^{-2i\sigma(Y-Y_{1},Y-Y_{2})/\theta}\frac{1}{2i}g(Y_{1})h(Y_{2})\mathrm{d}Y_{1}\mathrm{d}Y_{2}/(\pi\theta)^{2n},
  \end{equation}
  with $Y,Y_{1},Y_{2}\in \mathbb{R}^{2n}$ and $\sigma$ a symplectic form defined in  \fcolorbox{red}{ white}{\ref{M6}}.\\
  Let now $N$  be the integer which is given in (\fcolorbox{red}{ white}{\ref{M4}}). By \cite[Proposition 1.1]{SEDP_1998-1999____A3_0}, we can find a constant $C_{N}$ and a positive integer $l_{N}$, both depending only on $N$ but independent of $K$ and $\theta$, such that 
  $$\Vert (\partial_{\eta}p_{K}^{\tau})\sharp_{\theta}\big( \partial_{v}(p_{K}^{-\tau})\big)\Vert_{N;S(1,\Gamma)}\leq C_{N}\Vert \partial_{\eta}p_{K}^{\tau}\Vert_{l_{N};S(p_{K}^{\tau},\Gamma)}\Vert \partial_{v}(p_{K}^{-\tau})\Vert_{l_{N};S(p_{K}^{-\tau},\Gamma)},$$
  where the semi-norm $\Vert . \Vert_{k;S(M,\Gamma)}$ is defined by (\fcolorbox{red}{ white}{\ref{M1}}).\\ Moreover, using (iv), we have $\partial_{\eta}p \in S(\varepsilon p +\varepsilon^{-1}{\langle v \rangle}^{M} ,\Gamma)$, by taking $\varepsilon=K^{-1/2}$, we obtain $\partial_{\eta}p \in S(K^{-1/2}p_{K} ,\Gamma)$. By writing $\partial_{\eta}p_{K}^{\tau}=\tau p_{K}^{\tau-1}\partial_{\eta}p$, we obtain
  $$
      \vert \partial_{\eta}p_{K}^{\tau}\vert \leq {\widetilde C_{N}} K^{-1/2}p_{K}^{\tau}.$$
  Then arguing as above we can use induction on $\vert \alpha\vert+\vert\beta\vert$ to obtain, for $\vert \alpha\vert+\vert\beta\vert\geq 0$,
  \begin{align*}
      \vert \partial_{v}^{\alpha}\partial_{\eta}^{\beta}\partial_{\eta}p_{K}^{\tau}\vert 
      &\leq {\widetilde C_{N}}K^{-1/2}p_{K}^{\tau},
  \end{align*}
  which gives $\partial_{\eta}p_{K}^{\tau} \in S(K^{-1/2}p_{K}^{\tau} ,\Gamma)$ uniformly in $K$, moreover we have 
  $$\Vert (\partial_{\eta}p_{K}^{\tau})\Vert_{l_{N};S(p_{K}^{\tau},\Gamma)}\leq {\widetilde C_{N}}K^{-1/2}.$$
  On the other hand we have $p_{K}^{-\tau} \in S(p_{K}^{-\tau} ,\Gamma)$, and thus
  $$\Vert \partial_{v}(p_{K}^{-\tau})\Vert_{l_{N};S(p_{K}^{-\tau},\Gamma)}\leq {\widetilde C_{N}},$$
with ${\widetilde C_{N}}$ a constant depending only on $N$ but independent of $K$. As a result,
\begin{equation}\label{M7}
\Vert (\partial_{\eta}p_{K}^{\tau})\sharp_{\theta}\big( \partial_{v}(p_{K}^{-\tau})\big)\Vert_{N;S(1,\Gamma)}\leq C_{N}{{\widetilde C_{N}}}^{2}K^{-1/2}.
\end{equation}
Similarly,
\begin{equation}\label{M8}
\Vert (\partial_{v}p_{K}^{\tau})\sharp_{\theta}\big( \partial_{\eta}(p_{K}^{-\tau})\big) \Vert_{N;S(1,\Gamma)}\leq C_{N}{{\widetilde C_{N}}}^{2}K^{-1/2}.
\end{equation}
Using the estimates  (\fcolorbox{red}{ white}{\ref{M7}}), (\fcolorbox{red}{ white}{\ref{M8}})  we will estimate the semi-norm $\Vert R_{K} \Vert_{N;S(1,\Gamma)}$. Indeed,
\begin{align*}
 \Vert R_{K} \Vert_{N;S(1,\Gamma)}\leq \int_{0}^{1}\Vert (\partial_{v}p_{K}^{\tau})\sharp_{\theta}\big( \partial_{\eta}(p_{K}^{-\tau})\big) \Vert_{N;S(1,\Gamma)} \, \mathrm{d}\theta + \int_{0}^{1} \Vert (\partial_{\eta}p_{K}^{\tau})\sharp_{\theta}\big( \partial_{v}(p_{K}^{-\tau})\big)\Vert_{N;S(1,\Gamma)}\, \mathrm{d}\theta.
\end{align*}
Then
$$\Vert R_{K} \Vert_{N;S(1,\Gamma)}\leq 2 C_{N}{{\widetilde C_{N}}}^{2}K^{-1/2},$$
and thus by (\fcolorbox{red}{ white}{\ref{M4}}) 
$$\Vert R_{K}^{w} \Vert_{\mathcal{B}(L^{2})}\leq 2 C C_{N}{{\widetilde C_{N}}}^{2}K^{-1/2}$$
with $C$ a constant depending only on the dimension. Taking $K_{1}={\big(4 C C_{N}{{\widetilde C_{N}}}^{2}}\big)^{2}$, so we get for all $K\geq K_{1}$ $$\Vert R_{K}^{w} \Vert_{\mathcal{B}(L^{2})}\leq \frac{1}{2}< 1,$$ this implies $\text{Id}- R_{K}^{w}$ is invertible in the space $\mathcal{B}(L^{2})$.
In addition, its inverse is given by
\begin{equation}\label{M9}
     (\text{Id}- R_{K}^{w})^{-1}=\sum_{i=0}^{\infty}{(R_{K}^{w})}^{i} \in \mathcal{B}(L^{2}).
     \end{equation}
Based on (\fcolorbox{red}{ white}{\ref{M5}}) we obtain 
$$ (p_{K}^{\tau})^{w}\big((p_{K}^{-\tau})^{w}(\text{Id}- R_{K}^{w})^{-1}\big)=\text{Id}.$$
Similarly we can find a ${\hat{ R}_{K}}\in S(1,\Gamma)$ such that
$$ \big((\text{Id}- {\hat{ R}_{K}}^{w})^{-1}(p_{K}^{-\tau})^{w}\big)(p_{K}^{\tau})^{w}=\text{Id}.$$

Based on the above $ (p_{K}^{\tau})^{w}$ is invertible and its inverse $[ {(p_{K}^{\tau})^{w}]}^{-1}$ is written in the form 
$$ {[ (p_{K}^{\tau})^{w}]}^{-1}=(p_{K}^{-\tau})^{w}(\text{Id}- R_{K}^{w})^{-1}=(\text{Id}- {\hat{ R}_{K}}^{w})^{-1}(p_{K}^{-\tau})^{w}.$$
We have proved the conclusion in (I).
\item[II.] According to (I), we have ${[ (p_{K}^{\tau})^{w}]}^{-1}$
has the form
$${[ (p_{K}^{\tau})^{w}]}^{-1}=(p_{K}^{-\tau})^{w}(\text{Id}- R_{K}^{w})^{-1}.$$
By taking $H_{K}=(\text{Id}- R_{K}^{w})^{-1}$ and using  (\fcolorbox{red}{ white}{\ref{M9}}), we have  $ H_{K} $ is a continuous operator in $ L ^{2} $ uniformly in $ K $. So based on Theorem \fcolorbox{red}{ white}{\ref{M10}}, we have  $H_{K}$ is a pseudo-differential operator. We note by  $\delta(H_{K})$ its symbol. We have    $\delta(H_{K})$ belongs to $S(1,\Gamma)$ uniformly in $ K $.
From the above, we have $[ (p_{K}^{\tau})^{w}]^{-1}$ is a pseudo-differential operator (compositions of pseudo-differential operators). We denote $ h_{K} $ its symbol, $ h_{K} $ has the form $h_{K}=p_{K}^{-\tau}\sharp \delta(H_{K})$ and belongs to $S(p_{K}^{-\tau},\Gamma)$ uniformly in $K$.\\
We have proved the conclusion in (II).
\item[III.] Using the Theorem \fcolorbox{red}{ white}{\ref{M10}}, we have  $[\big (p_{K}^{\kappa}\big)^{w}]^{-1}$ is a pseudo-differential operator. Taking into account that the composition of the pseudo-differential operators is a pseudo-differential operator, we obtain
\begin{equation}\label{M30}
 \forall n \in \mathbb{Z},\hspace{0.2cm} {[ (p_{K}^{\kappa})^{w}]}^{n}\in \Psi(p_{K}^{\kappa n},\Gamma).
 \end{equation}
Now consider the case of the exponents $\tau \in \mathbb{R}$. By pseudo-differential calculus the problem will be reduced to
$$ \forall \tau  \in I,\hspace{0.2cm} {[  (p_{K}^{\kappa})^{w}]}^{\tau }\in \Psi(p_{K}^{\kappa\tau },\Gamma),$$ where $I$ an open interval of $\mathbb{R}$. 
Note that the operator $ (p_{K}^{\kappa})^{w}$ is self-adjoint because its symbol $p_{K}^{\kappa}$ is real. Now we will show that $ (p_{K}^{\kappa})^{w}$  is a positive operator for $ K $ sufficiently large,
wich is equivalent to show that
\begin{equation}\label{M25}
\big( (p_{K}^{\kappa})^{w}u,u\big)_{L^{2}}\geq 0,\hspace{0.2cm} \forall u \in  \mathcal{S}(\mathbb{R}^{n}).
\end{equation}
Using again the formula   (\fcolorbox{red}{ white}{\ref{M3}}), we can write
  \begin{align}\label{M11}
 (p_{K}^{\kappa/2})^{w}(p_{K}^{\kappa/2})^{w}= (p_{K}^{\kappa})^{w}- R_{K}^{w},
\end{align}
  where 
  $$ R_{K}=-\int_{0}^{1} (\partial_{\eta}p_{K}^{\kappa/2})\sharp_{\theta}\big( \partial_{v}(p_{K}^{\kappa/2})\big)\, \mathrm{d}\theta + \int_{0}^{1} (\partial_{v}p_{K}^{\kappa/2})\sharp_{\theta}\big( \partial_{\eta}(p_{K}^{\kappa/2})\big)\, \mathrm{d}\theta$$
  with $g\sharp_{\theta} h$ defined in (\fcolorbox{red}{ white}{\ref{M12}}). Based on the proof of (I), we have
  $$\partial_{\eta}p_{K}^{\kappa/2} \in S(K^{-1/2}p_{K}^{\kappa/2} ,\Gamma)$$ uniformly in $K$. Moreover we have $\partial_{v}p_{K}^{\kappa/2} \in S(p_{K}^{\kappa/2},\Gamma)$, so we get 
  $$ (\partial_{\eta}p_{K}^{\kappa/2})\sharp_{\theta}\big( \partial_{v}(p_{K}^{\kappa/2})\big), (\partial_{v}p_{K}^{\kappa/2})\sharp_{\theta}\big( \partial_{\eta}(p_{K}^{\kappa/2})\big)\hspace{0.2cm} \text{and}\hspace{0.1cm} R_{K} \in S(K^{-1/2}p_{K}^{\kappa} ,\Gamma)$$
  uniformly in $K$. Writing $R_{K}^{w}$ in the following form
  $$ R_{K}^{w} =
   {K}^{-1/2} (p_{K}^{\kappa/2} )^{w}\underbrace{{K}^{1/2} [(p_{K}^{\kappa/2} )^{w}]^{-1}R_{K}^{w}[ (p_{K}^{\kappa/2} )^{w}]^{-1}}_{\in\mathcal{B}(L^{2})\hspace{0.1cm} \text{uniformly in  } K}  (p_{K}^{\kappa/2} )^{w},$$
   we obtain 
   $$\vert ( R_{K}^{w}u,u)_{L^{2}}\vert \leq C_{0} {K}^{-1/2}\Vert (p_{K}^{\kappa/2})^{w}u \Vert_{L^{2}}^{2},$$
   with $ C_{0} $ some constant independent of $ K $. Let $K_{2}=16C_{0}^{2}$, then using the relation  (\fcolorbox{red}{ white}{\ref{M11}}) we get for all $K\geq K_{2}$
   $$( (p_{K}^{\kappa})^{w}u,u)_{L^{2}}\geq \frac{3}{4}\Vert (p_{K}^{\kappa/2})^{w} u\Vert_{L^{2}}^{2}\geq 0.$$
  Then, using the following formula (see for example \cite{Yosida1968}) we can write 
   \begin{equation}\label{M900}
{[ (p_{K}^{\kappa})^{w}]}^{\tau}=-\frac{\text{sin}(\pi \tau)}{\pi}\int_{0}^{\infty}s^{\tau}(s+(p_{K}^{\kappa})^{w})^{-1}\, \mathrm{d}s,\hspace{0.2cm} \tau \in ({-1,0}).
   \end{equation}
 First, $s+ (p_{K}^{\kappa})^{w}$ is a pseudo-differential operator and its  symbol $a_{K,s}$  verifies 
    $$\forall s \in ({0,1}) ,\hspace{0.2cm} a_{K,s}\in S(p_{K}^{\kappa} ,\Gamma)$$
    and $$ \forall s\geq 1,\hspace{0.2cm} a_{K,s}\in S(s\hspace{0.025cm}(p_{K}^{\kappa}) ,\Gamma)
    \text{ uniformly  in}\hspace{0.1cm} K\hspace{0.1cm} \text{and}\hspace{0.1cm} s.$$
     So using the Theorem \fcolorbox{red}{ white}{\ref{M10}}, $(s+ (p_{K}^{\kappa})^{w})^{-1}$   is a pseudo-differential operator and its  symbol $b_{K,s}$  verifies $$\forall s \in ({0,1}) ,\hspace{0.2cm} b_{K,s}\in S((p_{K}^{\kappa})^{-1} ,\Gamma)$$
   and $$ \forall s\geq 1,\hspace{0.2cm} b_{K,s}\in S(s^{-1} (p_{K}^{\kappa})^{-1} ,\Gamma)
    \text{ uniformly  in}\hspace{0.1cm} K\hspace{0.1cm} \text{and}\hspace{0.1cm} s.$$
   Then, taking  $u \in  \mathcal{S}(\mathbb{R}^{n})$ we have
   $${[ (p_{K}^{\kappa})^{w}]}^{\tau}u=-\frac{\text{sin}(\pi \tau)}{\pi}\int_{0}^{\infty}s^{\tau}(s+ (p_{K}^{\kappa})^{w})^{-1}u\, \mathrm{d}s,$$
   using the  formula (\fcolorbox{red}{ white}{\ref{M27}}), we get
   \begin{align*}
 {[ (p_{K}^{\kappa})^{w}]}^{\tau}u &=-\frac{\text{sin}(\pi \tau)}{\pi}\int_{0}^{\infty}s^{\tau} \left(\dfrac{1}{(2\pi)^{n}}\int_{\mathbb{R}^{2n}}e^{i(v-v').\eta}b_{K,s}(\dfrac{v+v'}{2},\eta)u(v')\hspace{0.1cm}dv'd\eta \right) \, \mathrm{d}s,\\
  &=\dfrac{1}{(2\pi)^{n}}\int_{\mathbb{R}^{2n}}e^{i(v-v').\eta}\left(-\frac{\text{sin}(\pi \tau)}{\pi}\int_{0}^{\infty}s^{\tau}  b_{K,s}(\dfrac{v+v'}{2},\eta)\, \mathrm{d}s\right) u(v')\hspace{0.1cm}dv'd\eta. 
   \end{align*}
  So ${[ (p_{K}^{\kappa})^{w}]}^{\tau}$ is a pseudo-differential operator and its  symbol $d_{K}$
 given by 
   $$ d_{K}=-\frac{\text{sin}(\pi \tau)}{\pi}\int_{0}^{\infty}s^{\tau}  b_{K,s}\, \mathrm{d}s.$$
   Then we will show that $d_{K}\in S(p_{K}^{\kappa\tau },\Gamma)$. Indeed, let $ \alpha,\beta \in \mathbb{N}^{n}, (v,\eta)\in \mathbb{R}^{2n},$
   \begin{align*}
   \vert \partial_{v}^{\alpha}\partial_{\eta}^{\beta}d_{K}(v,\eta)\vert&\leq C \left( \int_{0}^{1}s^{\tau}\vert \partial_{v}^{\alpha}\partial_{\eta}^{\beta}b_{K,s}\vert\, \mathrm{d}s + \int_{1}^{\infty}s^{\tau}\vert \partial_{v}^{\alpha}\partial_{\eta}^{\beta}b_{K,s}\vert\, \mathrm{d}s\right),\\
   &\leq C_{\alpha,\beta} \left( \int_{0}^{1}s^{\tau}(p_{K}^{\kappa})^{-1}\, \mathrm{d}s + \int_{1}^{\infty}s^{\tau-1}(p_{K}^{\kappa})^{-1}\, \mathrm{d}s\right),\\
   & \leq \Tilde{C}_{\alpha,\beta}\hspace{0.1cm} {(p_{K}^{\kappa})}^{-1},
   \end{align*}
which gives that $d_{K}\in S(p_{K}^{\kappa\tau },\Gamma)$. Take $I=({-1,0})$, then we get \begin{equation}\label{M31}
   \forall \tau \in I,\hspace{0.2cm}  {[ (p_{K}^{\kappa})^{w}]}^{\tau}\in \Psi(p_{K}^{\kappa\tau },\Gamma). \end{equation}
  In addition,  ${[ (p_{K}^{\kappa})^{w}]}^{\tau}$  is a bounded operator in $ L^{2} $ uniformly in $ K $. Indeed,\\ $(s+{ (p_{K}^{\kappa})^{w}})^{-1} \in \Psi(1,\Gamma)$ uniformly in $ K $, then there exists a constant $ C_{1}> 0 $ such that
   \begin{equation}\label{M34}
      \forall s \in ({0,1}) ,\hspace{0.2cm} \Vert (s+{
      (p_{K}^{\kappa})^{w}
      })^{-1} \Vert_{\mathcal{B}(L^{2})}\leq C_{1}. 
   \end{equation}
 On the other hand we have   $$\forall u \in  \mathcal{S}(\mathbb{R}^{n}),\hspace{0.2cm} ((s+(p_{K}^{\kappa})^{w})u,u)_{L^{2}} \geq s\Vert u \Vert_{L^{2}}^{2},$$
  which gives 
   \begin{equation}\label{M33}
       \forall s \geq 1 ,\hspace{0.2cm} \Vert (s+{
       (p_{K}^{\kappa})^{w}
       })^{-1} \Vert_{\mathcal{B}(L^{2})}\leq \frac{1}{s}.
   \end{equation}
   Using  (\fcolorbox{red}{ white}{\ref{M34}}),  (\fcolorbox{red}{ white}{\ref{M33}}) and the formula  (\fcolorbox{red}{ white}{\ref{M900}}), we obtain
   \begin{equation}
      \Vert {[ (p_{K}^{\kappa})^{w}]}^{\tau}\Vert_{\mathcal{B}(L^{2})}\leq C_{2},
   \end{equation}
  with $ C_{2} $ some constant independent  of $K$. \\
   Then based on (\fcolorbox{red}{ white}{\ref{M30}}), (\fcolorbox{red}{ white}{\ref{M31}}) we obtain for all  $\tau \in I$,
   \begin{equation}\label{M32}
 \forall n \in \mathbb{Z},\hspace{0.2cm} \left[ {[ (p_{K}^{\kappa})^{w}]}^{\tau}\right]^{n}\in \Psi(p_{K}^{\kappa n\tau},\Gamma),
 \end{equation}
 so we have 
 \begin{equation}
 \forall \tau  \in \mathbb{R},\hspace{0.2cm} {[ (p_{K}^{\kappa})^{w}]}^{\tau}\in \Psi(p_{K}^{\kappa\tau},\Gamma).
 \end{equation}
  \item[IV.] We will show that
\begin{equation}\label{N122}
\forall u \in  \mathcal{S}(\mathbb{R}^{n}),\hspace{0.2cm} \Vert{[ (p_{K}^{\kappa})^{w}]}^{\tau}u \Vert_{L^{2}}^{2}\sim \Vert (p_{K}^{\kappa\tau})^{w}u \Vert_{L^{2}}^{2},
\end{equation}
otherwise we will show that there are two constants   $c,C>0$ independent of $ K $ such that for all $u \in  \mathcal{S}(\mathbb{R}^{n})$, we have: 
\begin{equation}\label{N120}
    \Vert (p_{K}^{\kappa\tau})^{w}u \Vert_{L^{2}}^{2}\leq c\hspace{0.1cm} \Vert  {[ (p_{K}^{\kappa})^{w}]}^{\tau}u \Vert_{L^{2}}^{2},
   \end{equation}
and 
\begin{equation}\label{N121}
 \Vert  {[ (p_{K}^{\kappa})^{w}]}^{\tau}u \Vert_{L^{2}}^{2}\leq C\Vert (p_{K}^{\kappa\tau})^{w}u \Vert_{L^{2}}^{2}.
   \end{equation}
   We will start with inequality (\fcolorbox{red}{ white}{\ref{N120}}), using the conclusion  (III), we have that  $ {[(p_{K}^{\kappa})^{w}]}^{\tau}$ is a pseudo-differential operator belongs to  $\Psi(p_{K}^{\kappa\tau },\Gamma)$ uniformly in $K$. Taking\\ $v= {[ (p_{K}^{\kappa})^{w}]}^{-\tau}u$,  show inequality (\fcolorbox{red}{ white}{\ref{N120}}) is equivalent to showing that
    \begin{equation}\label{M22}
\Vert (p_{K}^{\kappa\tau})^{w}{[ (p_{K}^{\kappa})^{w}]}^{-\tau}v \Vert_{L^{2}}^{2} \leq C_{3} \Vert v \Vert_{L^{2}}^{2}\quad \forall v \in L^{2}.
\end{equation}
    Using the Theorem \fcolorbox{red}{ white}{\ref{M2}}, we have $ (p_{K}^{\kappa\tau})^{w}{[ (p_{K}^{\kappa})^{w}]}^{-\tau} \in \Psi(1,\Gamma)\hookrightarrow \mathcal{B}(L^{2})$ uniformly in $K$, hence the inequality (\fcolorbox{red}{ white}{\ref{N120}}).\\
   Similarly, using the fact that  ${[ (p_{K}^{\kappa})^{w}]}^{\tau}[{(p_{K}^{\kappa\tau})^{w}]}^{-1} \in \Psi(1,\Gamma)\hookrightarrow \mathcal{B}(L^{2})$ uniformly in $K$, then taking  $v=[(p_{K}^{\kappa\tau})^{w}]^{-1}u$, we get the inequality (\fcolorbox{red}{ white}{\ref{N121}}), so the estimate (\fcolorbox{red}{ white}{\ref{N122}})  is true. 
   \item[V.] We will first show that  \begin{equation}\label{M15}
\forall u \in S(\mathbb{R}^{n}),\hspace{0.2cm} \Vert (p_{K}^{\tau})^{w}u \Vert_{L^{2}}^{2}\sim \Vert (q_{K}^{\tau})^{w}u \Vert_{L^{2}}^{2},
\end{equation}
otherwise we will show that there are two constants   $C_{1},C_{2}>0$ independent of $K$ such that for all $u \in  \mathcal{S}(\mathbb{R}^{n})$ 
\begin{equation}\label{M13}
    \Vert (p_{K}^{\tau})^{w}u \Vert_{L^{2}}^{2}\leq C_{1}\Vert (q_{K}^{\tau})^{w}u \Vert_{L^{2}}^{2},
   \end{equation}
and
\begin{equation}\label{M14}
 \Vert ( q_{K}^{\tau})^{w}u \Vert_{L^{2}}^{2}\leq C_{2}\Vert (p_{K}^{\tau})^{w}u \Vert_{L^{2}}^{2}.
   \end{equation}
 We will start with inequality (\fcolorbox{red}{ white}{\ref{M13}}), using hypothesis (iii) we have that\\  $(q_{K}^{\tau})^{w} \in \Psi(p_{K}^{\tau },\Gamma)$ uniformly in $ K $. Using the conclusion (II), we have $(q_{K}^{\tau})^{w}$ is an invertible operator, moreover its inverse $[ (q_{K}^{\tau})^{w}]^{-1}$ belongs to $\Psi(p_{K}^{-\tau },\Gamma)$ uniformly in $K$.\\ Taking $v=[(q_{K}^{\tau})^{w}]^{-1}u$,  show inequality (\fcolorbox{red}{ white}{\ref{M13}}) is equivalent to showing that
   \begin{equation}\label{M18}
\Vert (p_{K}^{\tau})^{w}[(q_{K}^{\tau})^{w}]^{-1}v \Vert_{L^{2}}^{2} \leq C_{1} \Vert v \Vert_{L^{2}}^{2}\quad \forall v \in L^{2}.
\end{equation}
Using the Theorem \fcolorbox{red}{ white}{\ref{M2}}, we have  $ (p_{K}^{\tau})^{w}{[(q_{K}^{\tau})^{w}]}^{-1} \in \Psi(1,\Gamma)\hookrightarrow \mathcal{B}(L^{2})$ uniformly in $K$, hence the inequality (\fcolorbox{red}{ white}{\ref{M18}}).\\
Similarly, using the fact that $(q_{K}^{\tau})^{w}[(p_{K}^{\tau})^{w}]^{-1} \in \Psi(1,\Gamma)\hookrightarrow \mathcal{B}(L^{2})$ uniformly in $K$,  then taking $v=[(p_{K}^{\tau})^{w}]^{-1}u$, we get the inequality (\fcolorbox{red}{ white}{\ref{M14}}), so the estimate (\fcolorbox{red}{ white}{\ref{M15}}) is true.\\
In the same way we have
\begin{equation}\label{M23}
\forall u \in  \mathcal{S}(\mathbb{R}^{n}),\hspace{0.2cm} \Vert (q_{K}^{\tau})^{w}u \Vert_{L^{2}}^{2}\sim \Vert(q_{K}^{w})^{\tau}u   \Vert_{L^{2}}^{2}.
\end{equation}
Finally, using conclusion (IV) with $\kappa=1$, we get that
\begin{equation}\label{M19}
\forall u \in  \mathcal{S}(\mathbb{R}^{n}),\hspace{0.2cm} \Vert  (p_{K}^{w})^{\tau}u \Vert_{L^{2}}^{2}\sim \Vert (p_{K}^{\tau})^{w}u \Vert_{L^{2}}^{2}.
\end{equation}
 \item[VI.] We will show that
 \begin{equation}\label{M24}
     \forall u \in  \mathcal{S}(\mathbb{R}^{n}),\hspace{0.2cm} (p_{K}^{w}u,u)_{L^{2}}\sim (q_{K}^{w}u,u)_{L^{2}}.
 \end{equation}
 Indeed, let $ u \in  \mathcal{S}(\mathbb{R}^{n})$ we have 
 $$ (p_{K}^{w}u,u)_{L^{2}}=((p_{K}^{w})^{1/2}u,(p_{K}^{w})^{1/2}u)_{L^{2}}=\Vert  (p_{K}^{w})^{1/2}u \Vert_{L^{2}}^{2},$$
using the result (IV) with $\tau=\frac{1}{2}$,  we obtain that
 $$ \Vert   (p_{K}^{w})^{1/2}u \Vert_{L^{2}}^{2}\sim \Vert (p_{K}^{1/2})^{w}u \Vert_{L^{2}}^{2}\sim \Vert (q_{K}^{1/2})^{w}u \Vert_{L^{2}}^{2} \sim \Vert  (q_{K}^{w})^{1/2}u \Vert_{L^{2}}^{2} \sim (q_{K}^{w}u,u)_{L^{2}}.$$
 So the estimate (\fcolorbox{red}{ white}{\ref{M24}}) is true.
 \item[VII.] We will show that
 $$\forall u \in  \mathcal{S}(\mathbb{R}^{n}),\hspace{0.2cm}
 (p_{K}^{w}u,u)_{L^{2}}\geq \Vert u \Vert_{L^{2}}^{2}.$$ 
 Indeed, let $ u \in  \mathcal{S}(\mathbb{R}^{n})$
 $$ (p_{K}^{w}u,u)_{L^{2}}=(p_{K-1}^{w}u,u)_{L^{2}}+({\langle v \rangle}^{M}u,u)_{L^{2}},$$
based on the proof of  (\fcolorbox{red}{ white}{\ref{M25}}), we can show that there  exists a positive constant $K_{3}$ such that for all  $K\geq K_{3}$ we have $$(p_{K-1}^{w}u,u)_{L^{2}}\geq 0,$$ Which gives
 $$(p_{K}^{w}u,u)_{L^{2}}= \underbrace{(p_{K-1}^{w}u,u)_{L^{2}}}_{\geq 0}+({\langle v \rangle}^{M}u,u)_{L^{2}}\geq \Vert {\langle v \rangle}^{M/2}u \Vert_{L^{2}}^{2}.$$
With the condition  $M\geq 0$ we finally obtain 
 $$(p_{K}^{w}u,u)_{L^{2}}\geq \Vert u \Vert_{L^{2}}^{2}.$$
Finally, for a symbol $p$ which verifies the hypotheses of Theorem \fcolorbox{red}{ white}{\ref{M36}}, we can fix \begin{align}\label{M71}
    K\geq\underbrace{\text{max}(K_{1},K_{2},K_{3})}_{K_{0}}
\end{align} and we apply the results of Theorem \fcolorbox{red}{ white}{\ref{M36}} to the operator $p_{K}^{w}.$ 
\end{proof}
\subsection{Wick quantization}\label{P2}

Finally let us recall some basic properties of the Wick quantization, which is also called
anti-Wick in \cite{shubin1987pseudodifferential}. The importance in studying the Wick quantization lies in the facts
that positive symbols give rise to positive operators.\\
Let $Y=(v,\eta)$  be a point in $\mathbb{R}^{6}$. The Wick quantization of a symbol $q$ is given by $$q^{\text{Wick}}= (2\pi)^{-3}\int_{\mathbb{R}^{6}}q(Y)\Pi_{Y}\hspace{0.1cm}\mathrm{d}Y, $$ 
where $\Pi_{Y}$ is the projector associated to the Gaussian $\varphi_{Y}$ which is defined by
$$\varphi_{Y}(z)=\pi^{-3/4}e^{-\frac{1}{2}{\vert z-v \vert^{2}}}e^{iz\cdot\eta/2},\quad \forall z \in \mathbb{R}^{3}.$$
The main property of the Wick quantization is its positivity, i.e.,
\begin{align}\label{P7}
q(v,\eta)\geq0,\hspace{0.2cm} \forall(v,\eta)\in \mathbb{R}^{6}\hspace{0.2cm}\text{implies}\hspace{0.2cm} q^{\text{Wick}}\geq0.
\end{align}
According to Theorem 24.1 in \cite{shubin1987pseudodifferential}, the Wick and Weyl quantizations of a symbol $q$  are linked by the following identities
\begin{align}\label{P3}
    q^{\text{Wick}}=\big(q*\pi^{-3}e^{-\vert \cdot\vert^{2}}\big)^{w}=q^{w}+r^{w}
\end{align}
with  \begin{align}\label{P4}
r(Y)=\pi^{-3}\int_{0}^{1}\int_{\mathbb{R}^{6}} (1-\theta)q"(Y+\theta Y_{1})\vert Y_{1}\vert^{2} e^{-\vert Y_{1}\vert^{2}}\mathrm{d}Y_{1} \, \mathrm{d}\theta,
\end{align}
 where $q"(Y)$ is the Hessian of $q$ at the point $Y$.
Therefore, according to (\fcolorbox{red}{ white}{\ref{M4}}), if  $q\in S(1,\Gamma)$ then $ q^{\text{Wick}}$ is a bounded operator in $L^{2}$ .\\
We also recall the following composition formula obtained in the proof of Proposition 3.4 in \cite{Lerner}
\begin{align}\label{P6}
    q_{1}^{\text{Wick}}  q_{2}^{\text{Wick}}=[q_{1}q_{2}-q'_{1}\cdot q'_{2}+\frac{1}{i}\left\{q_{1},q_{2}\right\}]^{\text{Wick}} +T,
\end{align}
with $ T $ a bounded operator in $L^{2}(\mathbb{R}^{6})$ and $q'$ is the gradient of $q$ with respect to $Y$, when $ q_{1}\in L^{\infty}(\mathbb{R}^{6}) $
 and $q_{2}$ is a smooth symbol

whose derivatives of order $\geq 2$ are bounded on $\mathbb{R}^{6}$.
The notation $\left\{q_{1},q_{2}\right\}$ denotes by the Poisson bracket defined by
\begin{align}
  \left\{q_{1},q_{2}\right\}=  \frac{\partial{q_{1}}}{\partial{\eta}}\cdot\frac{\partial{q_{2}}}{\partial{v}}-\frac{\partial{q_{1}}}{\partial{v}}\cdot\frac{\partial{q_{2}}}{\partial{\eta}}.
\end{align}
\end{appendix}
\addcontentsline{toc}{section}{Acknowledgement}
  \textbf{Acknowledgement.} I would like to thank my supervisors Frédéric Hérau and\\ Isabelle Tristani  for their encouragements, helpful discussions and support during the maturation of this paper.
\bibliographystyle{acm}
\bibliography{Nadine}

\end{document}